\documentclass[11pt]{article}

\usepackage[show]{ed} 

\usepackage{amsmath}
\usepackage{amsthm}		
\usepackage{amssymb}
\usepackage{enumerate}
\usepackage{cases}
\usepackage{multirow}
\usepackage{subfigure}
\usepackage{graphicx}
\usepackage{appendix}
\usepackage{pgf,tikz}
\usepackage{mathrsfs} 
\usetikzlibrary{arrows}

\graphicspath{{./}}

\begin{document}
\newcommand{\done}[2]{\dfrac{d {#1}}{d {#2}}}
\newcommand{\donet}[2]{\frac{d {#1}}{d {#2}}}
\newcommand{\pdone}[2]{\dfrac{\partial {#1}}{\partial {#2}}}
\newcommand{\pdonet}[2]{\frac{\partial {#1}}{\partial {#2}}}
\newcommand{\pdonetext}[2]{\partial {#1}/\partial {#2}}
\newcommand{\pdtwo}[2]{\dfrac{\partial^2 {#1}}{\partial {#2}^2}}
\newcommand{\pdtwot}[2]{\frac{\partial^2 {#1}}{\partial {#2}^2}}
\newcommand{\pdtwomix}[3]{\dfrac{\partial^2 {#1}}{\partial {#2}\partial {#3}}}
\newcommand{\pdtwomixt}[3]{\frac{\partial^2 {#1}}{\partial {#2}\partial {#3}}}
\newcommand{\bs}[1]{\mathbf{#1}}
\newcommand{\bx}{\mathbf{x}}
\newcommand{\by}{\mathbf{y}}
\newcommand{\bd}{\mathbf{d}} 
\newcommand{\bn}{\mathbf{n}} 
\newcommand{\bP}{\mathbf{P}} 
\newcommand{\bp}{\mathbf{p}} 
\newcommand{\bz}{\mathbf{z}}
\newcommand{\ba}{\mathbf{a}}
\newcommand{\ol}[1]{\overline{#1}}
\newcommand{\rf}[1]{(\ref{#1})}
\newcommand{\xt}{\mathbf{x},t}
\newcommand{\hs}[1]{\hspace{#1mm}}
\newcommand{\vs}[1]{\vspace{#1mm}}
\newcommand{\eps}{\varepsilon}
\newcommand{\ord}[1]{\mathcal{O}\left(#1\right)} 
\newcommand{\oord}[1]{o\left(#1\right)}
\newcommand{\Ord}[1]{\Theta\left(#1\right)}
\newcommand{\PhiF}{\Phi_{\rm freq}}
\newcommand{\real}[1]{{\rm Re}\left[#1\right]} 
\newcommand{\im}[1]{{\rm Im}\left[#1\right]}
\newcommand{\hsnorm}[1]{||#1||_{H^{s}(\bs{R})}}
\newcommand{\hnorm}[1]{||#1||_{\tilde{H}^{-1/2}((0,1))}}
\newcommand{\norm}[2]{\left\|#1\right\|_{#2}}
\newcommand{\normt}[2]{\|#1\|_{#2}}
\newcommand{\on}[1]{\Vert{#1} \Vert_{1}}
\newcommand{\tn}[1]{\Vert{#1} \Vert_{2}}
\newcommand{\ts}{\tilde{s}}
\newcommand{\tGamma}{{\tilde{\Gamma}}}
\newcommand{\darg}[1]{\left|{\rm arg}\left[ #1 \right]\right|}
\newcommand{\bnabla}{\boldsymbol{\nabla}}
\newcommand{\dive}{\boldsymbol{\nabla}\cdot}
\newcommand{\curl}{\boldsymbol{\nabla}\times}
\newcommand{\Phixy}{\Phi(\bx,\by)}
\newcommand{\PhiOxy}{\Phi_0(\bx,\by)}
\newcommand{\dxPhixy}{\pdone{\Phi}{n(\bx)}(\bx,\by)}
\newcommand{\dyPhixy}{\pdone{\Phi}{n(\by)}(\bx,\by)}
\newcommand{\dxPhiOxy}{\pdone{\Phi_0}{n(\bx)}(\bx,\by)}
\newcommand{\dyPhiOxy}{\pdone{\Phi_0}{n(\by)}(\bx,\by)}

\newcommand{\rd}{\mathrm{d}}
\newcommand{\R}{\mathbb{R}}
\newcommand{\N}{\mathbb{N}}
\newcommand{\Z}{\mathbb{Z}}
\newcommand{\C}{\mathbb{C}}
\newcommand{\K}{{\mathbb{K}}}
\newcommand{\ri}{{\mathrm{i}}}
\newcommand{\re}{{\mathrm{e}}} 
\newcommand{\scA}{\mathscr{A}} 

\newcommand{\cA}{\mathcal{A}}
\newcommand{\cC}{\mathcal{C}}
\newcommand{\cS}{\mathcal{S}}
\newcommand{\cD}{\mathcal{D}}
\newcommand{\cone}{{c_{j}^\pm}}
\newcommand{\ctwo}{{c_{2,j}^\pm}}
\newcommand{\cthree}{{c_{3,j}^\pm}}

\newtheorem{thm}{Theorem}[section]
\newtheorem{lem}[thm]{Lemma}
\newtheorem{defn}[thm]{Definition}
\newtheorem{prop}[thm]{Proposition}
\newtheorem{cor}[thm]{Corollary}
\newtheorem{rem}[thm]{Remark}
\newtheorem{conj}[thm]{Conjecture}
\newtheorem{ass}[thm]{Assumption}
\newcommand{\tk}{\tilde{k}}
\newcommand{\tbx}{\tilde{\bx}}
\newcommand{\txi}{\tilde{\xi}}
\newcommand{\bxi}{\boldsymbol{\xi}}
\newcommand{\sD}{\mathsf{D}}
\newcommand{\sN}{\mathsf{N}}
\newcommand{\cT}{\mathcal{T}}
\newcommand{\dudnjump}{\left[ \pdone{u}{\bn}\right]}
\newcommand{\dudnjumptext}{[ \pdonetext{u}{\bn}]}
\newcommand{\cV}{\mathcal{V}}
\title{Acoustic scattering: high frequency boundary element methods and unified transform methods}

\author{S.\ N.\ Chandler-Wilde\footnotemark[1], S.\ Langdon\footnotemark[1]}

\maketitle

\renewcommand{\thefootnote}{\fnsymbol{footnote}}
\footnotetext[1]{Department of Mathematics and Statistics, University of Reading, Whiteknights PO Box 220, Reading RG6 6AX, UK, {\tt email: s.n.chandler-wilde@reading.ac.uk, s.langdon@reading.ac.uk}.  SL supported by EPSRC grants EP/K000012/1 and EP/K037862/1}
\renewcommand{\thefootnote}{\arabic{footnote}}

\begin{abstract}
  We describe some recent advances in the numerical solution of acoustic scattering problems. A major focus of the paper is the efficient solution of high frequency scattering problems via hybrid numerical-asymptotic boundary element methods. We also make connections to the unified transform method due to A.~S.~Fokas and co-authors, analysing particular instances of this method,  proposed by J.~A.~DeSanto and co-authors, for problems of acoustic scattering by diffraction gratings.
\end{abstract}

\section{Introduction}
\label{sec:intro}

The reliable simulation of processes in which acoustic waves are scattered by obstacles is of great practical interest, with applications including the modelling of sonar and other methods of acoustic detection, and the study of problems of outdoor noise propagation and noise control, for example associated with road, rail or aircraft noise.  Unless the geometry of the scattering obstacle is particularly simple, analytical solution of scattering problems is usually impossible, and hence in general numerical schemes are required.

Most problems of acoustic scattering can be formulated, in the frequency domain, as linear elliptic boundary value problems (BVPs).  In this chapter we describe how the boundary element method (BEM) can be used to solve such problems, and we summarise in particular recent progress in tackling high frequency scattering problems by combining classical BEMs with insights from ray tracing methods and high frequency asymptotics. We also make connections with a newer method for numerical solution of elliptic boundary value problems, the unified tranform method due to Fokas and co-authors, and detail the independent development of this method for acoustic scattering problems by DeSanto and co-authors.

In a homogeneous medium at rest, the acoustic pressure satisfies the linear wave equation.  Under the further assumption of harmonic ($\re^{-\ri\omega t}$) time dependence, the problem of computing perturbations in acoustic pressure reduces to the solution of the Helmholtz equation
\begin{equation}
  \Delta u + k^2 u =0, \quad \mbox{in } D\subset \R^d,
  \label{eqn:HE}
\end{equation}
where $d=2$ or $3$ is the dimension of the problem under consideration, and $D$ denotes the domain of propagation, the region in which the wave propagates. 
The positive constant $k:=\omega/c>0$, where $\omega$ is angular frequency and $c$ is the speed of sound in $D$, is known as the wavenumber.

We will describe below how to compute, by the boundary element method, solutions to~(\ref{eqn:HE}) that also satisfy certain boundary conditions on the boundary of the domain of propagation.  The most commonly relevant boundary condition is the impedance boundary condition 
\begin{equation}
  \frac{\partial u}{\partial \nu} - \ri k \beta u = h, \quad \mbox{on }\Gamma,
  \label{eqn:imp_bc}
\end{equation}
where $\Gamma$ denotes the boundary of $D$ and $\partial/\partial\nu$ denotes the normal derivative on the boundary, where $\nu(\bx)$ denotes the unit normal at $\bx\in\Gamma$ directed out of $D$.  The function $h$ is identically zero in acoustic scattering problems (problems where we are given an incident wave and a stationary scatterer and have to compute the resulting acoustic field), but is non-zero for radiation problems (where the motion of a radiating structure is given, and we have to compute the radiated acoustic field).  Finally, $\beta$ is the relative surface admittance which, in general, is a function of position on the boundary and of frequency $\omega$.  When $\beta=0$ the boundary is described as being acoustically rigid, or sound-hard (this is the Neumann boundary condition).  More generally, $\beta\in\C$, with $\real{\beta}\geq 0$ to satisfy the physical condition that the surface absorbs (rather than emits) energy.  When $|\beta|$ is very large, one can consider as an approximation to~(\ref{eqn:imp_bc}) the sound-soft (Dirichlet) boundary condition
\begin{equation}
  u=h, \quad \mbox{on }\Gamma.
  \label{eqn:ss}
\end{equation}

Equation \eqref{eqn:HE} models acoustic propagation in a homogeneous medium at rest. We are often interested in applications in propagation through a medium with variable wave speed. The BEM is well-adapted to compute solutions in the case when the wave speed $c$ is piecewise constant. In particular, when a homogeneous region  with a different wave speed is  embedded in a larger homogeneous medium, acoustic waves are transmitted across the boundary $\Gamma$ between the two media, \eqref{eqn:HE} holds on either side of $\Gamma$ with different values of $k=\omega/c$, and, at least in the simplest case when the density of the two media is the same, the boundary conditions on $\Gamma$ (so-called ``transmission conditions'') are that $u$ and $\partial u/\partial \nu$ are continuous across $\Gamma$.


The domain $D$ can be a bounded domain (e.g.\ for applications in room acoustics), but in many practical applications it may be unbounded (e.g.\ for outdoor noise propagation).  In this case, the complete mathematical formulation must also include a condition to represent the idea that the acoustic field (or at least some part of it, e.g.\ the part reflected by a scattering obstacle) is travelling outwards.  The usual condition imposed is the Sommerfeld radiation condition,
\begin{equation}
  \frac{\partial u}{\partial r}(\bx) - \ri k u(\bx) = o(r^{-(d-1)/2}),
  \label{eqn:SRC}
\end{equation}
as $r:=|\bx|\to\infty$, uniformly in $\hat{\bx}:=\bx/r$.

Numerical solution of (\ref{eqn:HE}) together with a boundary condition (and a radiation condition if the problem is posed on an unbounded domain) can be achieved by many means.  The key idea of the BEM is to reformulate the BVP as a boundary integral equation (i.e.\ an integral equation that holds on the boundary $\Gamma$), and then to solve that numerically.  Under the assumptions above (homogeneous medium, time-harmonic waves) the advantages of BEM over domain based methods such as finite difference or finite element methods are twofold: firstly, as the integral equation to be solved holds only on the boundary of the domain, the dimensionality of the problem is reduced; this has the obvious advantage that it is easier to construct an approximation space in a lower dimension, but moreover we will see below (in \S\ref{sec:HNA}) that it can be easier to understand the oscillatory behaviour of the solution on the boundary than throughout the domain, which can assist greatly in designing a good approximation space at high frequencies.  Secondly, for problems in unbounded domains, reformulating as an integral equation on the boundary removes the need to truncate the unbounded domain in order to be able to construct a discretisation space.

We describe the reformulation of our BVPs as boundary integral equations (BIEs) in
\S\ref{sec:BIE}.  Following~\cite{ChGrLaSp:12} we consider the case that $D$ is Lipschitz, and thus work throughout in a Sobolev space setting (see, e.g., \cite{McLean}; for the simpler case of smooth boundaries we refer to~\cite{CoKr:83}).  We then consider the numerical solution of these BIEs in~\S\ref{sec:BEM}, focusing in particular, in~\S\ref{sec:HNA}, on schemes that are well-adapted to the case when the wavenumber~$k$ is large. As we report, for many scattering problems these methods provably compute solutions of any desired accuracy with a cost, in terms of numbers of degrees of freedom and size of matrix to be inverted,  that is close to frequency independent.

There is a wide literature on boundary integral equation formulations for acoustic scattering problems and boundary element methods for linear elliptic BVPs:  see, e.g., \cite{At:97,CoKr:83,HsWe:04,Kress,McLean,Ne:01,sauter-schwab11,Sl:92,Sl:95}, and see also, e.g., \cite{Ih:98} for a comparison with finite element methods.  The question of how to develop schemes efficient for large~$k$ has been considered fully in the review article~\cite{ChGrLaSp:12}, where a very complete literature review of the topic can be found.  Our presentation in this chapter summarises some of the key ideas found therein, but primarily focuses, in \S\ref{sec:HNA}, on more recent developments.

This chapter is one of a collection of articles in significant part focussed on  the so-called ``unified transform'' or
``Fokas transform'', introduced by Fokas in 1997~\cite{Fo97}. This method is on the one hand an analytical transform technique which can be employed to solve linear BVPs with constant coefficents in canonical domains. But also this method, when applied in general domains, has many aspects in common with BIE and boundary element methods, in particular it reduces solution of the BVP to solution of an equation on the boundary $\Gamma$ for the unknown part of the Cauchy data ($u$ or its normal derivative on $\Gamma$), just as in the BIE method.

We discuss this methodology in \S\ref{sec:uni}. In particular, we describe how this method can be used to solve acoustic problems in interior domains (see also Spence \cite{Spence14} in this volume), proposing a new version of this method which computes the best approximation to the unknown boundary data from a space of restrictions of (generalised) plane waves to the boundary. We also discuss the application of this method to acoustic scattering problems, noting that the methodology applies in particular to so-called {\em rough surface scattering problems}, indeed has been developed independently as a numerical method for these problems by DeSanto and co-authors, in articles from 1981 onwards \cite{DeSanto:81,DeSaErHeMi:98,DeSaErHeMi:01,DeSaErHeKrMiSw:01,ArChDeSa:06}. The numerical implementations  we focus on have in common with the high frequency BEMs of \S\ref{sec:HNA} that they utilise oscillatory basis functions (restrictions of plane waves to $\Gamma$). The results reported in \S\ref{sec:undg} from \cite{ArChDeSa:06} suggest that, in terms of numbers of degrees of freedom required to achieve accurate approximations,  these methods outperform standard BEMs for certain scattering problems.

\tableofcontents

\section{Boundary integral equation formulations}
\label{sec:BIE}

In this section we state the BVPs introduced in \S\ref{sec:intro} more precisely, and we reformulate them as BIEs.  Given that many scatterers in applications have corners and edges, we will consider throughout domains $D\subset \R^d$ (with $d=2$ or 3) that are Lipschitz, usually denoting the boundary of our domain $D$ by $\Gamma$.  We denote the trace operator by $\gamma$, so that $\gamma u$ is just the restriction of $u$ to $\Gamma$ when $u$ is sufficiently regular, and the normal derivative operator by $\partial_\nu$, noting that $\partial_\nu u$ coincides with the classical definition of the normal derivative $\partial u/\partial \nu$ when $u$ is sufficiently regular (see \cite[Appendix~A]{ChGrLaSp:12} for more details).  
For further mathematical details on the results in this section, we refer particularly to~\cite[\S2]{ChGrLaSp:12}. Our Sobolev space notations are defined precisely in \cite{ChGrLaSp:12,McLean}. These notations are standard and we will not repeat them here except to note that, for a general open set $D$, following \cite{McLean}, $H^s(D)$, for $s\in\R$, will denote the space of restrictions to $D$ of elements of $H^s(\R^d)$, while, for $n\in\N$, $W^n(D)$ will denote those $u\in L^2(D)$ whose partial derivatives of order $\leq n$ are also in $L^2(D)$: in particular $W^1(D) = \{u\in L^2(D): \nabla u\in L^2(D)\}$. These notions coincide, that is $H^n(D)=W^n(D)$ (with equivalence of norms), when $D$ is Lipschitz~\cite{McLean}, but we need to tread carefully where $D$ is not Lipschitz, for example when studying screen problems. 

\subsection{Acoustic BVPs}

This paper is about scattering problems, which are predominantly BVPs in exterior domains. However, the BIE methods we use intimately make connections between problems in exterior and interior domains, so that we will also make mention of interior problems. (And we flag that, to introduce the unified transform in \S\ref{sec:uni}, we study its application first to interior problems in \S\ref{sec:unidp}.) As alluded to above, we will formulate all of our BVPs in Sobolev space settings.  We will consider as  BVPS: interior and exterior Dirichlet and impedance problems for the Helmholtz equation~(\ref{eqn:HE}) in the interior and exterior of Lipschitz domains; a Dirichlet problem in the exterior of a planar screen; and a particular Helmholtz transmission problem.

We state first, for $D$ a bounded Lipschitz domain, the {\bf interior Dirichlet problem}:
\begin{equation} \label{prob:idp}
\begin{array}{l}
  \mbox{Given }h\in H^{1/2}(\Gamma), \mbox{ find }u\in C^2(D)\cap H^1(D)\\
  \mbox{such that (\ref{eqn:HE}) holds in } D \mbox{ and }\gamma u =  h \mbox{ on }\Gamma.
\end{array}
\end{equation}

Next, for $D$ a bounded Lipschitz domain, the {\bf interior impedance problem}:
\begin{equation} \label{prob:iip}
\begin{array}{l}
  \mbox{Given }h\in H^{-1/2}(\Gamma), \mbox{ and } \beta\in L^{\infty}(\Gamma) \mbox{ (with }\real{\beta}\geq 0),\\
  \mbox{ find }u\in C^2(D)\cap H^1(D) \mbox{ such that (\ref{eqn:HE}) holds in } D \\
  \mbox{ and }\partial_\nu u - \ri k \beta \gamma u =  h \mbox{ on }\Gamma.
\end{array}
\end{equation}

Solvability results for~(\ref{prob:idp}) and~(\ref{prob:iip}) are well known (see, e.g., \cite{CoKr:83}, \cite[p.286]{McLean}, \cite[Theorems~2.1, 2.3]{ChGrLaSp:12}).  For the interior Dirichlet problem~(\ref{prob:idp}) there exists a sequence $0<k_1<k_2<\cdots$ of positive wavenumbers, with $k_m\to\infty$ as $m\to\infty$, such that~(\ref{prob:idp}) with $h=0$ has a non-trivial solution (so $-k_m^2$ is a Dirichlet eigenvalue of the Laplacian in $D$).  For all other values of $k>0$, (\ref{prob:idp}) has exactly one solution.  The same result holds for the Neumann problem~((\ref{prob:iip}) with $\beta=0$), with a different sequence of wavenumbers.  The interior impedance problem~(\ref{prob:iip}) with $\real{\beta}$ not identically zero has exactly one solution.

Next we state the exterior BVPs.  Suppose that $\Omega_+\subset \R^d$ is an unbounded Lipschitz domain with boundary~$\Gamma$, such that $\Omega_-:= \R^d\setminus \overline{\Omega_+}$, $d=2$ or $3$, is a bounded Lipschitz open set.  The {\bf exterior Dirichlet problem} is:
\begin{equation} \label{prob:edp}
\begin{array}{l}
  \mbox{Given }h\in H^{1/2}(\Gamma), \mbox{ find }u\in C^2(\Omega_+)\cap H_{\mathrm{loc}}^1(\Omega_+)\\
  \mbox{such that (\ref{eqn:HE}) holds in } \Omega_+, \gamma u =  h \mbox{ on } \Gamma,\\
  \mbox{and }u \mbox{ satisfies the radiation condition (\ref{eqn:SRC}).}
\end{array}
\end{equation}
The {\bf exterior impedance problem} is (where the normal $\nu$ here points out of $\Omega_-$ {\em into} the domain of propagation $\Omega_+$):
\begin{equation} \label{prob:eip}
\begin{array}{l}
  \mbox{Given }h\in H^{-1/2}(\Gamma), \mbox{ and }\beta\in L^\infty(\Gamma) \mbox{ (with }\real{\beta}\geq 0), \\
  \mbox{find }u\in C^2(\Omega_+)\cap H_{\mathrm{loc}}^1(\Omega_+) \mbox{ such that (\ref{eqn:HE}) holds in } \Omega_+,\\
  \partial_\nu u +\ri k\beta u =  h \mbox{ on }\Gamma,\\
  \mbox{and }u\mbox{ satisfies the radiation condition (\ref{eqn:SRC}).}
\end{array}
\end{equation}
Both problems~(\ref{prob:edp}) and~(\ref{prob:eip}) have exactly one solution (see, e.g., \cite{CoKr:83}, \cite[Theorem~2.10]{ChGrLaSp:12}).

In the above exterior Dirichlet problem the boundary $\Gamma$ is a closed surface, separating two Lipschitz open sets, $\Omega_+$ and $\Omega_-$. A variant, which models acoustic scattering by infinitely thin screens, is to consider the case where $\Gamma$ is an open surface, and the domain in which the BVP is to be solved lies on both sides of $\Gamma$. We will show results in \S\ref{sec:HNA} for the special case of a planar screen, when, for some bounded $C^0$ open subset $S\subset \R^{d-1}$,
\begin{equation} \label{eqn:screen}
\Gamma := \{(\tilde x,0): \tilde x\in S\}.
\end{equation}
With $\Gamma$ given by \eqref{eqn:screen} and $D:= \R^d\setminus \overline \Gamma$, the {\bf Dirichlet screen problem} is:
\begin{equation} \label{prob:screen}
\begin{array}{l}
  \mbox{Given }h\in H^{1/2}(\Gamma), \mbox{ find }u\in C^2(D)\cap W_{\mathrm{loc}}^1(D)\\
  \mbox{such that (\ref{eqn:HE}) holds in } D, \gamma_\pm u =  h \mbox{ on } \Gamma,\\
  \mbox{and }u \mbox{ satisfies the radiation condition (\ref{eqn:SRC}).}
\end{array}
\end{equation}
In the above formulation $\gamma_\pm$ are the trace operators onto the plane $x_d=0$ containing $\Gamma$ from the upper and lower half-planes, respectively. That this problem is well-posed dates back at least to Stephan \cite{stephan87} in the case that $S$ is a $C^\infty$ open set; this result is extended to the case that $S$ is merely $C^0$ in~\cite{ScreenCoerc}.

Finally we formulate a transmission problem for the Helmholtz equation. Let $\Omega_-$ be a bounded Lipschitz open set, $\Omega_+ := \R^d\setminus \overline{\Omega_-}$, $\Gamma$ be the common boundary of $\Omega_+$ and $\Omega_-$,  and suppose that $k$ in \eqref{eqn:HE} is a function of position rather than constant: precisely that, for some constants $k_+ >0$ and $k_-\in \C$ with $\real{k_-}>0$ and $\im{k_-}\geq 0$, $k(x) = k_+$ in $\Omega_+$ and $k(x) = k_-$ in $\Omega_-$. The {\bf Helmholtz transmission problem} we consider is:
\begin{equation} \label{prob:trans}
\begin{array}{l}
  \mbox{Given }h\in H^{1/2}(\Gamma), \; g\in H^{-1/2}(\Gamma), \mbox{ find }u\in C^2(\Omega_+\cup \Omega_-)\cap\\
   W_{\mathrm{loc}}^1(\Omega_+\cup \Omega_-)
  \mbox{ such that (\ref{eqn:HE}) holds in } \Omega_+\cup \Omega_-, \gamma_+u-\gamma_- u =  h \mbox{ and }\\ \partial_\nu^+ u-\partial_\nu^- u = g \mbox{ on } \Gamma,
  \mbox{ and }u \mbox{ satisfies the radiation condition (\ref{eqn:SRC}).}
\end{array}
\end{equation}
Here $\gamma_+ u$ and $\gamma_- u$ are the traces of $u$, and $\partial_\nu^\pm u$ the normal derivatives, taken from the $\Omega_+$ and $\Omega_-$ sides, respectively. That this BVP has exactly  one solution is shown for the case in which $\Omega_-$ is a $C^2$ domain in \cite{CoKr:83}, and the argument there easily extends to the Lipschitz case. We recall from the introduction that the different values for $k$ in $\Omega_\pm$ correspond physically to different wave speeds; allowing $k_-$ to have a positive imaginary part models an interior medium which dissipates energy as the wave propagates. We remark that \cite{CoKr:83} (and many other authors) consider a more general case with $\gamma_+u-\gamma_- u =  h$ replaced by $\mu_+\gamma_+u-\mu_-\gamma_- u =  h$, for some constants $\mu_\pm\in \C$, which models a jump in density in addition to wave speed across~$\Gamma$.

\subsection{Acoustic scattering problems}
In acoustic scattering problems one is interested in computing the scattered acoustic field $u^S$ produced when an incident field $u^I\in L^1_{\mathrm{loc}}(\R^d)$ interacts with some obstacle or obstacles (the {\em scatterer}) occupying the closed set $\Omega$ (with surface $\Gamma$), such that $\Omega_+:= \R^d\setminus \Omega$ is an unbounded domain. One assumes that the incident field itself is a solution of the Helmholtz equation, precisely that it satisfies~(\ref{eqn:HE}) in some neighbourhood $G$ of $\Omega$, in which case also $u^I|_G\in C^\infty(G)$.   We will refer throughout to the sum $u:= u^I+u^S$ as the {\em total acoustic field} (total field for short).

In many applications, the incident field is generated by a point source, i.e., for some $\bz\in \R^d\setminus \Omega$,
$$
u^I(\bx) = \Phi_k(\bx,\bz), \quad \bx\in \R^d\setminus \{\bz\},
$$
where $\Phi_k$ is the fundamental solution of the Helmholtz equation, given, in
the two-dimensional (2D) and three-dimensional (3D) cases, by
\begin{equation} \label{eq:Phi}
\Phi_k(\bx,\by) := \left\{\begin{array}{cc}
                       \displaystyle{\frac{\ri}{4}}H_0^{(1)}(k|\bx-\by|), & d=2,
                      \\ \\
                       \displaystyle{\frac{\exp(\ri k|\bx-\by|)}{4\pi|\bx-\by|}}, & d=3,
                    \end{array}\right.
\end{equation}
for $\bx,\by\in\R^d$, $\bx\neq \by$, where $H_\nu^{(1)}$ denotes the Hankel
function of the first kind of order $\nu$.  To represent a source far from the scatterer, the incident field can be taken to be a plane wave, i.e., for some $\hat \ba\in \R^d$ with $|\hat \ba|=1$,
\begin{equation}
  u^I(\bx) = \exp(\ri k \bx\cdot \hat \ba), \quad \bx\in \R^d.
  \label{eqn:plane}
\end{equation}

The scattering problems we consider are then particular cases of some of the BVPs above, namely the exterior Dirichlet and impedance problems, the Dirchlet screen problem, and the transmission problem.  The {\bf sound-soft scattering problem} is:
\begin{equation} \label{prob:ssp}
\begin{array}{l}
  \mbox{Find }u^S\in C^2(\Omega_+)\cap H_{\mathrm{loc}}^1(\Omega_+) \mbox{ such that (\ref{eqn:HE}) holds in } \Omega_+, \\
  \gamma u =  0 \mbox{ on } \Gamma\, (\mbox{so }\gamma u^S = -u^I|_{\Gamma}), \\
  \mbox{and }u^S \mbox{ satisfies the radiation condition (\ref{eqn:SRC}).}
\end{array}
\end{equation}
The {\bf impedance scattering problem} is:
\begin{equation} \label{prob:isp}
\begin{array}{l}
  \mbox{Find }u^S\in C^2(\Omega_+)\cap H_{\mathrm{loc}}^1(\Omega_+) \mbox{ such that (\ref{eqn:HE}) holds in } \Omega_+,\\
  \partial_\nu u +\ri k\beta u =  0 \mbox{ on }\Gamma \\
  (\mbox{so }(\partial_{\nu} + \ri k \beta \gamma) u^S|_{\Gamma} = -(\partial_{\nu} + \ri k \beta \gamma) u^I|_{\Gamma}),\\
  \mbox{and }u^S\mbox{ satisfies the radiation condition (\ref{eqn:SRC}).}
\end{array}
\end{equation}
The {\bf screen scattering problem} we consider is (where $\Gamma$ is given by \eqref{eqn:screen} and $D=\R^d\setminus \overline \Gamma$):
\begin{equation} \label{prob:scsp}
\begin{array}{l}
  \mbox{Find }u^S\in C^2(D)\cap W_{\mathrm{loc}}^1(D) \mbox{ such that (\ref{eqn:HE}) holds in } D, \\
  \gamma_\pm u =  0 \mbox{ on } \Gamma\, (\mbox{so }\gamma_\pm u^S = -u^I|_{\Gamma}), \\
  \mbox{and }u^S \mbox{ satisfies the radiation condition (\ref{eqn:SRC}).}
\end{array}
\end{equation}
Finally the {\bf transmission scattering problem} we consider is (where $k(x) = k_+$ in $\Omega_+$ and $k(x) = k_-$ in $\Omega_-$):
\begin{equation} \label{prob:tsp}
\begin{array}{l}
  \mbox{Find }u\in C^2(\Omega_+\cup \Omega_-)\cap W_{\mathrm{loc}}^1(\Omega_+\cup \Omega_-)
  \mbox{ such that (\ref{eqn:HE})}\\ \mbox{holds in } \Omega_+\cup \Omega_-, \gamma_+u=\gamma_- u  \mbox{ and } \partial_\nu^+ u=\partial_\nu^- u \mbox{ on } \Gamma,\\
  \mbox{and }u^S \mbox{ satisfies the radiation condition (\ref{eqn:SRC}).}
\end{array}
\end{equation}

\subsection{BIE formulations}
As above, suppose that $\Omega_-$ is a bounded Lipschitz open set with boundary~$\Gamma$ such that $\Omega_+:= \R^d\setminus \overline{\Omega_-}$ is a  Lipschitz domain, and let $\nu(\bx)$ denote the unit normal at $\bx\in\Gamma$ directed into $\Omega_+$.
To formulate  BIEs for~(\ref{eqn:HE}), we introduce the single-layer potential operator ${\cal S}_k: H^{-1/2}(\Gamma)\to H_{\mathrm{loc}}^1(\R^d)$ and the double-layer potential operator ${\cal D}_k: H^{1/2}(\Gamma)\to H_{\mathrm{loc}}^1(\Omega_{\pm})$, defined by
\[ 
  {\cal S}_k \phi(\bx) := \int_\Gamma \Phi_k(\bx,\by) \phi(\by) \, \rd s(\by), \quad \bx\in \R^d\setminus \Gamma,
\]
and
\[ 
  {\cal D}_k \phi(\bx) := \int_\Gamma \frac{\partial\Phi_k(\bx,\by)}{\partial \nu(\by)} \, \phi(\by)\, \rd s(\by), \quad \bx\in \R^d\setminus \Gamma,
\]
respectively, where the normal $\nu$ is directed into $\Omega_+$.  These layer potentials provide
solutions to~(\ref{eqn:HE}) in $\R^d\setminus\Gamma$; moreover, they also automatically satisfy the radiation condition~(\ref{eqn:SRC}).  In general all the standard BVPs for the Helmholtz
equation~(\ref{eqn:HE}) can be formulated as integral equations on $\Gamma$ using  these layer potentials.

Specifically, we can use Green's representation theorems, which lead to so-called {\em direct} BIE formulations (as we shall see in \S\ref{sec:HNA} these lend themselves particularly well to efficient approximation strategies when $k$ is large).
Denoting the exterior and interior trace operators, from $\Omega_+$ and $\Omega_-$, respectively,  by $\gamma_+$ and $\gamma_-$, and the exterior and interior normal derivative operators by $\partial_\nu^+$ and $\partial_\nu^-$, respectively,
  we have the following result for interior problems (see \cite[Theorem~2.20]{ChGrLaSp:12}).
\begin{thm} \label{thm:grt_int}
If $u\in H^1(\Omega_-)\cap C^2(\Omega_-)$ and, for some $k\geq 0$, $\Delta u + k^2u = 0$ in $\Omega_-$, then
\begin{equation} \label{eq:grt_int}
{\cal S}_k \partial_\nu^- u(\bx) - {\cal D}_k \gamma_- u(\bx) = \left\{\begin{array}{cc}
                                                               u(\bx), & \bx\in \Omega_-, \\
                                                               0, & \bx\in \Omega_+.
                                                             \end{array}
\right.
\end{equation}
\end{thm}
\noindent The following is the corresponding result for exterior problems (see \cite[Theorem~2.21]{ChGrLaSp:12}).
\begin{thm} \label{thm:grt_ext}
If $u\in H^1_{\mathrm{loc}}(\Omega_+) \cap C^2(\Omega_+)$ and, for some $k>0$, $\Delta u + k^2 u = 0$ in $\Omega_+$ and $u$ satisfies the Sommerfeld radiation condition \eqref{eqn:SRC} in $\Omega_+$, then
\begin{equation} \label{eq:grt_ext}
-{\cal S}_k \partial_\nu^+ u(\bx) + {\cal D}_k \gamma_+ u(\bx) = \left\{\begin{array}{cc}
                                                               u(\bx), & \bx\in \Omega_+, \\
                                                               0, & \bx\in \Omega_-.
                                                             \end{array}
\right.
\end{equation}
\end{thm}

The formulae~(\ref{eq:grt_int}) and~(\ref{eq:grt_ext}) lie at the heart of boundary integral methods.  Each expresses the solution throughout the domain in terms of its Dirichlet and Neumann traces on the boundary.  Thus for Dirichlet problems, if the Neumann data can be computed then these formulae immediately give a representation for the solution anywhere in the domain.  Likewise, for Neumann or impedance problems, knowledge of the Dirichlet data is sufficient to determine the solution anywhere in the domain.

In order to derive BIEs for~(\ref{eqn:HE}), for which the ``unknown'' to be computed will be the complementary boundary data required to complete the representation formula for the solution, we need to take Dirichlet and Neumann traces of~(\ref{eq:grt_int}) and~(\ref{eq:grt_ext}).  First, we define the acoustic single- and double-layer operators, $S_k:H^{-1/2}(\Gamma)\to H^{1/2}(\Gamma)$ and $D_k: H^{1/2}(\Gamma)\to H^{1/2}(\Gamma)$, by
\begin{equation} \label{eq:single_op}
  S_k \phi(\bx) := \int_\Gamma \Phi_k(\bx,\by) \phi(\by) \rd s(\by), \quad \bx\in \Gamma,
\end{equation}
and
\[ 
  D_k \phi(\bx) := \int_\Gamma \frac{\partial \Phi_k(\bx,\by)}{\partial \nu(\by)}  \phi(\by) \rd s(\by), \quad \bx\in\Gamma,
\]
respectively, where for Lipschitz $\Gamma$ and $\phi\in H^{1/2}(\Gamma)$, $D_k\phi$ is understood as a Cauchy principal value integral.  These appear on taking boundary traces of ${\cal S}_k$ and ${\cal D}_k$. When we apply the normal derivative operator $\partial_\nu$ to ${\cal S}_k$ and ${\cal D}_k$, we encounter two additional boundary integral operators, the acoustic adjoint double-layer operator $D_k^\prime: H^{-1/2}(\Gamma)\to H^{-1/2}(\Gamma)$ and the acoustic hypersingular operator $H_k: H^{1/2}(\Gamma)\to H^{-1/2}(\Gamma)$, defined by
\begin{equation} \label{eq:Dkpr}
  D^\prime_k \phi(\bx) := \int_\Gamma \frac{\partial \Phi_k(\bx,\by)}{\partial \nu(\bx)} \phi(\by) \rd s(\by),
\end{equation}
and
\[ 
  H_k \phi(\bx) := \frac{\partial}{\partial \nu(\bx)}\int_\Gamma  \frac{\partial \Phi_k(\bx,\by)}{ \partial \nu(\by)}  \psi(\by) \rd s(\by),
\]
respectively.  For Lipschitz $\Gamma$ and $\phi\in L^2(\Gamma)$, \eqref{eq:Dkpr} makes sense as a Cauchy principal value integral, for almost all $\bx\in \Gamma$, while, for $\psi\in H^1(\Gamma)$, $H_k\psi\in L^2(\Gamma)$ is defined in the sense explained in \cite[\S2.3]{ChGrLaSp:12}.

The following jump relations are shown in \cite{McLean}, \cite[\S2.3]{ChGrLaSp:12}.  As operators on $H^{-1/2}(\Gamma)$,
\begin{equation} \label{eq:S traces}
\gamma_\pm {\cal S}_k = S_k, \quad \partial^\pm_\nu {\cal S}_k = \mp \textstyle{\frac{1}{2}}I+D_k^\prime,
\end{equation}
where $I$ is the identity operator. Similarly, as operators on $H^{1/2}(\Gamma)$, we have
\begin{equation} \label{eq:D trace1}
\gamma_\pm {\cal D}_k = \pm \textstyle{\frac{1}{2}}I+D_k, \quad \partial_\nu^\pm {\cal D}_k = H_k.
\end{equation}

Applying the Dirichlet and Neumann trace operators to~(\ref{eq:grt_int}) and~(\ref{eq:grt_ext}), and using~(\ref{eq:S traces}) and~(\ref{eq:D trace1}), we obtain the Calder\'{o}n projection:
\begin{equation} \label{eq:calderon0}
c_{\pm} u = P_{\pm} c_{\pm} u,
\end{equation}
where  $c_{\pm} u = [\gamma_{\pm} u, \partial_\nu^{\pm} u]^T$ is the Cauchy data for $u$ on $\Gamma$, and
\[ 
P_\pm = \pm \left[\begin{array}{cc}
              \gamma_\pm {\cal D}_k & -\gamma_\pm {\cal S}_k \\
              \partial_\nu^\pm {\cal D}_k & -\partial_\nu^\pm {\cal S}_k
            \end{array}\right] = \textstyle{\frac{1}{2}}I \pm \left[\begin{array}{cc}
              D_k & -S_k \\
              H_k & -D_k^\prime
            \end{array}\right],
\]
where $I$ is the ($2\times 2$ matrix) identity operator.  Explicitly, we can rewrite the equations~(\ref{eq:calderon0}) as
\[ 
  \left(D_k-\textstyle{\frac{1}{2}}I\right) \gamma_+ u - S_k \partial_\nu^+ u =0
\]
and
\[ 
  H_k \gamma_+ u - \left(D^\prime_k+\textstyle{\frac{1}{2}}I\right) \partial_\nu^+ u = 0,
\]
for the exterior problem, and
\begin{equation} \label{eq:bie_idp1}
  \left(D_k+\textstyle{\frac{1}{2}}I\right) \gamma_- u - S_k \partial_\nu^- u =0
\end{equation}
and
\[ 
  H_k \gamma_- u - \left(D^\prime_k-\textstyle{\frac{1}{2}}I\right) \partial_\nu^- u = 0,
\]
for the interior problem.  Each of these equations is a linear relationship between the components $\gamma_\pm u$ and $\partial_\nu^\pm u$ of the Cauchy data $c_\pm u$.  These key results can be summarised in the following lemma \cite[Lemma~2.22]{ChGrLaSp:12}.

\begin{lem} \label{lem:grt_boundary}
If $u\in H^1(\Omega_-)\cap C^2(\Omega_-)$ and, for some $k> 0$, $\Delta u + k^2u = 0$ in $\Omega_-$, then
$P_-\,c_-u = c_-u$.
Similarly, if $u\in H^1_{\mathrm{loc}}(\Omega_+) \cap C^2(\Omega_+)$ and, for some $k>0$, $\Delta u + k^2 u = 0$ in $\Omega_+$ and $u$ satisfies the Sommerfeld radiation condition \eqref{eqn:SRC} in $\Omega_+$, then $P_+\,c_+u = c_+u$.
\end{lem}

This lemma is the basis for all the standard direct BIE formulations for interior and exterior acoustic BVPs.  For example, if $u$ satisfies the exterior Dirichlet problem \eqref{prob:edp}, then it follows immediately from Lemma~\ref{lem:grt_boundary} that $\partial_\nu^+ u$ satisfies both
\begin{equation} \label{eq:BIE_edp1}
  S_k \partial_\nu^+ u = \left(D_k-\textstyle{\frac{1}{2}}I\right) h
\end{equation}
and
\begin{equation} \label{eq:BIE_edp2}
  \left(D^\prime_k+\textstyle{\frac{1}{2}}I\right) \partial_\nu^+ u = H_k h.
\end{equation}
Similarly, if $u$ satisfies the interior Dirichlet problem \eqref{prob:idp}, then it follows from Lemma \ref{lem:grt_boundary} that
\begin{equation} \label{eq:BIE_idp1}
  S_k \partial_\nu^- u = \left(D_k+\textstyle{\frac{1}{2}}I\right) h
\end{equation}
and
\[ 
  \left(D^\prime_k-\textstyle{\frac{1}{2}}I\right) \partial_\nu^- u = H_k h.
\]

All these equations are BIEs of the form
\begin{equation} \label{eq:obvious}
  A v = f
\end{equation}
where $A$ is a linear boundary integral operator, or a linear combination of such operators and the identity, $v$ is the solution to be determined and $f$ is given data.  Noting that the same operator $A$ can arise from both interior and exterior problems, it is immediately apparent that, although exterior acoustic problems are generically uniquely solvable, the natural BIE formulations of these problems need not be uniquely solvable for all wavenumbers $k$.  As a specific instance, we noted above that the homogeneous interior Dirichlet problem has non-trivial solutions at a sequence $k_n$ of positive wavenumbers. If $k=k_n$ and $u$ is such a solution then $\partial_\nu^- u$ is a non-trivial solution of~(\ref{eq:BIE_idp1}) with $h=0$ (see, e.g., \cite[Theorem~2.4]{ChGrLaSp:12}) and so, for $k=k_n$, the BIE~(\ref{eq:BIE_edp1})
for the exterior Dirichlet problem \eqref{prob:edp} has infinitely many solutions.

Similar BIE formulations (with the same problems of non-uniqueness) can be derived by utilising the fact that the layer potentials satisfy~(\ref{eqn:HE}) and~(\ref{eqn:SRC}); to satisfy the BVPs, it just remains to take the Dirichlet or Neumann trace of the layer potentials (using the jump relations as above), and then to match with the boundary data.  The resulting formulations are known as indirect BIEs; we do not discuss these further here. As discussed above we will focus on direct formulations in which the unknown to be determined is the normal derivative or trace of the solution in the domain; it is possible as we discuss in~\S\ref{sec:HNA} to bring high frequency asymptotics to bear to understand the behaviour of these solutions and so design efficient approximation spaces.

In order to derive uniquely solvable BIE formulations, the classical approach (dating back to \cite{BuMi:71} for the direct and \cite{BrWe:65,Le:65,Pa:65} for the indirect formulation) is to solve a linear combination of the two equations arising for each problem from the Calder\'{o}n projection.  Taking a linear combination of \eqref{eq:BIE_edp1} and \eqref{eq:BIE_edp2}, we obtain
\begin{equation} \label{eq:BIE_dp_main}
  A_{k,\eta}^\prime \partial_\nu^+ u = B_{k,\eta} h,
\end{equation}
where $\eta\in \C$ is a parameter that we need to choose and $A_{k,\eta}^\prime$, $B_{k,\eta}$ are the operators
\[ 
  A_{k,\eta}^\prime = \textstyle{\frac{1}{2}}I + D_k^\prime -\ri \eta S_k \mbox{ and }  B_{k,\eta} = H_k + \ri \eta \left(\textstyle{\frac{1}{2}}I - D_k \right).
\]
Both $A_{k,\eta}^\prime$ and $B_{k,\eta}$ are invertible (considered as operators between appropriate pairs of Sobolev spaces) for all $k>0$ provided $\mathrm{Re}\, \eta \neq 0$, see e.g. \cite[Theorem~2.27]{ChGrLaSp:12}.

The corresponding direct formulation for the exterior impedance problem is
\begin{equation} \label{eq:BIE_eip3}
C_{k,\eta,\beta} \gamma_+ u = A_{k,\eta}^\prime h,
\end{equation}
where
\begin{equation} \label{eq:Cdef}
C_{k,\eta,\beta} \phi := B_{k,\eta}\phi + \ri k A_{k,\eta}^\prime (\beta \phi), \quad \phi\in H^{1/2}(\Gamma),
\end{equation}
is invertible (considered as an operator between an appropriate pair of Sobolev spaces) for all $k>0$ provided $\mathrm{Re}\, \eta \neq 0$;  again see \cite[Theorem~2.27]{ChGrLaSp:12}.

That the exterior Dirichlet, Neumann and impedance BVPs can be solved by combined potential direct integral equation formulations follows from, e.g., \cite[Corollary~2.28]{ChGrLaSp:12}.  Specifically:

\begin{cor} \label{cor:main_equivalence}
Suppose that $k>0$ and $\eta\in \C$ with $\mathrm{Re}\,\eta\neq 0$. Then both the following statements hold.

(i) If $u$ is the unique solution of \eqref{prob:edp} then $\partial_\nu^+ u\in H^{-1/2}(\Gamma)$ is the unique solution of \eqref{eq:BIE_dp_main}. Further, if $h=\gamma_+ u\in H^s(\Gamma)$ with $1/2<s\leq 1$ then $\partial_\nu^+ u\in H^{s-1}(\Gamma)$.

(ii) If $u$ is the unique solution of \eqref{prob:eip} then $\gamma_+ u\in H^{1/2}(\Gamma)$ is the unique solution of \eqref{eq:BIE_eip3}. Further, if $h=\gamma_+ u\in H^s(\Gamma)$ with $-1/2<s\leq 0$ then $\gamma_+ u\in H^{s+1}(\Gamma)$.
\end{cor}

Although the combined potential integral equations~(\ref{eq:BIE_dp_main}) and (\ref{eq:BIE_eip3}) are the most common integral equation formulations for exterior Dirichlet and impedance scattering problems, other formulations are possible.  One that is of particular interest for boundary element methods is the so called ``star-combined integral equation'', proposed for the exterior Dirichlet problem in the case when $\Omega_-$ is star-shaped with respect to an appropriately chosen origin in~\cite{SpChGrSm:11}.

Specifically, if $u$ satisfies the exterior Dirichlet problem \eqref{prob:edp} with $h\in H^1(\Gamma)$, then for $\eta(\bx) := k|\bx| + \ri (d-1)/2$, $\bx\in\Gamma$, we have
\begin{equation} \label{eq:mainie_gen}
  \scA_k  \partial_\nu^+ u = (\bx\cdot \left[ \nu H_k + \nabla_\Gamma D_k - \textstyle{\frac{1}{2}}\nabla_\Gamma\right] - \ri \eta  \left(-\textstyle{\frac{1}{2}}I+ D_k\right))h,
\end{equation}
where the ``star-combined'' operator $\scA_k: L^2(\Gamma)\to L^2(\Gamma)$ is defined by
\begin{equation} \label{eq:calAk}
  \scA_k := \bx\cdot \nu \left(\textstyle{\frac{1}{2}}I+D_k^\prime\right) + \bx \cdot \nabla_\Gamma S_k - \ri \eta S_k.
\end{equation}
It is shown in~\cite{SpChGrSm:11} that, if $\Omega_-$ is star-shaped with respect to the origin, specifically if, for some $c>0$,
\begin{equation} \label{eq:xdotn}
  \bx \cdot \nu \geq c, \quad \mbox{for almost all } \bx\in \Gamma,
\end{equation}
then $\scA_k:L^2(\Gamma)\to L^2(\Gamma)$ is invertible with $\|\scA_k^{-1}\|_{L^2(\Gamma)\leftarrow L^2(\Gamma)} \leq 2/c$.  Indeed, $\scA_k:L^2(\Gamma)\to L^2(\Gamma)$ is in fact coercive, as defined in~\S\ref{sec:BEM} below, 
a much stronger property whose significance for numerical solution is explained in~\S\ref{sec:BEM}.

Turning to the Dirichlet screen problem \eqref{prob:screen}, it follows easily from Theorem \ref{thm:grt_ext} that, where $\Gamma$ has the form \eqref{eqn:screen} and $D:= \R^d\setminus \overline \Gamma$, if $u\in W^1_{\mathrm{loc}}(D) \cap C^2(D)$ and, for some $k>0$, $\Delta u + k^2 u = 0$ in $D$ and $u$ satisfies the Sommerfeld radiation condition \eqref{eqn:SRC} in $D$, then
\begin{equation} \label{eq:grt_screen}
-{\cal S}_k \left[\partial_\nu u\right](\bx) + {\cal D}_k [u](\bx) =
                                                               u(\bx),  \quad \bx\in D,
\end{equation}
where $[u]:= \gamma_+u-\gamma_-u\in \tilde H^{1/2}(\Gamma)$ and $\left[\partial_\nu u\right] = \partial_\nu^+ u-\partial_\nu^- u\in \tilde H^{-1/2}(\Gamma)$ denote the jumps in $u$ and its normal derivative across the plane $x_d=0$ containing $\Gamma$. Here the normal is directed in the $x_d$ direction and, for $s\in\R$, the ``tilde'' space $\tilde H^s(\Gamma)$ denotes the set of those $\phi\in H^s(\R^{d-1})$ (here we are identifying $\R^{d-1}$ with the plane containing $\Gamma$) that have support in $\overline \Gamma$. It is clear that the jumps in $u$ and its normal derivative across the plane $x_d=0$, which are zero outside the screen, are in these ``tilde'' spaces. Thus if $u$ satisfies the Dirichlet screen problem, in which case $[u]=0$ (see \cite[\S3.3]{ScreenCoerc}), it holds that
$$
u(\bx) = -{\cal S}_k \left[\partial_\nu u\right](\bx),  \quad \bx\in D,
$$
and, taking traces, that
\begin{equation} \label{eq:screenie}
S_k\left[\partial_\nu u\right](\bx) = -h(\bx), \quad \bx\in \Gamma.
\end{equation}

The operator $S_k$ in this equation, defined on the screen $\Gamma$ by \eqref{eq:single_op}, is bounded and invertible as an operator $S_k:\tilde H^{-1/2}(\Gamma)\to H^{1/2}(\Gamma)$; see \cite{stephan87,ScreenCoerc}. Indeed as noted in \cite{Co:04} (this result uses that $\Gamma$ is planar, and is derived by applying the Fourier transform which diagonalises $S_k$; see \cite{ScreenCoerc}), and noting that $H^s(\Gamma)$ is the dual space of $\tilde H^{-s}(\Gamma)$ (see \cite{McLean,ScreenCoerc}), $S_k:\tilde H^{-1/2}(\Gamma)\to H^{1/2}(\Gamma)$, like the operator $\scA_k$ introduced above, is coercive in the sense of \S\ref{sec:proj} below, moreover with the $k$-dependence of the coercivity constants understood in each case.

We remark that it is often considered surprising that it is possible to write down coercive formulations of BVPs for the Helmholtz equation, which BVPs are often considered to be ``highly indefinite'', at least in the high frequency regime; see the discussion in~\cite{MoSp14}.

If $u$ satisfies the transmission problem \eqref{prob:trans}, then both representations \eqref{eq:grt_int} and \eqref{eq:grt_ext} hold. Applying \eqref{eq:S traces} and \eqref{eq:D trace1} we deduce that a linear operator equation of the form \eqref{eq:obvious}  holds for the unknown $v = \left[
                       \gamma_+u,
                       \partial_\nu^+ u \right]^T \in H^{1/2}(\Gamma)\times H^{-1/2}(\Gamma)$, with
\begin{equation} \label{eq:BIE_trans}
A = \left[\begin{array}{cc}
            I + D_{k_-}-D_{k_+} & S_{k_+}-S_{k_-} \\
            H_{k_-}-H_{k_+} & I + D_{k_+}^\prime-D_{k_-}^\prime
          \end{array}
\right], \;\; f = \left[
                                        \begin{array}{c}
                                          -\frac{1}{2}h-D_{k_-}h + S_{k_-}g \\
                                            -\frac{1}{2}g+D_{k_-}^\prime g - H_{k_-}h\\
                                        \end{array}
                                      \right].
\end{equation}
The operator $A$ is bounded and invertible as an operator on $H^{1/2}(\Gamma)\times H^{-1/2}(\Gamma)$, but also, adapting arguments of \cite{ToWe93}, as an operator on $H^{1}(\Gamma)\times L^2(\Gamma)$, and as an operator on $L^2(\Gamma)\times L^2(\Gamma)$ \cite{GrHeLa:13}.

We conclude this section by stating precisely direct boundary integral equation formulations for the sound-soft, impedance, screen, and transmission scattering problems, \eqref{prob:ssp}, \eqref{prob:isp}, \eqref{prob:scsp}, and \eqref{prob:tsp}.  These follow immediately from: the integral equation formulations \eqref{eq:BIE_dp_main} and \eqref{eq:mainie_gen} for the sound-soft scattering problem \eqref{prob:ssp}, and \eqref{eq:screenie} for the screen problem \eqref{prob:scsp}, in each case applying these equations with $u$ replaced by $u^S$ and $h= -u^I|_\Gamma$; from the integral equation \eqref{eq:BIE_eip3} for the impedance scattering problem, with $u$ replaced by $u^S$ and $h = -(\partial_{\nu} + \ri k \beta \gamma) u^I|_{\Gamma}$; and from \eqref{eq:obvious} and \eqref{eq:BIE_trans} for the transmission problem \eqref{prob:tsp}, with $u$ replaced by $u^*$ defined as $u^S$ in $\Omega_+$ and $u$ in $\Omega_-$, and $h=-u^I|_\Gamma$, $g=-\partial_\nu u^I|_\Gamma$.  But thanks to the special form of the boundary data $h$ and $g$, we can work with versions of these integral equations with simplified expressions for the inhomogeneous terms.  Specifically (see, e.g., \cite[Theorem~2.43]{ChGrLaSp:12}), in the case that $u^S$ satisfies the sound-soft scattering problem \eqref{prob:ssp} it holds that $\partial_\nu^+ u\in L^2(\Gamma)$ and
\begin{equation} \label{eq:grt_ss}
  u(\bx) = u^I(\bx) - \int_\Gamma \Phi_k(\bx,\by) \partial_\nu^+ u(\by)\, \rd s(\by), \quad \bx\in \Omega_+;
\end{equation}
similarly, if $u^S$ satisfies the screen scattering problem \eqref{prob:scsp}, then $\left[\partial_\nu u\right]\in \tilde H^{-1/2}(\Gamma)$ and
\begin{equation} \label{eq:grt_screen_scat}
  u(\bx) = u^I(\bx) - \int_\Gamma \Phi_k(\bx,\by) \left[\partial_\nu u\right] (\by)\, \rd s(\by), \quad \bx\in D.
\end{equation}
In the case that $u^S$ satisfies the impedance scattering problem \eqref{prob:isp} it holds that $\gamma_+ u\in H^1(\Gamma)$ and
\begin{equation} \label{eq:grt_imp}
  u(\bx) = u^I(\bx) + \int_\Gamma \left(\frac{\partial\Phi_k(\bx,\by)}{\partial \nu(\by)} +\ri k\beta(\by) \Phi_k(\bx,\by)\right) \gamma_+ u(\by)\, \rd s(\by), \quad \bx\in \Omega_+.
\end{equation}
Finally, in the case that $u$ satisfies the transmission scattering problem \eqref{prob:tsp} it holds that $\gamma_+ u =\gamma_- u\in H^1(\Gamma)$ and $\partial_\nu^+u=\partial_\nu^-u\in L^2(\Gamma)$, and (e.g., \cite{GrHeLa:13})
\begin{equation} \label{eq:grt_trans}
  u(\bx) = u^I(\bx) + \int_\Gamma \left(\frac{\partial\Phi_k(\bx,\by)}{\partial \nu(\by)}\gamma_+u(\by) -\Phi_k(\bx,\by) \partial_\nu^+ u(\by)\right)\, \rd s(\by), \quad \bx\in \Omega_+,
\end{equation}
while \eqref{eq:grt_int} holds in $\Omega_-$.

The direct versions of the combined potential equations are then the integral equation \eqref{eq:BIE_dp_main} for the sound-soft problem (with $u$ replaced by $u^S$ and $h=-u^I|_\Gamma$) and \eqref{eq:BIE_eip3} for the impedance scattering problem (with $u$ replaced by $u^S$ and $h=-(\partial_{\nu} + \ri k \beta \gamma) u^I|_{\Gamma}$).
Versions of these equations with simplified expressions for the inhomogeneous terms on the right hand side are as follows (see, e.g., \cite[Theorems~2.46, 2.47]{ChGrLaSp:12}).  Suppose that $u^S$ satisfies the sound-soft scattering problem \eqref{prob:ssp}. Then $\partial_\nu^+ u\in L^2(\Gamma)$ satisfies the integral equation
\begin{equation} \label{eq:BIE_sss3}
  A^\prime_{k,\eta} \partial_\nu^+ u = f_{k,\eta} := [\partial_\nu^+ u^I - \ri \eta u^I]|_\Gamma,
\end{equation}
and this equation is uniquely solvable for all $k>0$ if $\mathrm{Re}\,\eta\neq 0$.  In the case that $\Omega_-$ is star-shaped with respect to the origin, satisfying \eqref{eq:xdotn} for some $c>0$, and $\eta(\bx) := k|\bx| + \ri (d-1)/2$, $\bx\in\Gamma$, an alternative boundary integral equation formulation is
\begin{equation} \label{eq:BIE_sss4}
  \scA_k \partial_\nu^+ u = f_k := [\bx\cdot \nabla u^I-\ri\eta u^I]|_\Gamma,
\end{equation}
with $\scA_k$ defined by \eqref{eq:calAk}, and this equation is uniquely solvable for all $k>0$.  We will refer to \eqref{eq:BIE_sss4} as the {\em star-combined} integral equation.

If $u^S$ satisfies the impedance scattering problem \eqref{prob:eip} and $\eta\in \C$, then $\gamma_+ u\in H^1(\Gamma)$ satisfies the integral equation
\begin{equation} \label{eq:BIE_imp2}
C_{k,\eta,\beta} \,\gamma_+ u = -\left.\left[\frac{\partial u^I}{\partial \nu}-\ri \eta u^I\right]\right|_\Gamma,
\end{equation}
with $C_{k,\eta,\beta}$ defined by \eqref{eq:Cdef}. If $\mathrm{Re}\, \eta \neq 0$ then $\gamma_+ u$ is the unique solution in $H^{1/2}(\Gamma)$ of this equation.
If $u^S$ satisfies the screen scattering problem \eqref{prob:scsp} then, from \eqref{eq:screenie}, it follows that $\left[\partial_\nu u\right]\in \tilde H^{-1/2}(\Gamma)$ satisfies
\begin{equation} \label{eq:BIE_sc2}
S_k \left[\partial_\nu u\right] = u^I|_\Gamma\, .
\end{equation}
Finally, if $u$ satisfies the transmission scattering problem \eqref{prob:tsp} then  \eqref{eq:obvious} holds for $v = \left[\gamma_+u, \partial_\nu^+u\right]^T\in H^1(\Gamma)\times L^2(\Gamma)$ with $A$ given by \eqref{eq:BIE_trans} and $f=\left[ u^I|_\Gamma, \partial_\nu u^I|_\Gamma\right]$  \cite{GrHeLa:13}.

\section{The boundary element method}
\label{sec:BEM}

In this section we describe methods for the numerical solution of integral equations of the form~(\ref{eq:obvious}), i.e.
\begin{equation}
  A v  = f,
  \label{eqn:BIE}
\end{equation}
where $A: \cV\to \cV^\prime$ is a linear boundary integral operator mapping some Hilbert space $\cV$ to its dual space $\cV^\prime$, or a linear combination of such operators and the identity, $v\in\cV$ is the solution to be determined and $f\in\cV^\prime$ is given data.  Specifically: for the integral equation \eqref{eq:BIE_sss3}, we have $A=A^\prime_{k,\eta}$, $\cV=\cV^\prime= L^2(\Gamma)$, $v=\partial_\nu^+ u$, and $f=f_{k,\eta}$; for \eqref{eq:BIE_sss4}, we have $A=\scA_k$, $\cV=\cV^\prime= L^2(\Gamma)$, $v=\partial_\nu^+ u$,  and $f=f_{k}$; for \eqref{eq:BIE_imp2}, we have $A=C_{k,\eta,\beta}$, $\cV=H^{1/2}(\Gamma)$, $\cV^\prime= H^{-1/2}(\Gamma)$, $v=\gamma_+ u$,  and $f=-\left.\left[\frac{\partial u^I}{\partial \nu}-\ri \eta u^I\right]\right|_\Gamma$; and for \eqref{eq:BIE_sc2} we have $A=S_k$, $\cV = \tilde H^{-1/2}(\Gamma)$, $\cV^\prime = H^{1/2}(\Gamma)$, $v=[\partial_\nu u]$, and $f= u_I|_\Gamma$. Finally, for the transmission scattering problem \eqref{prob:tsp}, $A$ is given by   \eqref{eq:BIE_trans}, we can take $\cV=\cV^\prime = L^2(\Gamma)\times L^2(\Gamma)$,  and $v = \left[\gamma_+u, \partial_\nu^+u\right]^T$, $f=\left[ u^I|_\Gamma, \partial_\nu u^I|_\Gamma\right]$.

In order that the equation we are solving is well-posed, it is important that the boundary integral operator $A$ is invertible.  In order to prove convergence of numerical schemes and deduce error estimates, additional properties of the operator are also required, as we discuss below.

The most commonly used methods for the solution of equations of the form~(\ref{eqn:BIE}) arising from scattering problems are Galerkin, collocation and Nystr\"{o}m schemes.  For an operator equation of the form~(\ref{eqn:BIE}) the
Galerkin method consists of choosing a finite dimensional approximating space $\cV_N\subset\cV$ and then seeking an approximate solution $v_N \in \cV_N$ such that
\begin{equation}
  \langle A v_N , w_N\rangle = \langle f, w_N \rangle \ , \;\;\mbox{ for all } w_N\in \cV_N,
  \label{eqn:gal_gen}
\end{equation}
where, for $f\in \cV^\prime$, $w\in \cV$, $\langle f,w\rangle$ denotes the action of the functional $f$ on $w$.  In all of the examples we consider below, this duality pairing is just the $L^2$ inner product.  In terms of practical implementation, this requires two integrations (corresponding to the integral operator $A$ and the duality pairing on the left hand side); to avoid this, the collocation or Nystr\"{o}m methods are often preferred.  For an operator equation of the form~(\ref{eqn:BIE}), with $A: L^{\infty}(\Gamma)\to L^{\infty}(\Gamma)$, the collocation method consists of choosing a finite dimensional approximating space $\cV_N$, and a set of $N$ distinct node points $x_1,\ldots,x_N\in\Gamma$, and then seeking an approximate solution $v_N \in \cV_N$ such that
\[
  A v_N (x_i) \ = \ f(x_i) \mbox{ for  } i=1,\ldots,N.
\]
The Nystr\"{o}m method is even simpler - it consists of just discretising the integral operator in~(\ref{eqn:BIE}) directly, via a quadrature rule, i.e. solving
\[ A_N v_N = f, \]
where $A_N$, $N=1,2,\ldots$, represents a convergent sequence of numerical integration operators (see, e.g., \cite[Chapter~12]{Kress}).

Although Nystr\"{o}m and collocation methods are both simpler to implement than Galerkin methods, we focus on the Galerkin method here.  One part of the rationale for this choice is that a key step (and our major focus below) in designing both Galerkin and collocation methods is designing subspaces $\cV_N$ that can approximate the solution accurately, with a relatively low number of degrees of freedom $N$. Everything we say below about designing $\cV_N$ for the Galerkin method applies equally to the collocation method, and indeed to other numerical schemes where we select the numerical solution from an approximating subspace.  The second part of our rationale is that, for collocation and Galerkin and other related methods, choosing a subspace from which we select the numerical solution is only part of the story. We have also to design our numerical scheme so that the numerical solution selected is ``reasonably close'' to the best possible approximation from the subspace. It is only for the Galerkin method that any analysis tools have been developed that can guarantee that this is the case for the hybrid numerical-asymptotic schemes we will describe in~\S\ref{sec:HNA} below, leading to guaranteed convergence and error bounds, at least for some problems and geometries, that we discuss in \S\ref{sec:HNA}. (In fairness we should note that, while it is not known how to carry out the complete error analysis needed, numerical experiments suggest that both Nystr\"{o}m \cite{BGMR04,GBR05,BrRe:07} and collocation \cite{ArChLa:07} hybrid methods can be as effective as Galerkin methods for numerical solution of certain high frequency scattering problems.)
The last part of our rationale is that, solving many of the implementation issues for users, open source Galerkin BEM software is becoming available, for example the boundary element software package BEM++ that exclusively uses Galerkin methods, which we discuss a little more in~\S\ref{sec:standard} below.

\subsection{The Galerkin method}
\label{sec:proj}

In this section we present a general framework for the solution of~(\ref{eqn:BIE}) by the Galerkin method, and we begin by writing it in weak form as follows:
find $v\in \cV$ such that
\begin{equation}\label{eq:var}
  \langle Av, w\rangle = \langle f,w\rangle, \quad \mbox{for all } w\in \cV.
\end{equation}
The Galerkin method for approximating \eqref{eq:var} then seeks a solution $v_N\in \cV_N\subset \cV$, where $\cV_N$ is a finite-dimensional subspace, requiring that
\begin{equation} \label{eq:varfin}
  \langle Av_N, w_N\rangle = \langle f,w_N\rangle, \quad \mbox{for all } w_N\in \cV_N.
\end{equation}
The standard analysis of Galerkin methods assumes that the operator $A: \cV\to \cV^\prime$ is injective and takes the form $A = B + C$,
where $C: \cV\to \cV^\prime$ is compact and $B: \cV\to \cV^\prime$ is coercive, by which we mean that, for some $\alpha>0$ (the {\em coercivity constant}),
\begin{equation} \label{eq:coercive}
  |\langle Bv, v\rangle| \geq \alpha \|v\|_{\cV}^2, \quad \mbox{for all } v\in \cV.
\end{equation}
In this case we say that $A$ is a compactly perturbed coercive operator and the following theorem holds (see, e.g., \cite[Theorem~8.11]{St:08}).
\begin{thm} \label{thm:ceagen}
Suppose the bounded linear operator $A:\cV\to\cV^\prime$ is a compactly perturbed coercive operator and is injective. Suppose moreover that $(\cV_N)_{N\in \N}$ is a sequence of approximation spaces converging to $\cV$ in the sense that
$$
  \inf_{w_N\in \cV_N} \, \|w-w_N\|_{\cV} \to 0 \quad \mbox{ as } N\to\infty,
$$
for every $w\in \cV$. Then there exists $N_0\in \N$ and $C>0$ such that, for $N\geq N_0$, \eqref{eq:varfin} has exactly one solution $v_N\in \cV_N$, which satisfies
the quasi-optimal error estimate
\begin{equation}
  \|v-v_N\|_{\cV} \leq C \, \inf_{w_N\in \cV_N} \, \|v-w_N\|_{\cV}.
  \label{eqn:quasiopt}
\end{equation}
\end{thm}

This theorem is relevant to the combined potential integral equation~\eqref{eq:BIE_imp2} for the impedance scattering problem, for general Lipschitz $\Gamma$, since $C_{k,\eta,\beta}:H^{1/2}(\Gamma)\to H^{-1/2}(\Gamma)$ is a compactly perturbed coercive operator and is also injective for all $k>0$ if $\mathrm{Re}\, \eta\neq 0$ (see~\cite[Theorem~7.8]{McLean} and \cite[Theorem~2.27]{ChGrLaSp:12}).  Theorem \ref{thm:ceagen} also applies to the integral equation \eqref{eq:BIE_sss3} for the sound-soft scattering problem if $\Gamma$ is $C^1$, and also, for the 2D case, if $\Gamma$ is Lipschitz and a curvilinear polygon (see~\cite[\S2.11]{ChGrLaSp:12} for details).  However, it cannot be applied directly to  \eqref{eq:BIE_sss3} for general Lipschitz $\Gamma$, for which case it is an open question whether $A^\prime_{k,\eta}$ is a compactly perturbed coercive operator.  Introducing additional operators into the equations it is possible to formulate BIEs for this problem that do satisfy the conditions required by Theorem~\ref{thm:ceagen}, but the downside of this is that this approach requires additional computational effort.  For a survey of recent progress in this area we refer to \cite[\S2.11 and \S5.7]{ChGrLaSp:12}, see also \cite{BuHi:05b,BuSa:06,EnSt:07,EnSt:08}.

However, even in cases where it can be applied, the  classical theory does not tell us how $C$ and $N_0$ depend on $k$.  Moreover, suppose that we solve~(\ref{eqn:BIE}) using the Galerkin method~(\ref{eqn:gal_gen}) on a family of, for example, $N$-dimensional spaces $\cV_N$  of piecewise polynomial functions of fixed degree on a quasi-uniform sequence of meshes on $\Gamma$ with diameter $h \to 0$ (so that $N$ is proportional to $h^{1-d}$).  Then, for such standard piecewise polynomial approximation spaces, the best approximation error $\inf_{w_N \in \cV_{N}} \Vert v - w_N \Vert_{\cV}$ will be highly $k$-dependent.  For example, if $v$ is $p+1$ times continuously differentiable on each mesh interval then standard estimates for piecewise polynomial approximation of degree $p$ suggest that we might be able to bound the difference between $v$ and its best polynomial approximation $p_N$ on each interval by
\begin{equation}
  \| v - p_N \|_{\cV} \ \leq \ C h^{p+1} \|v^{(p+1)}\|_{\cV},
  \label{eqn:best_standard}
\end{equation}
for some constant $C>0$ (which may depend on $p$).  But since $v$ solving~(\ref{eqn:BIE}) represents either $\partial_\nu^+ u$ or $\gamma_+ u$, where $u$ solves~(\ref{eqn:HE}), it follows that $v$ will be oscillatory, with the $j$th derivative of $v$ being, in general, of order $k^j$.  In this case, (\ref{eqn:best_standard}) becomes
\[ \| v - p_N \|_{\cV} \ \leq \ C (hk)^{p+1}, \]
with this estimate carrying over naturally to the right-hand side of~(\ref{eqn:quasiopt}), i.e.
\[ \inf_{w_N \in \cV_{N}} \Vert v - w_N \Vert_{\cV} \ \leq \ C (hk)^{p+1}. \]
Thus  $h$  needs to decrease with  $\mathcal{O}(k^{-1})$, and possibly faster, just to keep the error bounded as $k\to\infty$.
Hence,  standard (piecewise) polynomial BEMs applied directly to approximate the oscillatory solution of scattering problems will have complexity of at least $\mathcal{O}(h^{1-d}) = \mathcal{O}(k^{d-1})$ as $k\to\infty$ (see, e.g., \cite{GrLoMeSp14} for further details).

In the case that the integral operator $A$ in~(\ref{eqn:BIE}) is bounded and coercive, we can follow a slightly different approach.  In this case, the Lax-Milgram lemma guarantees that \eqref{eq:var} has exactly one solution $v\in \cV$ for every $w\in \cV$, with
\[ 
  \|v\|_{\cV}\leq \alpha^{-1} \|f\|_{\cV'},
\]
and existence of the Galerkin solution and quasi-optimality is guaranteed by C\'ea's lemma (see, e.g., \cite[Lemma~2.48]{ChGrLaSp:12}).

\begin{lem} \label{lem:cea} {\bf C\'ea's lemma.} Suppose that the bounded linear operator $A:\cV\to\cV^{\prime}$ is coercive (i.e.\ it satisfies~(\ref{eq:coercive}) for some $\alpha>0$), and moreover that, for some constant $C>0$, $|\langle Av,w \rangle|\leq C\|v\|_{\cV}\|w\|_{\cV}$, for all $v,w\in\cV$.  Then~\eqref{eq:varfin} has exactly one solution $v_N\in \cV_N$, which satisfies
\begin{equation} \label{eq:genquasi}
\|v-v_N\|_{\cV} \leq \frac{C}{\alpha} \, \inf_{w_N\in \cV_N} \, \|v-w_N\|_{\cV}.
\end{equation}
\end{lem}
This theorem is relevant to the star-combined formulation~(\ref{eq:BIE_sss4}) for sound-soft scattering by star-shaped obstacles, to the BIE~(\ref{eq:BIE_sc2}) for the screen scattering problem, and also to the standard combined potential formulation~(\ref{eq:BIE_sss3}) for a certain range of geometries (see~\cite{SpKaSm:14} for details).  It is sometimes presumed, since the standard domain-based variational formulations of BVPs for the Helmholtz equation are  standard examples of indefinite problems where coercivity does not hold, at least for sufficiently large $k$, that the same should hold true for weak formulations arising via integral equation formulations. However recent results, discussed in \cite[\S5]{ChGrLaSp:12} and see also \cite{SpChGrSm:11, BS, BePhSp:14, SpKaSm:14,ScreenCoerc}), show that coercivity holds for these BIEs for a range of geometries, with $\alpha$ bounded away from zero for all sufficiently large~$k$, moreover with the $k$-dependence of $\alpha$ and $C$ in~(\ref{eq:genquasi}) explicitly known in many cases.

The advantage of using this version of C\'ea's lemma, as opposed to that stated in Theorem~\ref{thm:ceagen}, is that, if the $k$-dependence of the continuity and coercivity constants are known, then everything in~(\ref{eq:genquasi}) is $k$-explicit, as opposed to the unknown $k$-dependence of the constants $C$ and $N_0$ in~(\ref{eqn:quasiopt}).  In either case though, the question of how the best approximation estimate $\inf_{w_N\in \cV_N} \, \|v-w_N\|$ depends on $k$ is crucial.

In the next section we briefly mention recent progress on standard schemes, with piecewise polynomial approximation spaces, where the best approximation estimate may grow with~$k$ but other ideas can be applied to improving efficiency and accuracy even at large frequencies.  These approaches have the advantage that little need be known a-priori about the behaviour of the solution.  On the other hand, if one can understand the high frequency asymptotic behaviour of the solution then one can design one's approximation space accordingly in order to efficiently represent the oscillatory solution at high frequencies.  This is the idea behind the hybrid numerical-asymptotic approach, which we describe more fully in~\S\ref{sec:HNA}.

\subsection{BEM with standard (piecewise polynomial) approximation spaces}
\label{sec:standard}

As alluded to above, boundary element methods with standard piecewise polynomial approximation spaces suffer from the restriction that the number of degrees of freedom required to achieve a prescribed level of accuracy will grow rapidly (at least with $O(k)$ for 2D and $O(k^2)$ for 3D problems) as $k$ increases. (Recent progress in the analysis of the standard BEM for the Helmholtz equation, teasing out explicitly the dependence of this error on the wavenumber, is reported in \cite{LoMe:10,GrLoMeSp14}.) As boundary element methods lead to linear systems with dense matrices, this can render the computational cost (both in terms of speed and memory) of standard linear solvers impractical when $k$ is very large.

To ease this problem, much effort has been put into developing preconditioners (see, e.g., \cite{HaCh:03,BrLi:13,HiJeUr:13}), efficient iterative solvers (see, e.g., \cite{AmMa:98, ChNe:00, ErGa:12}, fast multipole methods (see, e.g., \cite{Da:02, DoJiCh:03, DaHa:04, ChCrGiGrEtHuRoYaZh:06, OfStWe:06}), and matrix compression techniques (see, e.g., \cite{BaHa:08, Be:08}) for Helmholtz and related problems.

These advances have enabled the solution of larger and larger problems, but a common argument put forward against the widespread usage of boundary element methods is the added investment required by the user to initiate code development incorporating such features, certainly in comparison to more intuitive and less computationally complicated finite difference and particularly finite element methods.  Moreover, whereas there exist many widely available computational resources for efficient development of finite element software, until recently comparable resources have been somewhat lacking for boundary element computations.  To redress this balance, there has been much recent activity on the development of software packages specifically for boundary element methods.  The key aim of much of this software is to make many of the recent advances outlined above accessible to the more casual user.

We mention in particular the package BEM++, which has recently been developed for the solution of a range of 3D linear elliptic PDEs, including the Helmholtz equation (see~\cite{SmArBePhSc:13} for details).  BEM++ utilises $hp$-Galerkin schemes, with low order approximation spaces (piecewise constant or piecewise linear continuous spaces), or else allows higher order polynomials on flat triangles.  Moreover, BEM++ implements many of the developments listed above within a single framework.  For further details we refer to~\cite{SmArBePhSc:13}.  Many other codes have also been developed in recent years, for example BETL~\cite{HiKi:12} and HyENA~\cite{MeMeUrRa:10}. 
For more details, we refer to the review of this topic in~\cite{SmArBePhSc:13}.

\subsection{Hybrid numerical-asymptotic BEMs for high frequency problems}
\label{sec:HNA}

In this section we describe recent progress in the design, analysis and implementation of hybrid numerical-asymptotic (HNA) BEMs for BVPs for the Helmholtz equation that model time harmonic acoustic wave scattering.  As alluded to above, the problem~(\ref{eqn:HE}) has solutions that oscillate in space with wavelength $\lambda=2\pi/k$ (e.g.\ the plane waves $u(\bx)=\exp(\ri k \bx.\hat{\ba})$, where $\hat{\ba}\in\R^d$ is a unit vector, are solutions).  Since the number of oscillations of linear combinations of such solutions will in general grow linearly with $k$ (in each direction), the number of degrees of freedom required to represent this oscillatory solution by conventional (piecewise polynomial) boundary elements must also grow with order $k^{d-1}$.  This lack of robustness with respect to increasing values of $k$ (which puts many problems of practical interest beyond the reach of standard algorithms) is the motivation behind the development of HNA algorithms.

The key idea of the HNA approach is to exploit, within the design of the numerical method, detailed information about the oscillations of the solution, based on advances in the understanding of the high frequency behaviour of solutions to the Helmholtz equation.  Known highly oscillatory components of $v$ solving~(\ref{eqn:BIE}) (with, for example,  $v=\partial_\nu^+u$ or $v=\gamma_+u$, where $u$ solves~(\ref{eqn:HE})) are treated exactly in the
algorithm, leaving only more slowly-varying components  to be  approximated by piecewise polynomials, i.e.\ we combine conventional piecewise polynomial approximations with high-frequency asymptotics to build basis functions suitable for representing the oscillatory solutions.  This yields methods which require very much fewer degrees of freedom as $k\to\infty$.

More specifically, we approximate the ($k$-dependent) solution $v$ using (in general) an ansatz of the form:
\begin{equation} \label{eqn:ansatz}
  v(\bx) \ \approx\  V_{0}(\bx,k) \ + \ \sum_{m=1}^M  V_m(\bx,k) \, \exp(\ri k \psi_m(\bx)), \quad \bx\in\Gamma \ .
\end{equation}
In this representation, $V_0$ is a known (generally oscillatory)
function, the phases $\psi_m$ are chosen {\em a priori}, and the amplitudes $V_{m}$, $m=1,\ldots,M$, are approximated numerically.
The idea (and in many cases this can be rigorously proven) is that,
if the phases are carefully chosen, then $V_m(\cdot,k)$, $m=1,\ldots,M$, will be much less oscillatory than $v$
and so can be better approximated by piecewise polynomials than $v$ itself.

The crucial first step then lies in choosing $V_0$ and $\psi_m$, $m=1,\ldots,M$, appropriately, and the exact way in which this is done depends on the geometry and boundary conditions of the problem under consideration.  Section~3 of the recent review article~\cite{ChGrLaSp:12} provides a historical survey before explaining the key ideas for a number of examples, specifically scattering by smooth convex obstacles (see also~\cite{BGMR04,DGS07,HuVa:07a,GH11}), an impedance half-plane (see also~\cite{ChLaRi:04,LaCh:06}), sound-soft convex polygons (see also~\cite{Convex,HeLaMe:11}), convex curvilinear polygons (see also~\cite{LaMoCh:10}), convex impedance polygons (see also~\cite{CWLM}), nonconvex polygons (see also~\cite{NonConvex}) and multiple scattering configurations (see also~\cite{GBR05,ER09,AnBoEcRe:10}).

Here, we describe some more recent advances, overlapping only in a fairly minor way with the examples described in~\cite{ChGrLaSp:12}.  We begin by illustrating the HNA scheme in the context of a problem of scattering by a single sound-soft screen (equivalently, scattering by an aperture in a single sound-hard screen, see~\cite{HeLaCh:13} for details), then describing how this idea extends to scattering by sound-soft convex polygons, outlining the added difficulties that arise in the case that the obstacle is nonconvex (presenting sharper results in terms of $k$-dependence than those outlined in~\cite{ChGrLaSp:12}), and finally considering scattering by penetrable polygons (the transmission problem).  We demonstrate that HNA methods have the potential to solve scattering problems accurately in a computation time that is (almost) independent of frequency.

\subsubsection{Scattering by screens} \label{sec:screens}
To get the main ideas across, we first describe an HNA BEM for a simple 2D geometry, scattering by a single planar sound-soft screen.  This problem, indeed the more general problem of scattering by an arbitrary collinear array of such screens, has been treated by HNA BEM methods with a complete numerical analysis in~\cite{HeLaCh:13}.

To be precise then, we consider the 2D problem of scattering of the time harmonic incident plane wave~(\ref{eqn:plane}) 
by a sound soft screen
\[ \Gamma := \{(x_1,0)\in\R^2\,:\, 0<x_1<L\}, \]
where $L$ is the length of the screen and $\bx=(x_1,x_2)\in\R^2$, i.e.\ we solve the BVP~(\ref{prob:scsp}) with $d=2$.  An example of our scattering configuration, for $k=5$ and for $k=20$, is shown in Figure~\ref{fig1}.  The increased oscillations for $k=20$ can be clearly seen.
\begin{figure}[htbp]
\begin{minipage}{0.5\linewidth}
\centering
\includegraphics[scale=0.33]{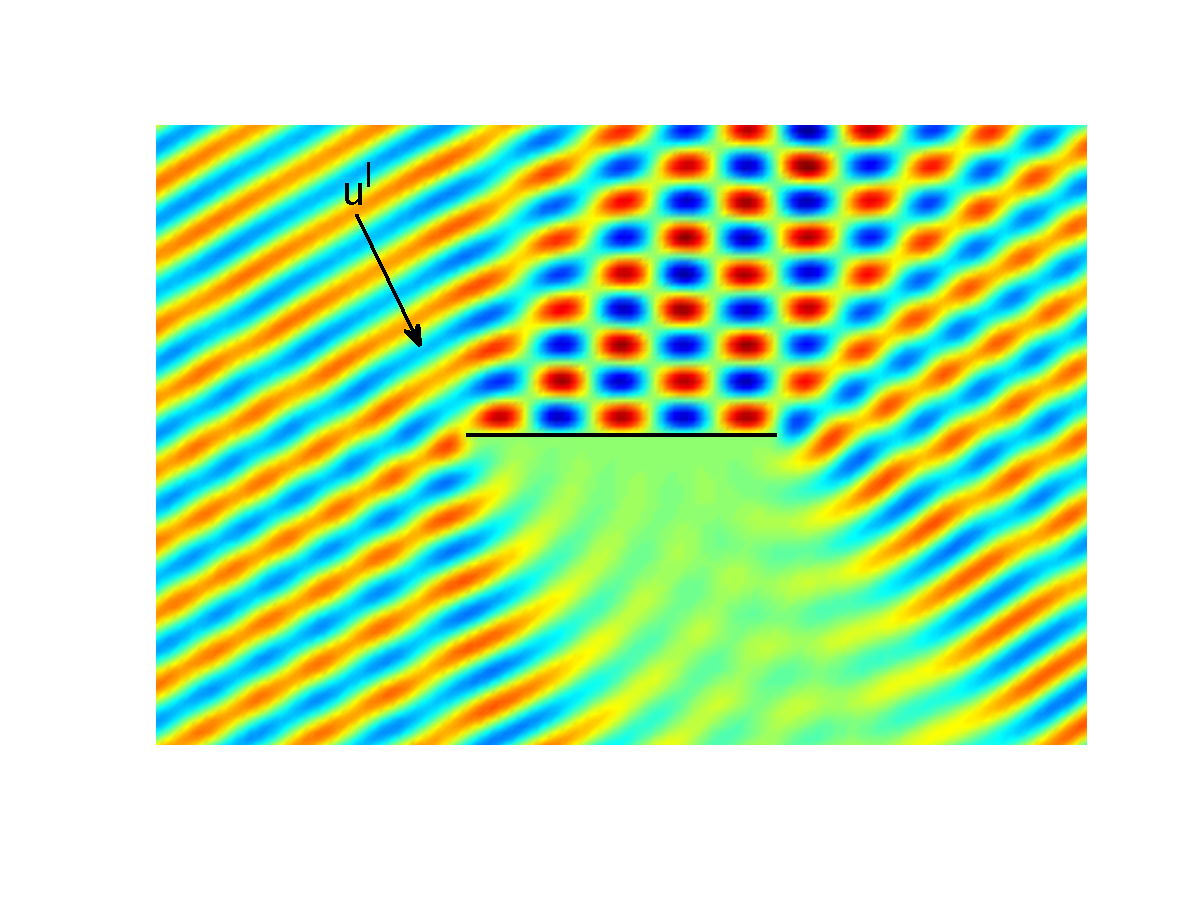}
\end{minipage}%
\begin{minipage}{0.5\linewidth}
\centering
\includegraphics[scale=0.33]{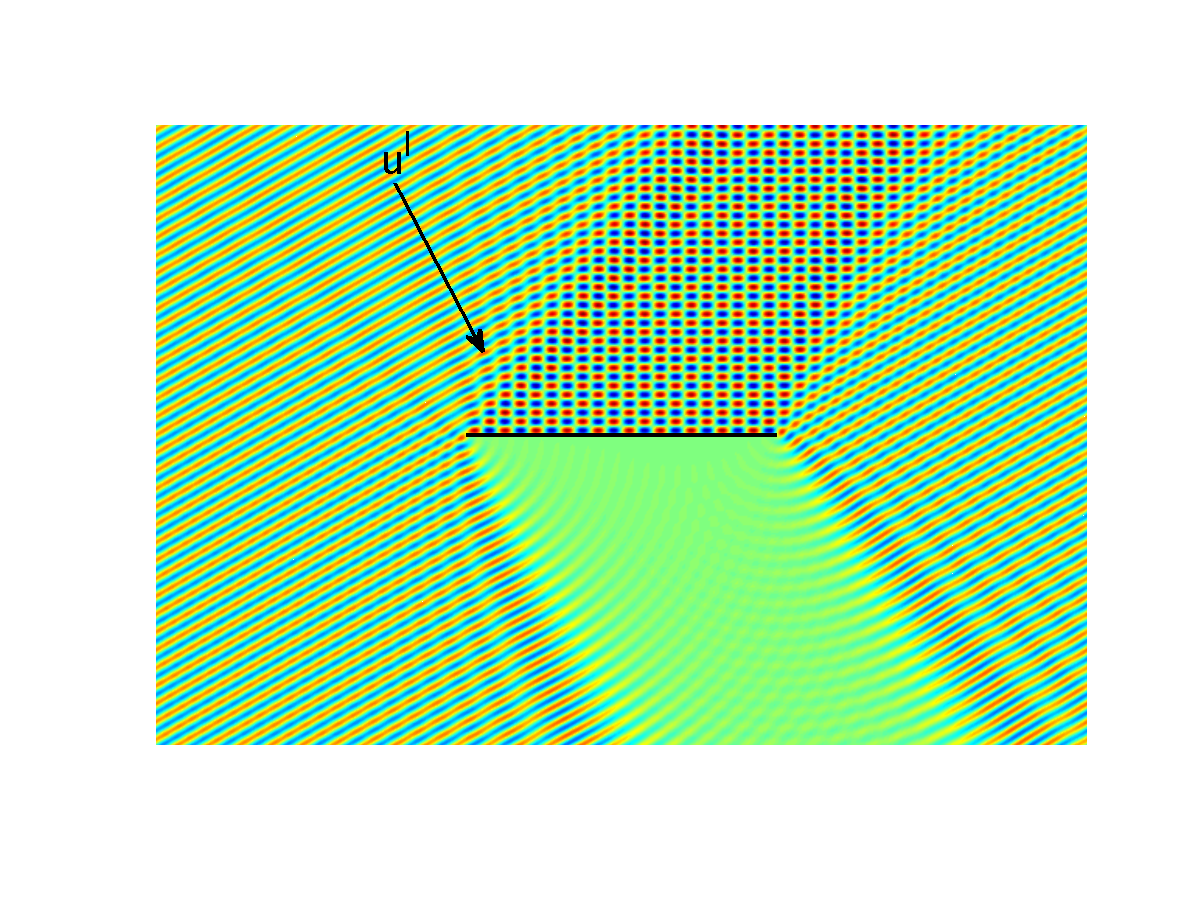}
\end{minipage}
\caption{Total field, scattering by a screen, $k=5$ (left) and $k=20$ (right)}
\label{fig1}
\end{figure}
The solution of this BVP satisfies~(\ref{eq:grt_screen_scat}), and hence our problem reduces to solving the BIE~(\ref{eq:BIE_sc2}).

The HNA method for solving~(\ref{eq:BIE_sc2}) uses an approximation space that is specially adapted to the high frequency asymptotic behaviour of the solution $[\partial_\nu u]$ on $\Gamma$, which we now consider.  Representing a point $\bx\in\Gamma$ parametrically by $\bx(s):=(s,0)$, where $s\in(0,L)$, the following theorem is proved in~\cite{HeLaCh:13} (this is derived directly from~(\ref{prob:scsp}) using an elementary representation for the solution in the half-plane above the screen in terms of a Dirichlet half-plane Green's function - for details see~\cite{HeLaCh:13} and cf.~\cite[Theorem~3.2 and Corollary~3.4]{Convex} and \cite[\S3]{HeLaMe:11}):
\begin{thm}
  \label{thm:31130812}  Suppose that $k\geq k_0>0$.  Then
  \begin{align}
    \left[\partial_\nu u\right]\left(\bx(s)\right)&=\Psi\left(\bx(s)\right)+v^+(s)\re^{\ri ks}+v^-\left(L-s\right)\re^{-\ri ks}, \qquad s\in\left(0,L\right), \label{eqn:27130812}
  \end{align}
  where $\Psi:=2\partial u^I/\partial\nu$, and the functions $v^\pm(s)$ are analytic in the right half-plane $\real{s}>0$, where they satisfy the bound
  \begin{align*}
    \left|v^\pm(s)\right|&\leq C_1 \mathcal{M} k\left|ks\right|^{-\frac{1}{2}}, 
  \end{align*}
  where
  \begin{align*}
    \mathcal{M}&:=\sup_{\bx\in D}\left|u(\bx)\right| \leq C(1+k),
  \end{align*}
  and the constants $C, C_1>0$ depend only on $k_0$ and $L$.
\end{thm}

The representation~(\ref{eqn:27130812}) is of the form~(\ref{eqn:ansatz}), with $V_0(\bx(s),k)=\Psi(\bx(s))$, $M=2$, $V_1(\bx(s),k)=v^+(s)$, $V_2(\bx(s),k)=v^-(s)$, $\psi_1(\bx(s))=s$, and $\psi_2(\bx(s))=-s$.  Using this representation we can design an appropriate approximation space $V_{N,k}$ to represent
\begin{align}
  \varphi(s)&:=\frac{1}{k}\left(\left[\partial_\nu u\right]\left(\bx(s)\right)-\Psi\left(\bx(s)\right)\right), \quad s\in(0,L).
  \label{eqn:varphi_screen}
\end{align}
Here $N$ denotes the total number of degrees of freedom in the method, and the subscript $k$ in $V_{N,k}$ serves as a reminder that the hybrid approximation space depends explicitly on the wavenumber $k$.  The function $\varphi$, which we seek to approximate, can be thought of as the scaled difference between $[\partial_\nu u]$ and its ``Physical Optics'' approximation $\Psi$ (which represents the direct contribution of the incident and reflected waves), with the $1/k$ scaling ensuring that $\varphi$ is nondimensional, and reflecting that $[\partial_\nu u]=O(k)$ as $k\to\infty$.  The second and third terms on the right hand side of~(\ref{eqn:27130812}) represent the diffracted rays emanating from the ends of the screen at $s=0$ and at $s=L$, respectively.  As alluded to earlier, instead of approximating $\varphi$ directly by conventional piecewise polynomials we instead use the representation (\ref{eqn:27130812}) with $v^+(s)$ and $v^-(L-s)$ replaced by piecewise polynomials supported on overlapping geometric meshes, graded towards the singularities at $s=0$ and $s=L$ respectively.  We proceed by describing our mesh, which is illustrated in Figure~\ref{fig:mesh}.
\begin{figure}[htbp]
  \begin{center}
    \includegraphics[height=3.5cm, width=9cm]{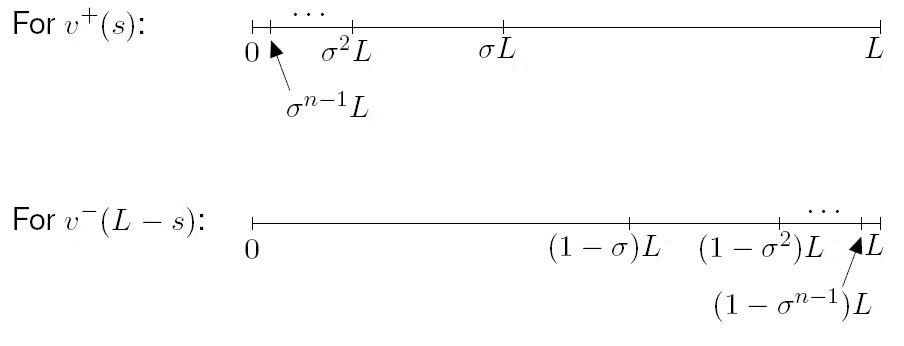}
  \end{center}
  \caption{Overlapping geometric meshes for approximation of $v^+$ and $v^-$}
  \label{fig:mesh}
\end{figure}

\begin{defn}
\label{def:51130812}
Given $L>0$ and an integer $n\geq1$ we denote by $\mathcal{G}_n(0,L)$ the geometric mesh on $[0,L]$ with $n$ layers, whose meshpoints $x_i$ are defined by
\begin{align*}
x_0:=0, \quad x_i:=\sigma^{n-i}L, \quad i=1,2,\ldots,n,
\end{align*}
where $0<\sigma<1$ is a grading parameter, with a smaller value providing a more severe mesh grading.  (We use $\sigma=0.15$ below.)
We further denote by $\mathcal{P}_{p,n}(0,L)$ the space of piecewise polynomials of order $p$ on the geometric mesh $\mathcal{G}_n(0,L)$, i.e.
\begin{align*}
\mathcal{P}_{p,n}(0,L)&:=\left\{\rho:[0,L]\to\C\,:\,\left.\rho\right|_{(x_{i-1},x_i)}\mbox{ is a polynomial of degree $\leq p$}\right\}.
\end{align*}
\end{defn}

We are now in a position to define the approximation space $V_{N,k}$.  Let
\begin{align*}
  V^+_{p,n}&:=\left\{\rho(s)\textrm{e}^{\ri ks}\,:\,\rho\in\mathcal{P}_{p,n}(0,L)\right\}, \\
  V^-_{p,n}&:=\left\{\rho\left(L-s\right)\textrm{e}^{-\ri ks}\,:\,\rho\in\mathcal{P}_{p,n}\left(0,L\right)\right\},
\end{align*}
with which we approximate, respectively, the terms $v^+(s)\re^{\ri ks}$ and $v^-\left(L-s\right)\re^{-\ri ks}$ in the representation~(\ref{eqn:27130812}).  We choose values for $p$ and $n$ and then define
\begin{equation}
  V_{N,k}:=\mbox{span}\left\{V^+_{p,n}\cup V^-_{p,n}\right\},
  \label{eqn:VNk}
\end{equation}
which has a total number of degrees of freedom $N=2n(p+1)$.
The following best approximation result is shown in~\cite[Theorem~5.1]{HeLaCh:13}.
\begin{thm}
\label{thm:57170912}
Let $n$ and $p$ satisfy $n\geq cp$ for some constant $c>0$ and suppose that $k\geq k_0>0$.  Then, there exist constants $C,\tau>0$, dependent only on $k_0$, $L$, and $c$, such that
\begin{align*}
  \inf_{w_N\in V_{N,k}}\left\|\varphi-w_N\right\|_{\tilde{H}^{-\frac{1}{2}}\left(\Gamma\right)}&\leq C k \re^{-p\tau}.
\end{align*}
\end{thm}
Identifying $\Gamma$ with $(0,L)$, here $\tilde{H}^{-1/2}(\Gamma)=\tilde{H}^{-1/2}(0,L)\subset H^{-1/2}(\R)$, $\tilde{H}^{-1/2}(0,L)$ just the subspace of those $\psi\in H^{-1/2}(\R)$ that have support in $[0,L]$. And then $\|\cdot\|_{\tilde{H}^{-1/2}(\Gamma)}$ is just the standard norm on the Sobolev space $H^{-1/2}(\R)$ (see, e.g., \cite{McLean}).\footnote{We note that the bounds in~\cite{HeLaCh:13} are in fact in terms of a natural wavenumber dependent Sobolev norm $\|\cdot\|_{\tilde{H}_k^{-1/2}(\Gamma)}$, but it is an easy calculation, this feeding into the results we state here, that $\|\cdot\|_{\tilde{H}_k^{-1/2}(\Gamma)} \leq \max(1,k^{1/2})  \|\cdot\|_{\tilde{H}^{-1/2}(\Gamma)}$.}

Having designed an appropriate approximation space $V_{N,k}$, we use a Galerkin method to select an element so as to efficiently approximate $\varphi$.  That is, we seek $\varphi_N\in V_{N,k}$ 
such that
\begin{align}
\left\langle S_k\varphi_N,w_N\right\rangle_{\Gamma}=\frac{1}{k}\left\langle u^I-S_k\Psi,w_N\right\rangle_{\Gamma}, \quad \mbox{for all } w_N\in V_{N,k}, \label{eqn:gal130812}
\end{align}
where 
the duality pairings in~(\ref{eqn:gal130812}) are simply $L^2(\Gamma)$ inner products.
The following error estimate follows from~\cite[Theorem~6.1]{HeLaCh:13}.
\begin{thm}
\label{GalerkinThm}
If the assumptions of Theorem~\ref{thm:57170912} hold, then there exist constants $C,\tau>0$, dependent only on $k_0$, $L$,  and $c$, such that
\[ \left\|\varphi-\varphi_N\right\|_{\tilde{H}^{-1/2}(\Gamma)}\leq C k^{3/2} \re^{-p\tau}. \]
\end{thm}
Note that, if $n$ is chosen proportional to $p$, precisely so that $c_2p\geq n\geq c_1p$, for some constants $c_2>c_1>0$, then the total number of degrees of freedom $N=O(p^2)$, and Theorem~\ref{GalerkinThm} implies that we can achieve any required accuracy with $N$ growing like $\log^2{k}$ as $k\to\infty$, rather than like $k$ as for a standard BEM.

We now present numerical results for the solution of (\ref{eqn:gal130812}).  The screen is of length $L=2\pi$ and hence our scatterer is~$k$ wavelengths long (recall that $\lambda=2\pi/k$).  The angle of incidence is $\pi/6$ measured anticlockwise from the downward vertical, as in Figure~\ref{fig1}.  In our experiments we choose $n=2(p+1)$, so that the total number of degrees of freedom is $N=2n(p+1)=4(p+1)^2$.  Since the total number of degrees of freedom depends only on $p$, we adjust our notation by defining $\psi_p(s):=\varphi_N(s)$, and we take the ``reference'' solution to be $\psi_7$.  In Figure~\ref{fig:screen_bdy} we plot $|\psi_7|\approx|\varphi|$, for $k=20$ and for $k=10240$.  The singularities at the edge of the screen can be clearly seen, as can the increased oscillations for larger~$k$ (the apparently shaded area is an artefact of the rapidly oscillating solution).
\begin{figure}[htbp]
\begin{minipage}{0.5\linewidth}
\centering
\includegraphics[scale=0.33]{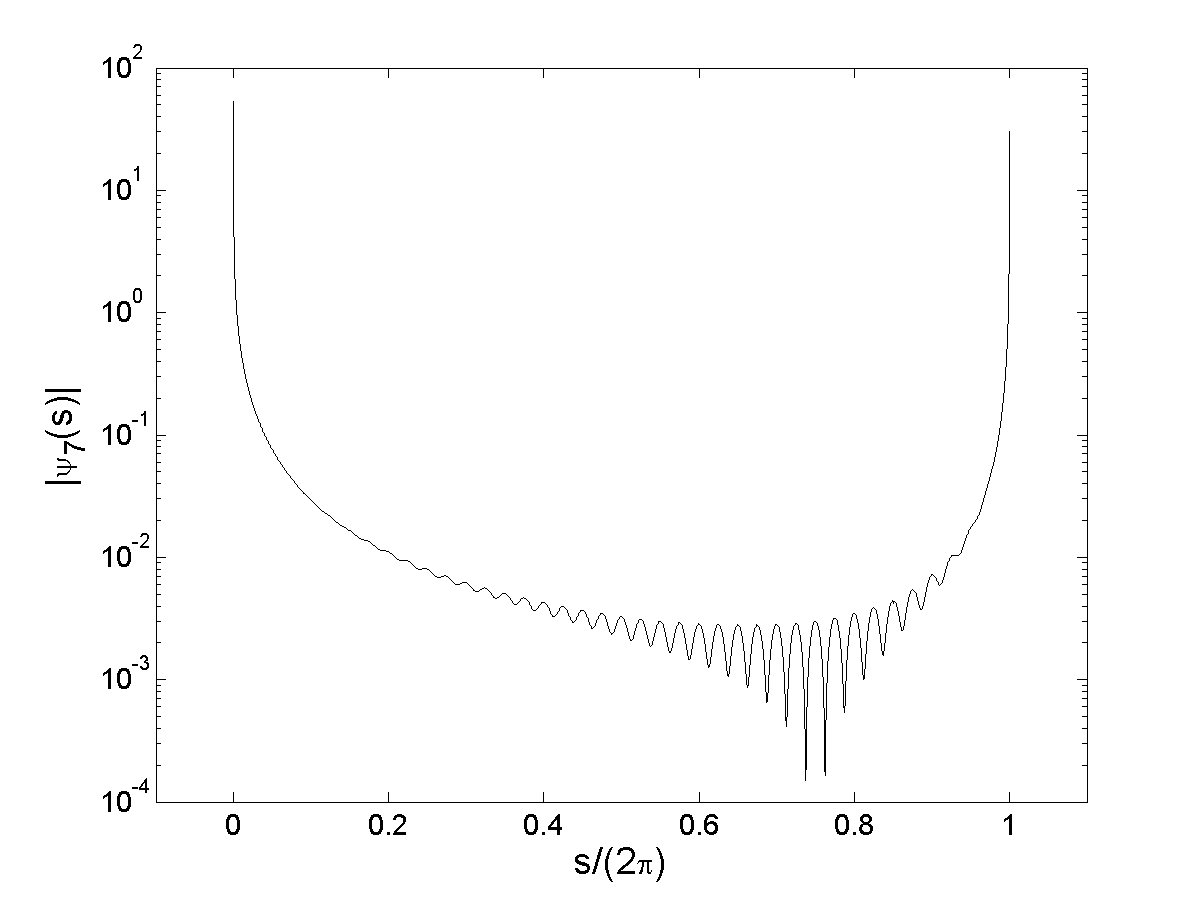}
\end{minipage}%
\begin{minipage}{0.5\linewidth}
\centering
\includegraphics[scale=0.33]{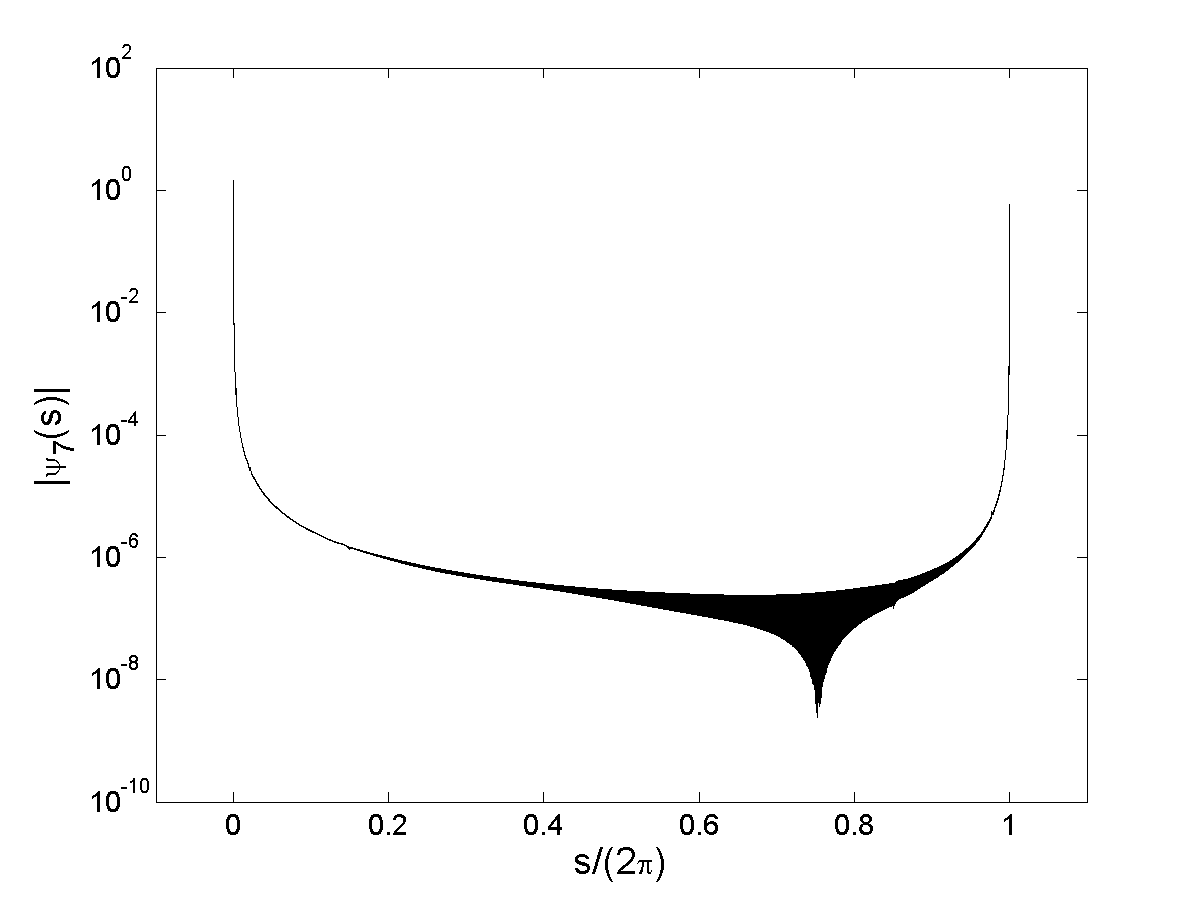}
\end{minipage}
\caption{The boundary solution $|\psi_7|\approx|\varphi|$, as given by~(\ref{eqn:varphi_screen}), for $k=20$ (left) and $k=10240$ (right), scattering by a screen}
\label{fig:screen_bdy}
\end{figure}

In Figure~\ref{fig2} we plot on a logarithmic scale the relative $L^1$ errors
\[ \frac{\left\|\psi_7-\psi_p\right\|_{L^1(\Gamma)}}{\left\|\psi_7+\Psi/k\right\|_{L^1(\Gamma)}}, \]
against $p$ for a range of $k$.  This is the relative error in $L^1(\Gamma)$ norm in our approximation, $\psi_p+\Psi/k$, to the dimensionless quantity $\frac{1}{k}[\partial_\nu u]$ (recall~(\ref{eqn:varphi_screen})).   (We calculate $L^1$ rather than $\tilde{H}^{-1/2}(\Gamma)$ norms for simplicity of computation.)  The $L^1$ norms are calculated using a high-order Gaussian quadrature routine on a mesh graded towards the endpoint singularities.
\begin{figure}[htbp]
 \centering
  \includegraphics[width=0.75\linewidth]{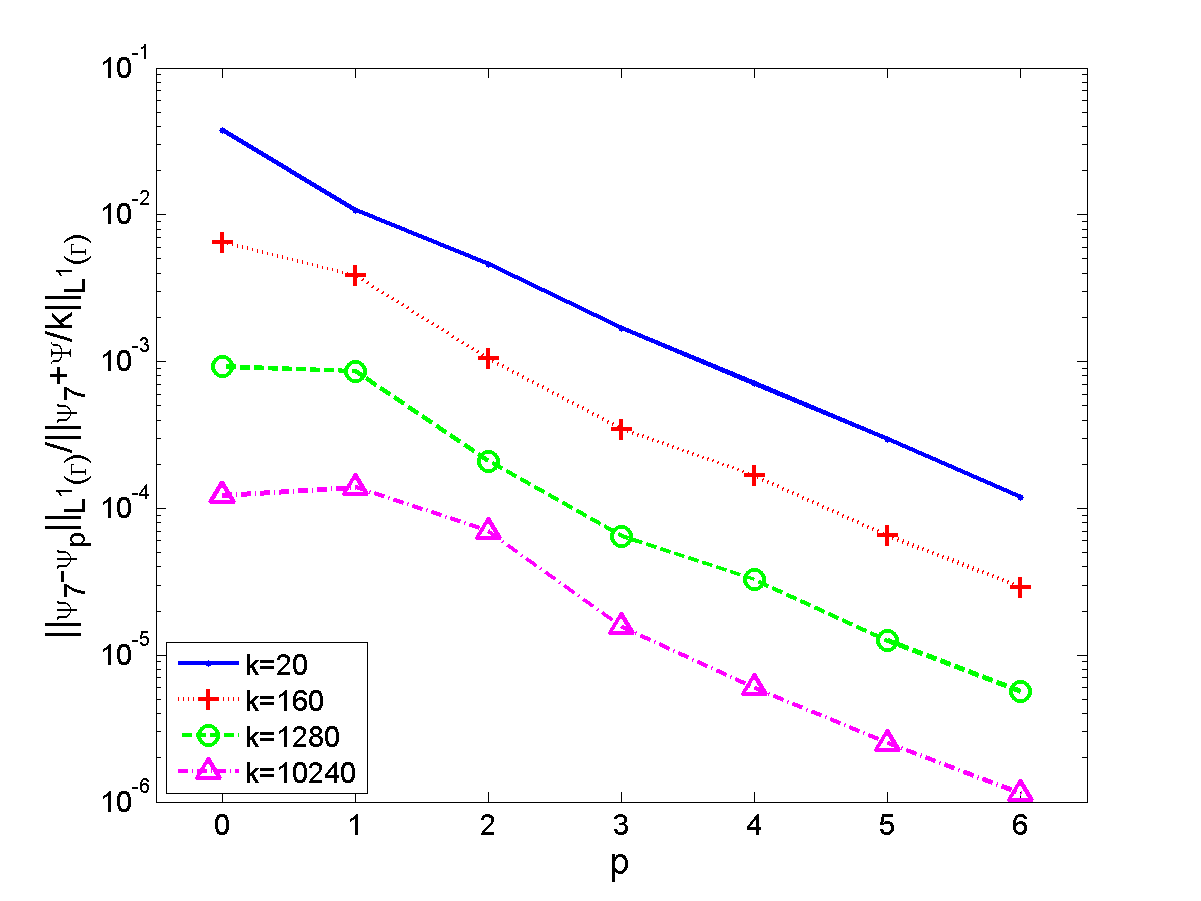}
 \caption{Relative errors in our approximation to $\frac{1}{k}[\partial_\nu u]$, scattering by a screen}
 \label{fig2}
\end{figure}
The linear plots demonstrate exponential decay as the polynomial degree, $p$, increases, as predicted in Theorem~\ref{GalerkinThm}.  More significantly, the relative error decreases as~$k$ increases for fixed~$p$.
This behaviour is better than might be expected from Theorem~\ref{GalerkinThm}, and in particular suggests that the error estimate in Theorem~\ref{GalerkinThm} is not sharp in its $k$-dependence.  Further results, including calculations of errors measured in an $\|\cdot\|_{\tilde{H}^{-1/2}(\Gamma)}$ norm, computations for an array of collinear screens, and computations of the solution in the domain and the far field pattern, can be found in~\cite{HeLaCh:13}.

\subsubsection{Scattering by convex polygons} The ideas outlined above for the screen problem can also be applied to the case of scattering by polygons.  First, we consider the 2D problem of scattering of the time harmonic incident plane wave~(\ref{eqn:plane}) by a sound-soft polygon with boundary $\Gamma$, i.e.\ problem~(\ref{prob:ssp}).
The solution $u$ then has the representation~(\ref{eq:grt_ss}), where $\partial_{\nu}^+ u$ satisfies~(\ref{eq:BIE_sss3}) and, if $\Omega_-$ is star-shaped,~(\ref{eq:BIE_sss4}).

We denote the number of sides of the polygon by $n_s$, and the corners (labelled in order counterclockwise) by $P_j$, $j=1,\ldots,n_s$.  We set $P_{n_s+1}:=P_1$, and then for $j=1,\ldots,n_s$ we denote the side of the polygon connecting the corners $P_j$ and $P_{j+1}$ by $\Gamma_j$, with length $L_j$.  We denote the exterior angle at the corner $P_j$ by $\omega_j$, $j=1,\ldots,n_s$.  We say that a side $\Gamma_j$ is illuminated by the incident plane wave given by~(\ref{eqn:plane}) if $\hat{\ba}\cdot\nu<0$ on $\Gamma_j$, and is in shadow if $\hat{\ba}\cdot\nu\geq 0$ on $\Gamma_j$.

As for the screen problem, the HNA method for solving~(\ref{eq:BIE_sss3}) or~(\ref{eq:BIE_sss4}) uses an approximation space that is specially adapted to the high frequency asymptotic behaviour of $\partial_{\nu}^+ u$ on $\Gamma$.  The nature of this behaviour is highly dependent on the geometry of the polygon: in particular, we consider convex and nonconvex polygons separately, considering first the case that the polygon is convex, and treating the nonconvex case in the next subsection.

Denoting the distance of $\bx\in\Gamma_j$ from $P_j$ by $s$, the following theorem follows from~\cite[Theorem~3.2, Theorem~4.3]{HeLaMe:11}.
\begin{thm}
  \label{thm:convex_reg}  Let $k\geq k_0>0$.  Then on any side $\Gamma_j$
  \begin{equation}
    \partial_{\nu}^+ u\left(\bx(s)\right) = \Psi\left(\bx(s)\right)+v_j^+(s)\re^{\ri ks}+v_j^-\left(L_j-s\right)\re^{-\ri ks}, \qquad \bx(s)\in\Gamma_j,
    \label{eqn:convex}
  \end{equation}
  where
  \[ \Psi:=\left\{\begin{array}{ll} 2\partial u^I/\partial\nu & \mbox{if }\Gamma_j\mbox{ is illuminated}, \\
                                                            0 & \mbox{if }\Gamma_j\mbox{ is in shadow}
           \end{array}\right.
  \]
  and the functions $v_j^\pm(s)$ are analytic in the right half-plane $\real{s}>0$, with
  \begin{equation}
    |v_j^{\pm}(s)|\leq C \left\{\begin{array}{ll} k^{3/2}\log^{1/2}(2+k)|ks|^{-\delta_j^{\pm}}, & 0<|s|\leq 1/k,\\
                                   k^{3/2}\log^{1/2}(2+k)|ks|^{-1/2}, & |s|> 1/k, \end{array}\right.
    \label{eqn:vjpmbound}
  \end{equation}
  where $\delta_j^{\pm}\in(0,1/2)$ are given by $\delta_j^+:=1-\pi/\omega_j$ and $\delta_j^-:=1-\pi/\omega_{j+1}$, and
  the constant $C>0$ depends only on $k_0$ and $\Gamma$.
\end{thm}

Using the representation~(\ref{eqn:convex}) on each side of the polygon leads to a representation of the form~(\ref{eqn:ansatz}) with $V_0(\bx(s),k)=\Psi(\bx(s))$, $M=2n_s$, $V_{2j-1}(\bx(s),k)=v_j^+(s)$ if $\bx(s)\in\Gamma_j$ and zero otherwise, $V_{2j}(\bx(s),k)=v_j^-(s)$ if $\bx(s)\in\Gamma_j$ and zero otherwise, $\psi_{2j-1}(\bx(s))=\bx(s)\cdot\hat{d}_j$, and $\psi_{2j}(\bx(s))=-\bx(s)\cdot\hat{d}_j$, $j=1,\ldots,n_s$, where $\hat{d}_j$ is a unit vector parallel to $\Gamma_j$.
Our approximation space $\tilde{V}_{N,k}$ for
\begin{equation}
  \varphi(s):=\frac{1}{k}\left(\partial_{\nu}^+ u \left(\bx(s)\right)-\Psi\left(\bx(s)\right)\right), \quad \bx(s)\in\Gamma.
  \label{eqn:varphi_convex}
\end{equation}
is then identical on each side of the polygon to that for the screen~(\ref{eqn:VNk}), in effect replacing $v_j^+(s)$ and $v_j^-(L_j-s)$ by piecewise polynomials supported on overlapping geometric meshes, graded towards the singularities at $s=0$ and $s=L_j$ respectively.  As for the screen, the function $\varphi$, which we seek to approximate, can be thought of as the scaled difference between $\partial_{\nu}^+ u$ and its ``Physical Optics'' approximation $\Psi$, which again represents the direct contribution of the incident and reflected waves (when they are present).  The second and third terms in~(\ref{eqn:convex}) represent the diffracted rays emanating from the corners $P_j$ and $P_{j+1}$, respectively.

The following best approximation estimate follows from \cite[Theorem~4.3, Theorem~5.5]{HeLaMe:11}.
\begin{thm}
  If $c,k_0>0$ and $n\geq cp$, $k\geq k_0$, then, for some $C,\tau>0$, depending only on $c$, $k_0$ and $\Gamma$,
  \[
    \inf_{w_N\in \tilde{V}_{N,k}} \left\| \partial_{\nu}^+ u - w_N \right\|_{L^2(\Gamma)}\leq Ck^{1/2+\alpha}\log^{1/2}(2+k) \re^{-p\tau},
  \]
  where $\alpha:=1-\min_{m=1,\ldots,n_s}(1-\pi/\omega_m) \in (1/2,1)$.
  \label{thm:convex_best}
\end{thm}

Having designed an appropriate approximation space $\tilde{V}_{N,k}$ we use the Galerkin method to select an element to approximate $\varphi$.  Since convex polygons are star-shaped, in this case we can use the integral equation formulation~(\ref{eq:BIE_sss4}), i.e. we seek $\varphi_N\in\tilde{V}_{N,k}$ such that
\begin{equation} \label{eq:BIE_sss4_discrete}
  \langle\scA_k \varphi_N,w_N\rangle_{\Gamma} = \frac{1}{k} \langle f_k - \Psi,w_N\rangle_{\Gamma}, \mbox{ for all } w_N \in \tilde{V}_{N,k}.
\end{equation}
Thanks to the coercivity of the integral operator $\scA_k$, we have the following error estimate (cf. \cite[Corollary~6.2]{HeLaMe:11}):
\begin{thm}
\label{thm:convex_approx}
  If the assumptions of Theorem~\ref{thm:convex_best} hold then there exist constants $C,\tau>0$, dependent only on $k_0$, $c$ and $\Gamma$, such that
  \[ \left\|\varphi-\varphi_N\right\|_{L^2(\Gamma)} \leq C k^{\alpha}\log^{1/2}(2+k) \re^{-p\tau}. \]
\end{thm}

To compute the solution in the domain, we rearrange~(\ref{eqn:varphi_convex}) to get
\begin{equation}
  \partial_{\nu}^+ u\left(\bx(s)\right) = k \varphi(s)+ \Psi\left(\bx(s)\right) \approx k \varphi_N(s)+ \Psi\left(\bx(s)\right), \quad \bx(s)\in\Gamma,
  \label{eqn:dudnapprox}
\end{equation}
and then we insert this approximation to $\partial_{\nu}^+ u$ into the representation formula~(\ref{eq:grt_ss}) to get an approximation to $u$, which we denote by $u_N$.  We then have the following error estimate (cf. \cite[Theorem~6.3]{HeLaMe:11}, \cite[Corollary~64]{NonConvex}):
\begin{thm}
\label{thm:convex_uapprox}
  If the assumptions of Theorem~\ref{thm:convex_best} hold then there exist constants $C,\tau>0$, dependent only on $k_0$, $c$ and $\Gamma$, such that
  \[ \frac{\left\|u-u_N\right\|_{L^\infty(\Omega_+)}}{\left\|u\right\|_{L^\infty(\Omega_+)}} \leq C k\log(2+k) \re^{-p\tau}. \]
\end{thm}

Similarly, we can derive an approximation to the far field pattern (FFP) of the scattered field, given explicitly for $\hat{\bx}=\bx/|\bx|$ by
\begin{equation}
  F(\hat{\bx}) = - \int_\Gamma \re^{-\ri k \hat{\bx}\cdot \by} \partial_{\nu}^+ u(\by) \, \rd s(\by), \quad \hat{\bx}\in\mathbb{S}^1,
  \label{eqn:FFP}
\end{equation}
where $\mathbb{S}^1$ denotes the unit circle.  Efficient computation of the far field patten is of interest in many applications, see, e.g., \cite{CoKr:92}.
To compute an approximation $F_N$ to $F$, we again just insert the approximation~(\ref{eqn:dudnapprox}) into the integral~(\ref{eqn:FFP}).  We then have the following estimate (cf. \cite[Theorem~6.4]{HeLaMe:11}, \cite[Corollary~64]{NonConvex}):
\begin{thm}
\label{thm:convex_Fapprox}
  If the assumptions of Theorem~\ref{thm:convex_best} hold then there exist constants $C,\tau>0$, dependent only on $k_0$, $c$ and $\Gamma$, such that
  \[ \left\|F-F_N\right\|_{L^\infty(\mathbb{S}^1)} \leq C k^{1+\alpha}\log^{1/2}(2+k) \re^{-p\tau}. \]
\end{thm}
\noindent Note that the estimates above for the solution in the domain and the FFP follow from results in~\cite{NonConvex} and are actually a little sharper than those in~\cite{HeLaMe:11}.

The algebraically $k$-dependent prefactors in the error estimates of Theorems~\ref{thm:convex_approx}, \ref{thm:convex_uapprox} and~\ref{thm:convex_Fapprox} can be absorbed into the exponentially decaying factors by allowing $p$ to grow modestly ($O(\log^2 k)$) with increasing $k$.  In practice, numerical results \cite{HeLaMe:11,ChGrLaSp:12} suggest that this is pessimistic, and that in many cases a fixed accuracy of approximation can be achieved without any requirement for the number of degrees of freedom to increase with~$k$.


To illustrate the approach described above, we present numerical results for the problem of scattering by a sound soft equilateral triangle, of side length $2\pi$, so that the number of wavelengths per side is equal to~$k$.  The total field for $k=10$ is plotted in Figure~\ref{fig:triangle}
\begin{figure}[ht]
\begin{center}
  \includegraphics[width=0.75\linewidth]{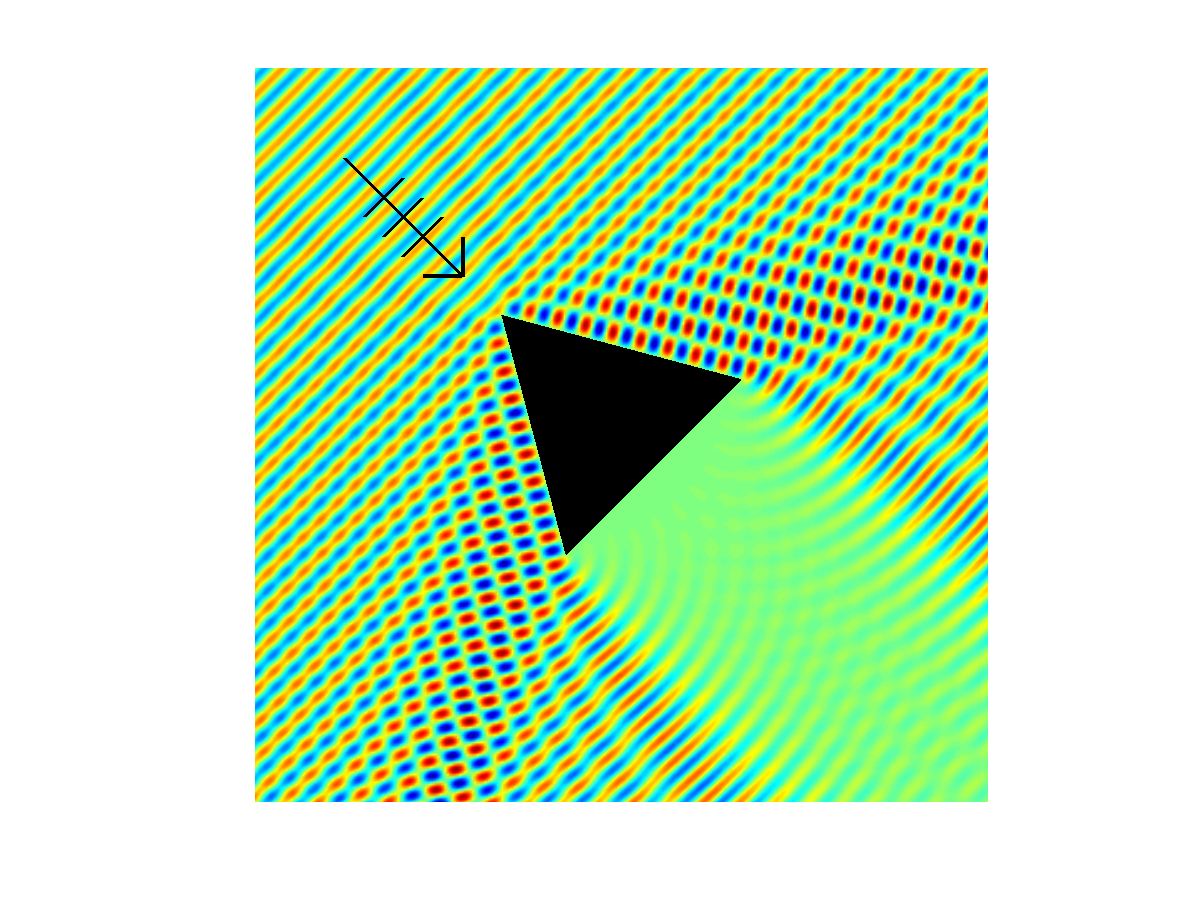}
\end{center}
\caption{Total field, scattering by a triangle}
\label{fig:triangle}
\end{figure}

In our computations we choose $n=2(p+1)$, as for the screen results above, so that the total number of degrees of freedom is $N=6n(p+1)=12(p+1)^2$.  Since the total number of degrees of freedom depends only on $p$, we again adjust our notation by defining $\psi_p(s):=\varphi_N(s)$.
In Figure~\ref{fig:tri_errors} we plot on a logarithmic scale the relative $L^2$ errors
\begin{align*}
\frac{\left\|\psi_6-\psi_p\right\|_{L^2\left(\Gamma\right)}}{\left\| \frac{1}{k}\partial_{\nu}^+ u \right\|_{L^2\left(\Gamma\right)}},
\end{align*}
against $p$ for a range of values of $k$, this quantity an estimate of the relative error in our approximation \eqref{eqn:dudnapprox} to $\partial_{\nu}^+ u$.  (Again we take the ``reference'' solution, our approximation to the true solution $\varphi$, to be $\psi_6$.)  This example is identical to one that appears in~\cite{HeLaMe:11}, except that here we show results for much higher values of $k$ (in~\cite{HeLaMe:11} the largest value of $k$ tested was $k=5120$), and we here plot relative errors in our approximation to $\partial_{\nu}^+ u$, rather than the relative errors in the approximation $\psi_p$ to $\varphi$ computed in \cite{HeLaMe:11}.  The $L^2$ norms are calculated using a high-order Gaussian quadrature routine on a mesh graded towards the endpoint singularities; see~\cite{HeLaMe:11} for details.
\begin{figure}[ht]
\begin{center}
  \includegraphics[width=0.75\linewidth]{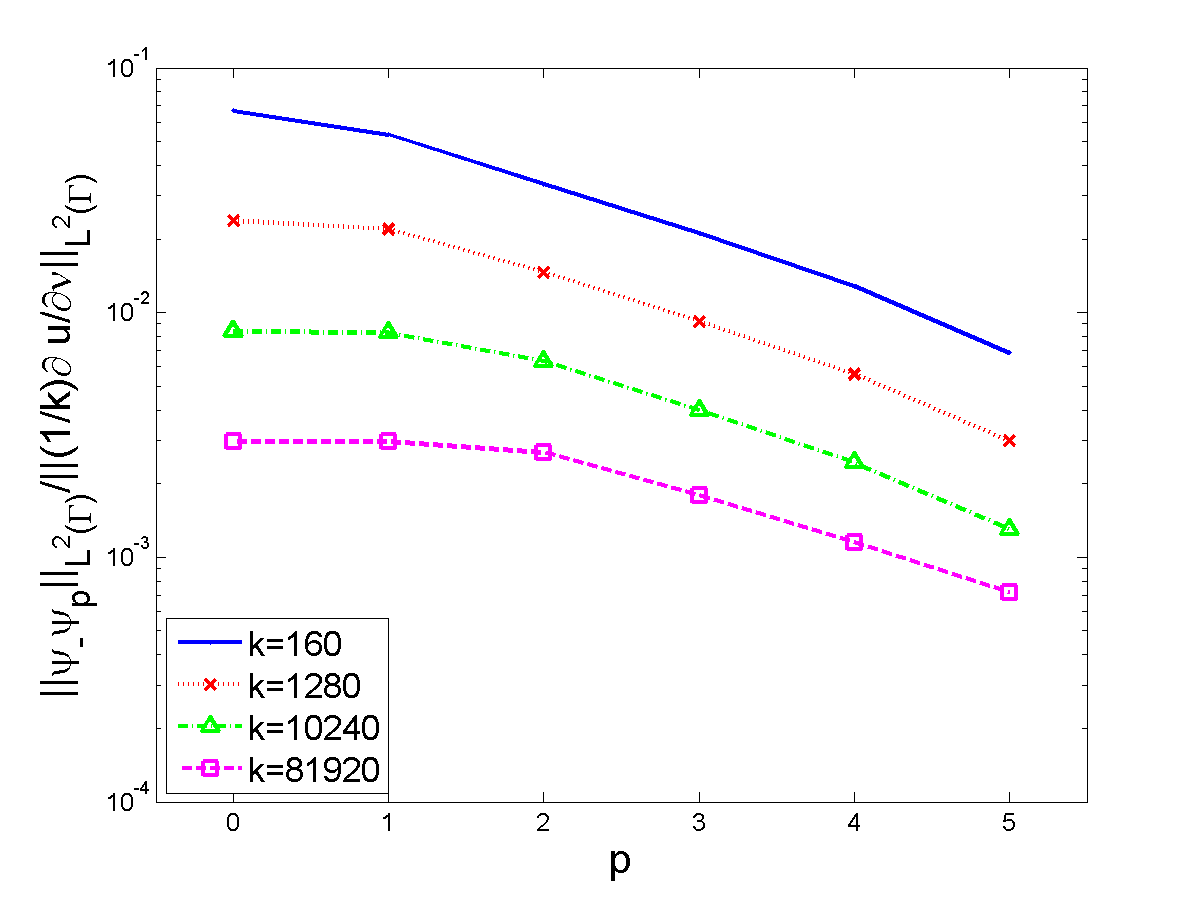}
\end{center}
\caption{Relative $L^2$ errors in $\frac{1}{k}\partial_\nu^+ u $}
\label{fig:tri_errors}
\end{figure}
For fixed $k$, as $p$ increases the error decays exponentially, as predicted by Theorem~\ref{thm:convex_approx}.  For fixed $p$, the error seems to decrease as $k$ increases.  To investigate this further, in Table~\ref{table:tri_errors} we show the relative $L^2$ errors (computed exactly as in Figure~\ref{fig:tri_errors}) as computed with fixed $p=3$ (and hence $N=192$), for a wider range of values of $k$.  The relative error decreases as~$k$ increases, for a fixed number of degrees of freedom, suggesting that the $k$-dependence of our error estimate in Theorem~\ref{thm:convex_approx} is pessimistic.
\begin{table}[htbp]
  \begin{center}
    \begin{tabular}{|r|c|c|c|c|}
       \hline
      $k$ & $\frac{N}{L/\lambda}$  & $ \|\psi_7-\psi_4\|_{L^2(\Gamma)}/\|\frac{1}{k}\partial_\nu^+ u \|_{L^2(\Gamma)}$ & COND & cpt(s) \\
      \hline
    5 &  1.28$\times10^{+1}$ &    7.91$\times10^{-2}$ &    5.36$\times10^{+1}$ &    4.68$\times10^{+2}$ \\
   10 &  6.40$\times10^{+0}$ &    6.18$\times10^{-2}$ &    2.38$\times10^{+1}$ &    3.84$\times10^{+2}$ \\
   20 &  3.20$\times10^{+0}$ &    4.76$\times10^{-2}$ &    2.97$\times10^{+1}$ &    3.40$\times10^{+2}$ \\
   40 &  1.60$\times10^{+0}$ &    3.64$\times10^{-2}$ &    3.85$\times10^{+1}$ &    3.72$\times10^{+2}$ \\
   80 &  8.00$\times10^{-1}$ &    2.77$\times10^{-2}$ &    5.08$\times10^{+1}$ &    3.55$\times10^{+2}$ \\
  160 &  4.00$\times10^{-1}$ &    2.10$\times10^{-2}$ &    6.76$\times10^{+1}$ &    4.31$\times10^{+2}$ \\
  320 &  2.00$\times10^{-1}$ &    1.60$\times10^{-2}$ &    9.00$\times10^{+1}$ &    3.57$\times10^{+2}$ \\
  640 &  1.00$\times10^{-1}$ &    1.21$\times10^{-2}$ &    1.20$\times10^{+2}$ &    4.93$\times10^{+2}$ \\
 1280 &  5.00$\times10^{-2}$ &    9.18$\times10^{-3}$ &    1.60$\times10^{+2}$ &    5.33$\times10^{+2}$ \\
 2560 &  2.50$\times10^{-2}$ &    6.96$\times10^{-3}$ &    2.13$\times10^{+2}$ &    5.23$\times10^{+2}$ \\
 5120 &  1.25$\times10^{-2}$ &    5.27$\times10^{-3}$ &    2.81$\times10^{+2}$ &    5.44$\times10^{+2}$ \\
10240 &  6.25$\times10^{-3}$ &    3.99$\times10^{-3}$ &    3.67$\times10^{+2}$ &    6.40$\times10^{+2}$ \\
20480 &  3.13$\times10^{-3}$ &    3.03$\times10^{-3}$ &    4.75$\times10^{+2}$ &    6.12$\times10^{+2}$ \\
40960 &  1.56$\times10^{-3}$ &    2.29$\times10^{-3}$ &    6.04$\times10^{+2}$ &    6.34$\times10^{+2}$ \\
81920 &  7.81$\times10^{-4}$ &    1.79$\times10^{-3}$ &    7.48$\times10^{+2}$ &    7.83$\times10^{+2}$ \\
   \hline
    \end{tabular}
  \end{center}
  \caption{Relative $L^2$ errors for scattering by a triangle, fixed $p=3$ (and hence $N=192$), various $k$.}
  \label{table:tri_errors}
\end{table}

We also show $N/(L/\lambda)$, the average number of degrees of freedom per wavelength, the condition number (COND) of the linear system that we solve (details of our implementation can be found in~\cite{HeLaMe:11}), and the computing time (cpt) in seconds to set up and solve our linear system.  All computations were carried out in Matlab, using a standard desktop PC.  We surmise that it might be possible to reduce these computing times with some effort to optimise the code.  The key point of these last two columns though is that both the condition number and computing time grow only very slowly as~$k$ increases, for fixed $p$.  Implementation of our scheme includes the need to evaluate oscillatory integrals; details of the approach we use for that can be found in~\cite[\S4]{ChGrLaSp:12}.  Whereas for standard boundary element methods standard practice suggests that ten degrees of freedom per wavelength are required for ``engineering accuracy'', the results in Table~\ref{table:tri_errors} suggest that $\frac{1}{k}\partial_\nu^+ u $ can be computed with much lower than 1\% relative error with a fixed number of degrees of freedom per wavelength (indeed, over 1000 wavelengths per degree of freedom for the case $k=81920$) with low condition numbers for the linear system and a total computing time of the order of ten minutes or so (and increasing only slowly as $k$ increases).


\subsubsection{Scattering by nonconvex polygons} For nonconvex polygons, we encounter behaviour that does not occur for convex polygons, as a result of which the leading-order asymptotic behaviour on $\Gamma$ is more complicated.  In particular, in this case we may see partial illumination of a side of the nonconvex polygon (whereas for a convex polygon a side is either completely illuminated or completely in shadow) and/or rereflections (where a wave that has been reflected from one side of the polygon may be incident on another side of the polygon), as illustrated in Figure~\ref{fig:nc_phenom}.
\begin{figure}[htbp]
\begin{minipage}{0.5\linewidth}
\centering
\includegraphics[scale=0.25]{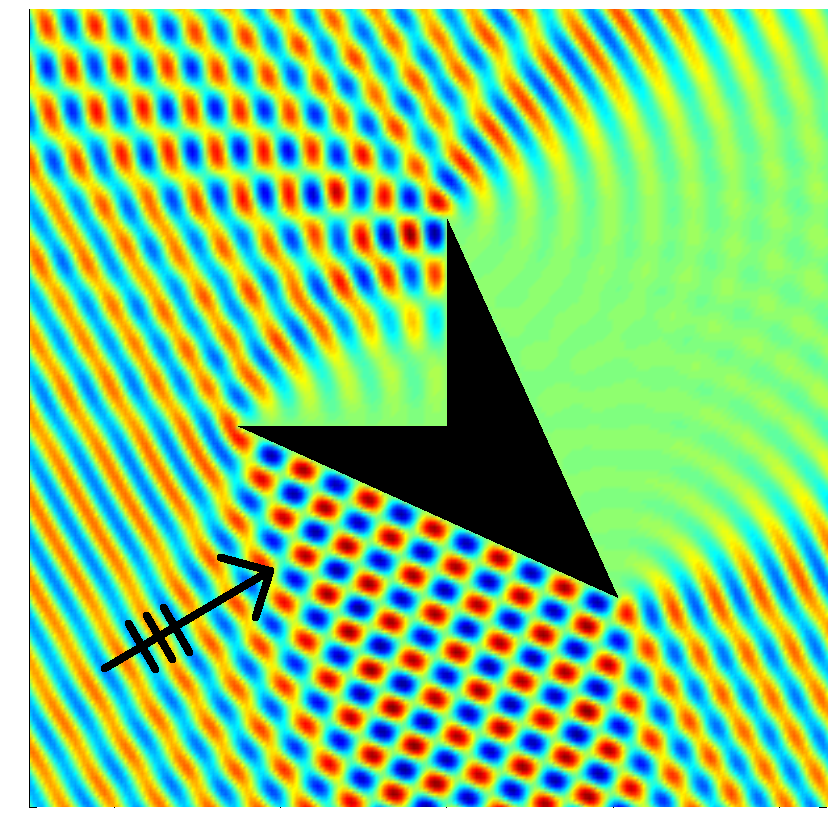}
\end{minipage}%
\begin{minipage}{0.5\linewidth}
\centering
\includegraphics[scale=0.35]{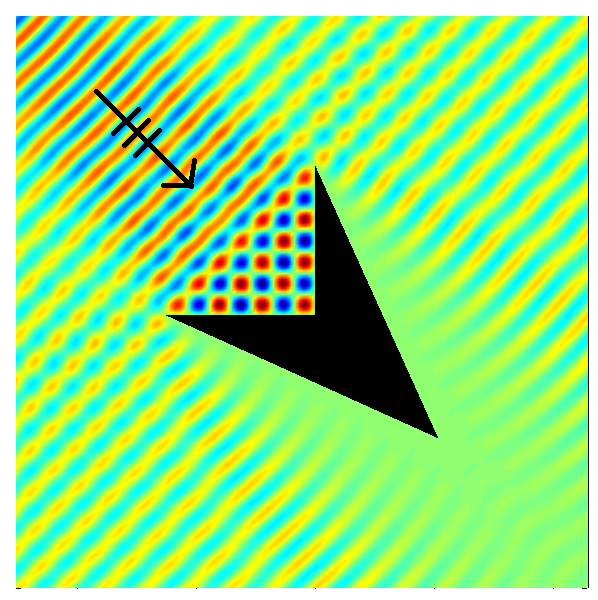}
\end{minipage}
\caption{Partial illumination (left) and rereflections (right)}
\label{fig:nc_phenom}
\end{figure}

We restrict attention to a particular class of nonconvex polygons that satisfy the following assumptions (a description of how the approach described below can be extended to polygons that do not satisfy these assumptions can be found in \cite[\S8]{NonConvex}):
\begin{ass}
  Each exterior angle $\omega_j$, $j=1,\ldots,n_s$, is either a right angle or greater than~$\pi$.
  \label{ass:1}
\end{ass}
\begin{ass}
  At each right angle, the obstacle lies within the dashed lines shown in Figure~\ref{fig:right_angle}.
  \label{ass:2}
\end{ass}
\begin{figure}[ht]
\begin{center}
  \includegraphics[width=0.75\linewidth]{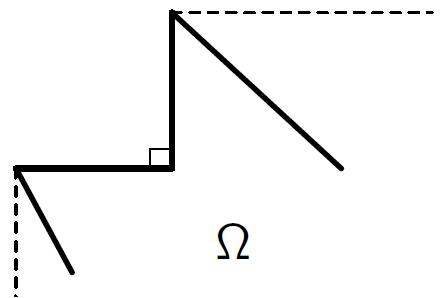}
\end{center}
\caption{Assumption~\ref{ass:2} on geometry of nonconvex polygon is that it lies entirely within the semi-infinite dashed lines}
\label{fig:right_angle}
\end{figure}
Polygons satisfying these criteria may or may not be star-shaped.  For each side $\Gamma_j$, $j=1,\ldots,n_s$, if either $\omega_j$ or $\omega_{j+1}$ is a right angle then we define that side to be a ``nonconvex'' side, otherwise we say it is a ``convex'' side, as illustrated for a particular non-star-shaped example in Figure~\ref{fig:notation_nc}.
\begin{figure}[ht]
\begin{center}
  \includegraphics[width=0.75\linewidth]{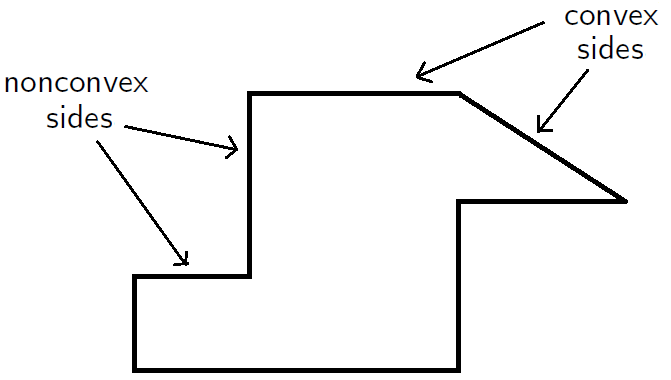}
\end{center}
\caption{Convex and nonconvex sides, for a non-star-shaped example}
\label{fig:notation_nc}
\end{figure}
On convex sides, $\partial_\nu^+u$ behaves exactly as in the convex case, and the approximation results above hold.  However, on nonconvex sides we need to consider the possibilites of partial illumination and/or rereflections.  To illustrate our approach, we consider the behaviour at a point $\bx(s)$ on a nonconvex side $\Gamma_j$, distance $s$ from $P_j$ and $r$ from $P_{j-1}$, as illustrated in Figure~\ref{fig:geometry_nc}.
\begin{figure}[ht]
\begin{center}
  \includegraphics[width=0.75\linewidth]{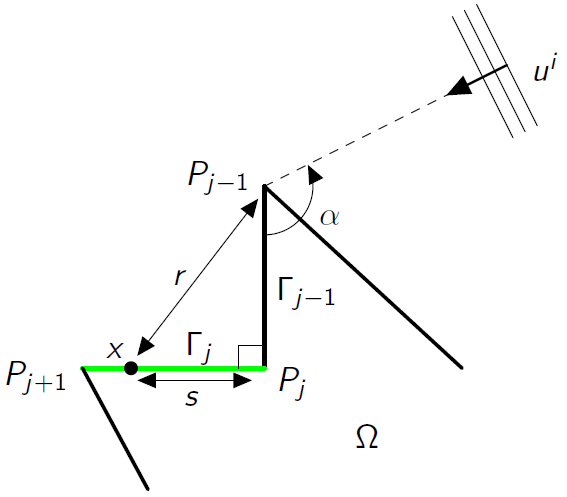}
\end{center}
\caption{Geometry of a nonconvex side $\Gamma_j$}
\label{fig:geometry_nc}
\end{figure}
Then $\Gamma_j$ will be fully illuminated if $\pi\leq\alpha<3\pi/2$ (where $\alpha$ is the incident angle shown in Figure~\ref{fig:geometry_nc}),
$\Gamma_j$ will be partially illuminated for some values of $\alpha$ in the range $\pi/2<\alpha<\pi$ (e.g., in the case that $L_j=L_{j-1}$, $\Gamma_j$ will be partially illuminated for $3\pi/4<\alpha<\pi$), and $\Gamma_j$ will be in shadow otherwise.  There will be reflections from $\Gamma_{j-1}$ onto $\Gamma_j$ if $\pi<\alpha<3\pi/2$.  Whatever the value of $\alpha$, there will be diffraction from $P_{j-1}$ and $P_{j+1}$ (either directly from the incident wave, or from waves that have travelled around $\Gamma$).

For $\bx(s)\in\Gamma_j$, where $s$ denotes distance from $P_j$ as in Figure~\ref{fig:geometry_nc}, the following theorem follows from~\cite[Theorem~36]{NonConvex}.
\begin{thm}
  \label{thm:nonconvex_reg}  Let $k\geq k_0>0$.  Then, for $\bx(s)\in\Gamma_j$,
  \begin{equation}
    \partial_\nu^+ u\left(\bx(s)\right) = \Psi^d\left(\bx(s)\right)+v_j^+(L_j+s)\re^{\ri ks}+v_j^-\left(L_j-s\right)\re^{-\ri ks}+\tilde{v}_j(s)\re^{\ri kr},
    \label{eqn:nonconvex}
  \end{equation}
  where
  \[ \Psi^d:=\left\{\begin{array}{ll} 2\partial u^d/\partial\nu & \mbox{if } \frac{\pi}{2}\leq\alpha\leq\frac{3\pi}{2}, \\
                                                            0 & \mbox{otherwise},
           \end{array}\right.
  \]
  the functions $v_j^\pm(t)$ have the same properties as those for the convex sides, in particular are analytic in the right half-plane $\real{t}>0$ and satisfy the bounds~(\ref{eqn:vjpmbound}), and the function $\tilde{v}_j$ is analytic in a complex $k$-independent neighbourhood $D_\epsilon$ of the side $\Gamma_j$ with
  \begin{equation}
    \left|\tilde{v}_j(t)\right| \leq Ck\log^{1/2} (2+k), \quad t\in D_\epsilon,
    \label{eqn:tildev_bound}
  \end{equation}
  where the constant $C>0$ depends only on $k_0$ and $\Gamma$.
\end{thm}
Here $u^d$ is the known solution of a canonical diffraction problem, namely that of scattering by a semi-infinite ``knife edge'', precisely scattering of $u^I$ by the semi-infinite sound-soft screen starting at $P_{j-1}$ and extending vertically down through $P_j$; for details, see \cite[Lemma~35]{NonConvex}.

Using the representation~(\ref{eqn:nonconvex}) on each side of the polygon again leads to a representation of the form~(\ref{eqn:ansatz}).
Our approximation space $\hat{V}_{N,k}$ for
\begin{equation}
  \varphi(s):=\frac{1}{k}\left(\partial_\nu^+ u \left(\bx(s)\right)-\Psi^d\left(\bx(s)\right)\right), \quad \bx(s)\in\Gamma,
  \label{eqn:varphi_nonconvex}
\end{equation}
is then identical on each convex side of the polygon to that for the convex polygon and the screen~(\ref{eqn:VNk}), again replacing $v_j^+(s)$ and $v_j^-(L_j-s)$ by piecewise polynomials supported on overlapping geometric meshes, graded towards the singularities at $s=0$ and $s=L_j$ respectively.  However, on nonconvex sides it is slightly different.  In this case, we note first that in the representation~(\ref{eqn:nonconvex}) we have $v_j^+(L_j+s)$ rather than $v_j^+(s)$ (compare with~(\ref{eqn:convex})); hence $v_j^+$ is not singular on $\Gamma_j$ when it is nonconvex, and we approximate it by a polynomial (of order $p$) supported on the whole side $\Gamma_j$.  Secondly, we notice from the analyticity of $\tilde{v}_j$ and the bound~(\ref{eqn:tildev_bound}) that it is also sufficient to approximate $\tilde{v}_j$ by a polynomial (of order $p$) supported on the whole side~$\Gamma_j$.

The following best approximation estimate follows from \cite[Theorem~56]{NonConvex}.
\begin{thm}
  If $c,k_0>0$ and $n\geq cp$, $k\geq k_0$, then, for some $C,\tau>0$, depending only on $c$, $k_0$, and $\Gamma$,
  \begin{equation}
    \inf_{w_N\in \hat{V}_{N,k}} \left\| \partial_\nu^+ u - w_N \right\|_{L^2(\Gamma)}\leq Ck^{1/2+\alpha}\log^{1/2}(2+k) \re^{-p\tau},
    \label{eqn:nonconvex_best}
  \end{equation}
  where $\alpha:=1-\min_{m=1,\ldots,n_s, \, \omega_m\neq\pi/2}(1-\pi/\omega_m) \in (1/2,1)$.
  \label{thm:nonconvex_best}
\end{thm}
Comparing Theorems~\ref{thm:convex_best} and~\ref{thm:nonconvex_best} we see that the best approximation estimate is identical for convex and nonconvex polygons, again implying that we can achieve any required accuracy with $N$ growing like $\log^2{k}$ as $k\to\infty$, rather than like $k$ as for a standard BEM.  Note also that the result~(\ref{eqn:nonconvex_best}) is sharper (in terms of $k$-dependence) than that in~\cite[Theorem~3.13]{ChGrLaSp:12}.

Having designed an appropriate approximation space $\hat{V}_{N,k}$ we again use the Galerkin method to select an element so as to effectively approximate $\varphi$.  For star-shaped polygons satisfying Assumptions~\ref{ass:1}--\ref{ass:2}, we can again use the integral equation formulation~(\ref{eq:BIE_sss4}), i.e. we seek $\varphi_N\in\hat{V}_{N,k}$ such that
\begin{equation} \label{eq:BIE_sss4_discrete_nc}
  \langle\scA_k \varphi_N,w_N\rangle_{\Gamma} = \frac{1}{k} \langle f_k - \Psi^d,w_N\rangle_{\Gamma}, \quad \mbox{for all } w_N \in \hat{V}_{N,k},
\end{equation}
in which case, due to the coercivity of the integral operator $\scA_k$, we have the following error estimates (cf. \cite[Theorems~61--63]{NonConvex}):
\begin{thm}
\label{thm:nonconvex_approx}
  For star-shaped polygons (satisfying also Assumptions~\ref{ass:1}--\ref{ass:2}), if the assumptions of Theorem~\ref{thm:nonconvex_best} hold then there exist constants $C,\tau>0$, dependent only on $k_0$, $c$ and $\Gamma$, such that
  \begin{eqnarray*}
    \left\|\varphi-\varphi_N\right\|_{L^2(\Gamma)} & \leq & C k^{\alpha}\log^{1/2}(2+k) \re^{-p\tau}, \\
    \displaystyle{\frac{\left\|u-u_N\right\|_{L^\infty(\Omega_+)}}{\left\|u\right\|_{L^\infty(\Omega_+)}}} & \leq & C k\log(2+k) \re^{-p\tau}, \\
    \left\|F-F_N\right\|_{L^\infty(\mathbb{S}^1)} & \leq & C k^{1+\alpha}\log^{1/2}(2+k) \re^{-p\tau}.
  \end{eqnarray*}
\end{thm}
Our approximations $u_N$ and $F_N$ to the solution in the domain and the far field pattern are constructed exactly as for the convex case, as outlined above.  As for the convex polygonal case, our conclusion is again that $N$ proportional to $p^2$ growing like $\log^2 k$ as $k$ increases provably maintains accuracy.

We present numerical results for the two cases shown in Figure~\ref{fig:nc_phenom}.  These examples have also been studied in~\cite{NonConvex}, but our results here are new, as detailed below.  The nonconvex sides of the scatterer have length $2\pi$ and the convex sides have length $4\pi$, so the total length of the boundary is $12\pi$, which is $6k$ wavelengths (recalling that the wavelength is $\lambda=2\pi/k$).  In our computations we again choose $n=2(p+1)$, so that for this example the total number of degrees of freedom is $N=12p^2+28p+16$.  Since $N$ depends only on $p$, we again adjust our notation by defining $\psi_p(s):=\varphi_N(s)$, and we take our ``reference'' solution to be $\psi_7$.  In Figure~\ref{fig:nc_errors} we plot on a logarithmic scale the relative $L^2$ errors
\begin{align*}
\frac{\left\|\psi_7-\psi_p\right\|_{L^2\left(\Gamma\right)}}{\left\| \frac{1}{k}\partial_\nu^+ u \right\|_{L^2\left(\Gamma\right)}},
\end{align*}
against $p$ for a range of values of $k$.  As for the convex polygon, we scale the relative error by the quantity that we are actually trying to approximate, namely $(1/k)\partial_\nu^+ u = \varphi+\Psi^d/k \approx \psi_7 + \Psi^d/k$ (\ref{eqn:varphi_nonconvex}).  This is in contrast to results in~\cite{NonConvex}, where the relative errors were scaled by $\|\psi_7\|_{L^2(\Gamma)}$, i.e.\ by the ``reference'' solution to the BIE, which is only a component of the quantity we are seeking to approximate (compare Figure~\ref{fig:nc_errors} with~\cite[Fig.7 and Table 1]{NonConvex}).  As for the convex polygon, the $L^2$ norms are again calculated using a high-order Gaussian quadrature routine on a mesh graded towards the endpoint singularities.
\begin{figure}[htbp]
\begin{minipage}{0.5\linewidth}
\centering
\includegraphics[scale=0.33]{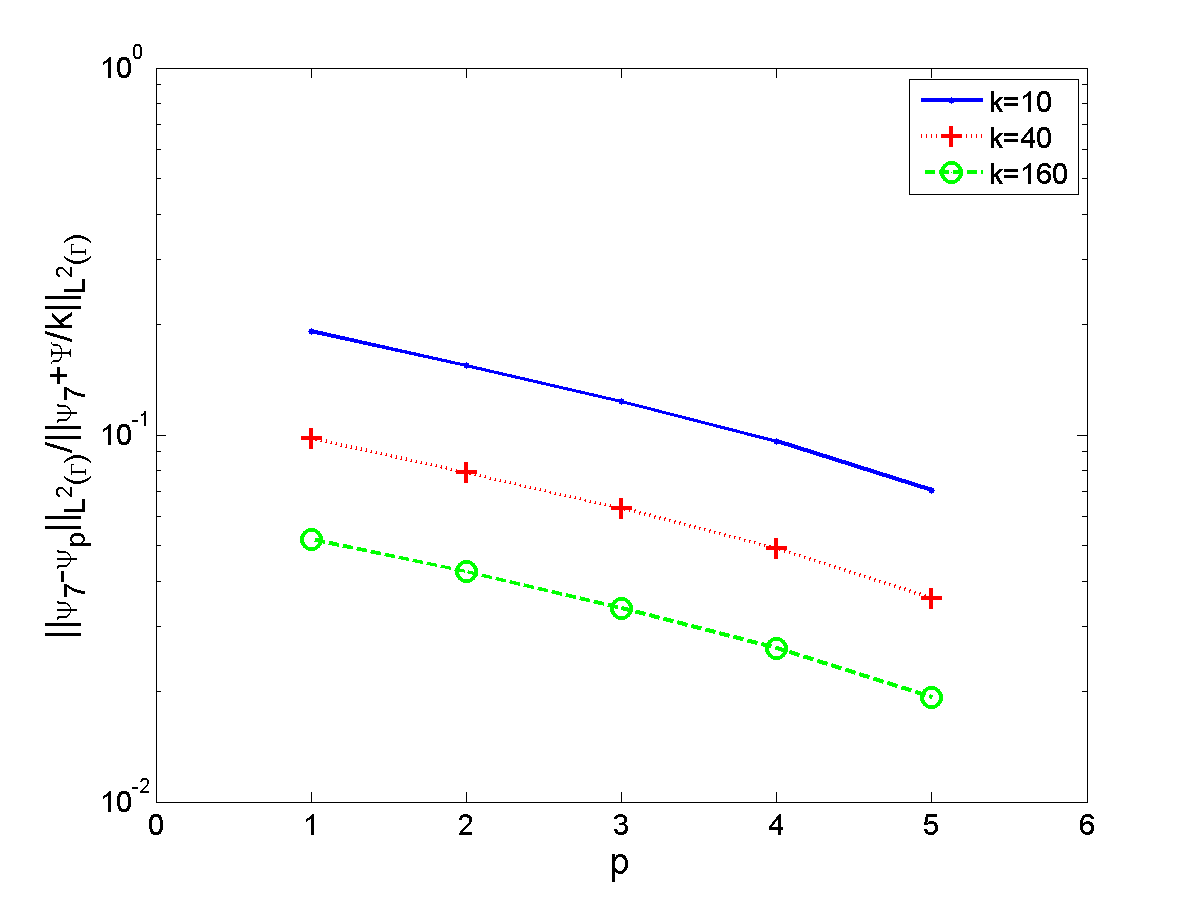}
\end{minipage}%
\begin{minipage}{0.5\linewidth}
\centering
\includegraphics[scale=0.33]{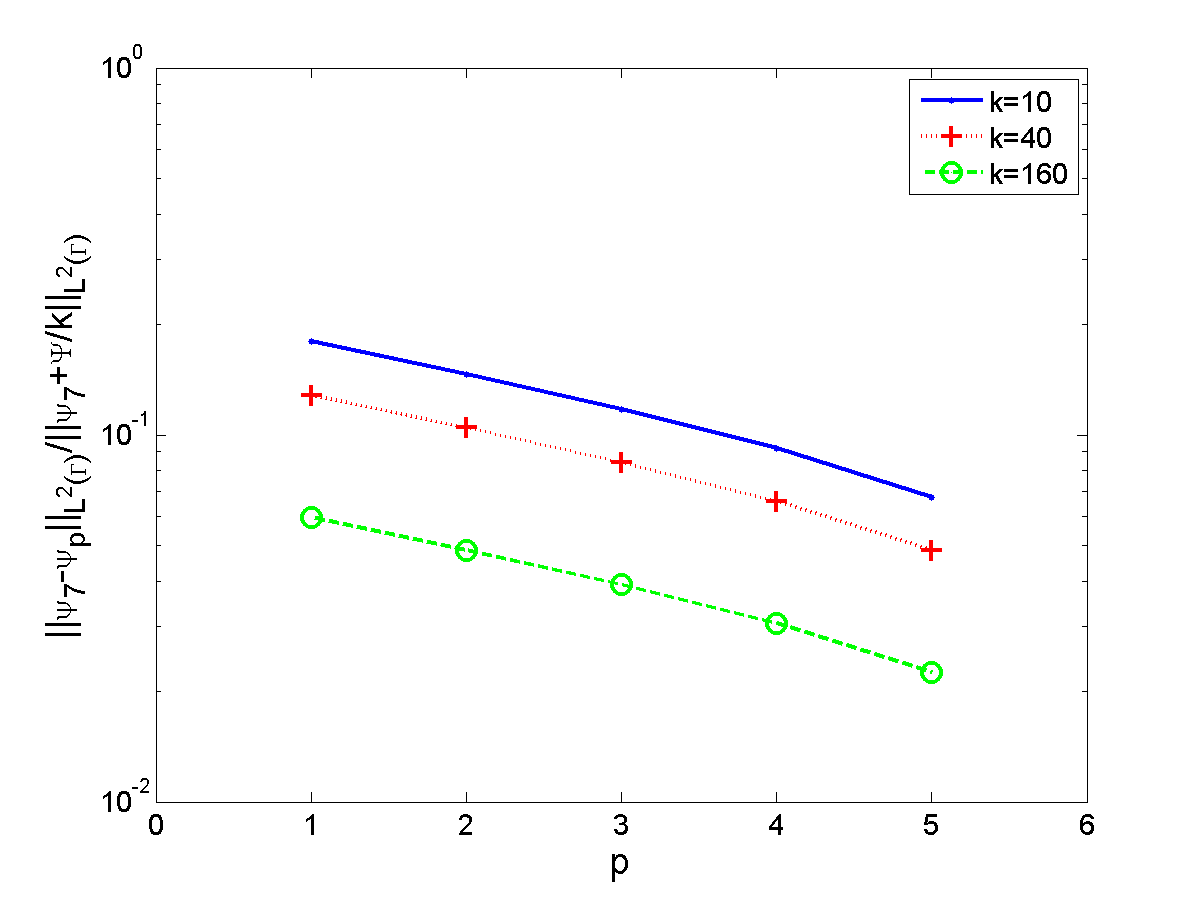}
\end{minipage}
\caption{Relative $L^2$ errors, scattering by a nonconvex polygon, with partial illumination (left) and rereflections (right), as in Figure~\ref{fig:nc_phenom}.}
\label{fig:nc_errors}
\end{figure}
For fixed $k$, as $p$ increases the error decays exponentially, as predicted by Theorem~\ref{thm:nonconvex_approx}.  For fixed $p$, the error seems to decrease as $k$ increases.


\subsubsection{Transmission scattering problems} Finally we consider the transmission scattering problem~(\ref{prob:tsp}), for which the behaviour is more complicated still, incorporating as it does multiple internal rereflections.  To illustrate this, consider the scenario illustrated in Figure~\ref{fig:trans1}, in which we show an incident wave striking a penetrable polygonal scatterer.
\begin{figure}[htbp]
\begin{minipage}{0.5\linewidth}
\centering
\includegraphics[scale=0.38]{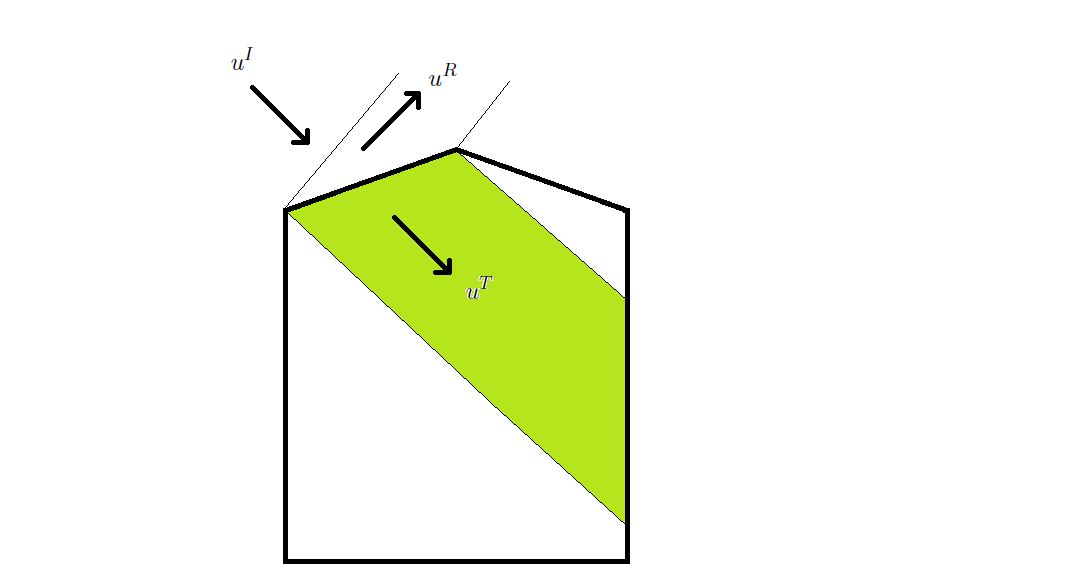}
\end{minipage}%
\begin{minipage}{0.5\linewidth}
\centering
\includegraphics[scale=0.38]{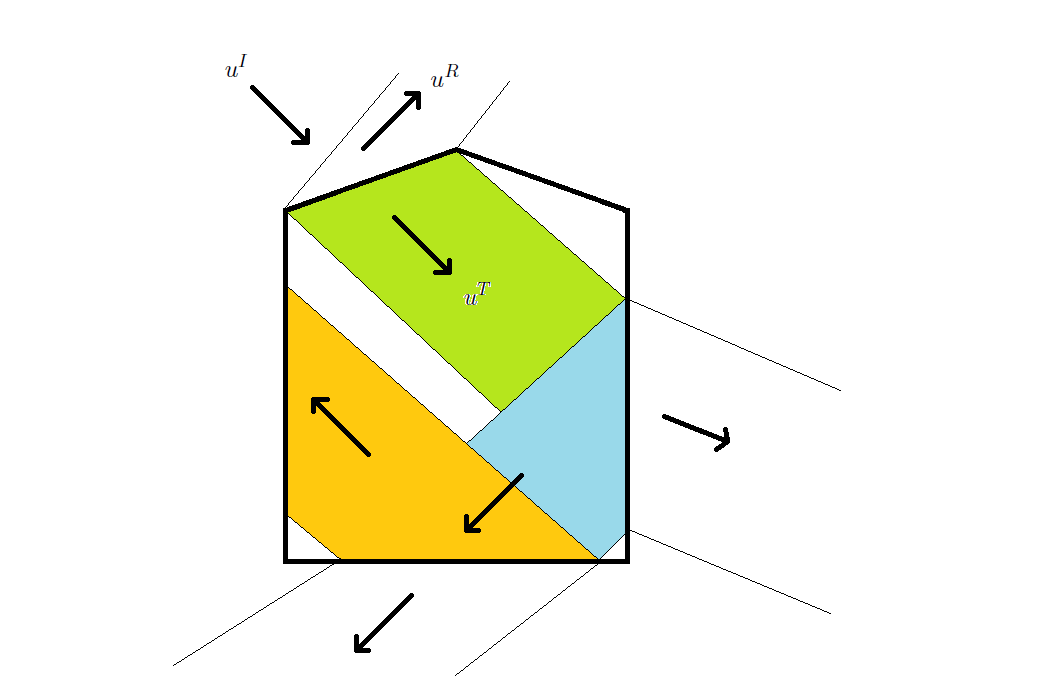}
\end{minipage}
\caption{Illustration, for the transmission scattering problem, of incident ($u^I$), reflected ($u^R$) and primary transmitted ($u^T$) field (left), with multiple internal rereflections (right)}
\label{fig:trans1}
\end{figure}
On the left of Figure~\ref{fig:trans1} we show what happens as the incident wave $u^I$ strikes one side of the boundary $\Gamma$.  Following Snell's Law (see, e.g., \cite[Appendix~A]{GrHeLa:13}), this gives rise to a ``reflected'' wave $u^R$ (travelling from $\Gamma$ into $\Omega_+$) and a ``transmitted'' wave $u^T$ that passes into $\Omega_-$.  For the case that the scatterer is impenetrable (as in all the other problems considered earlier in this section), only the reflected wave would be present here.  As the transmitted wave passes through $\Omega_-$ it may decay (if $\im{k_-}>0$), but if $k_-$ is constant then its direction does not change.  As this transmitted wave strikes $\Gamma$ it again gives rise to another transmitted wave (that now passes through $\Gamma$ into $\Omega_+$), and another reflected wave (which passes through $\Omega_-$ again, in a new direction).  This reflected wave again may lose energy as it passes through $\Omega_-$, but again it will strike $\Gamma$, leading to further reflected and transmitted waves, as shown on the right of Figure~\ref{fig:trans1}.  These multiple internally reflected waves are not present for impenetrable scatterers, and make the task of designing a hybrid approximation space (based on the ansatz~(\ref{eqn:ansatz})) extremely challenging.

We restrict attention here to the case of convex polygonal scatterers, and outline briefly the approach presented in~\cite{GrHeLa:13}.  For nonconvex scatterers or for scatterers with curved surfaces the problem would be significantly harder, and we do not discuss such generalisations.  Utilising Snell's law, we can write down explicit formulae for each term in the (in general infinite) series of ``rereflected'' waves (the first three of which are illustrated on the right of Figure~\ref{fig:trans1}).  We refer to~\cite{GrHeLa:13} for details, but note that these formulae rely on a complete understanding of what happens when a plane wave passes from one (possibly absorbing) homogeneous medium to another, and further that there appear to be some misconceptions in the literature regarding the solution to that canonical problem, which are addressed fully in~\cite[Appendix~A]{GrHeLa:13}.  Framing this in the context of~(\ref{eqn:ansatz}), this corresponds to expressing $V_0$ as an infinite series of these ``rereflected'' waves, that must be truncated in any numerical algorithm.

The part of the total field that is not represented by this series corresponds primarily to the ``diffracted'' waves emanating from the corners of the polygon (though other wave components, e.g.\ lateral waves, may also be present; again we refer to the discussion in~\cite{GrHeLa:13} for details).  These may originate from the incident field (e.g., as illustrated on the left of Figure~\ref{fig:trans2} below), or else they may originate from the waves that travel through the interior of the polygon striking corners that are on the ``shadow'' side of the polygon (using the same definition as for impenetrable scatterers above).  Either way these ``diffracted'' waves travel through $\Omega_+$ with speed determined by $\real{k_+}$, and through $\Omega_-$ with speed determined by $\real{k_-}$ (and decay determined by $\im{k_-}$).  These differing wavespeeds in $\Omega_-$ and $\Omega_+$ of course imply differing wavelengths, as illustrated in Figure~\ref{fig:trans2}.  A full consideration of the wave behaviour of the solution would also need to take into account reflection of these waves from each side of $\Gamma$.  We do not consider such rereflections of diffracted waves here.
\begin{figure}[htbp]
\begin{minipage}{0.5\linewidth}
\centering
\includegraphics[scale=0.38]{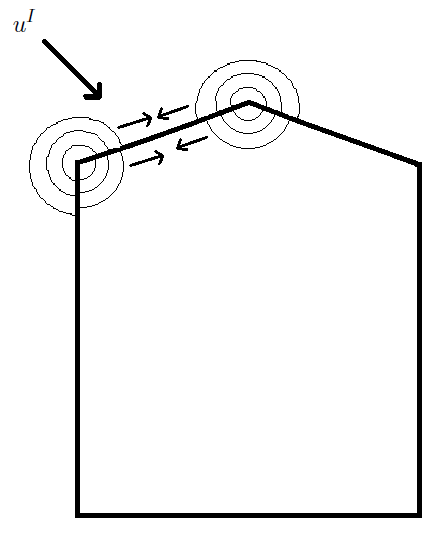}
\end{minipage}%
\begin{minipage}{0.5\linewidth}
\centering
\includegraphics[scale=0.38]{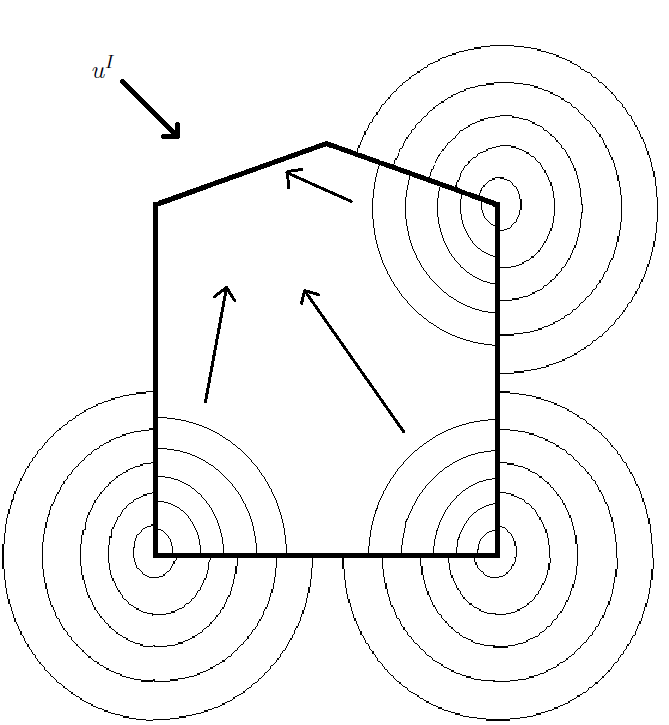}
\end{minipage}
\caption{Illustration of waves ``diffracted'' by the corners of a particular side of $\Gamma$ (left), and by other corners of the polygon (right)}
\label{fig:trans2}
\end{figure}

  We turn now to the integral equation formulation for the solution of the transmission scattering problem \eqref{prob:tsp}. We need to solve \eqref{eq:obvious} for $v = \left[\gamma_+u, \partial_\nu^+u\right]^T\in H^1(\Gamma)\times L^2(\Gamma)$, with $A$ given by \eqref{eq:BIE_trans} and $f=\left[ u^I|_\Gamma, \partial_\nu u^I|_\Gamma\right]$.  The HNA ansatz for $v$ then needs to incorporate the behaviour shown in Figures~\ref{fig:trans1} and~\ref{fig:trans2}.  On any given side $\Gamma_j$ (of length $L_j$) of the polygon, with $s$ representing distance from $P_{j-1}$, we represent the solution as
\begin{eqnarray}
  v(s) & = & v_0(s) + v_+^+(s) \re^{\ri k_+ s} + v_+^-(L_j-s) \re^{-\ri k_+ s} + v_-^+(s) \re^{\ri k_- s}  \nonumber \\
       &   & \quad + v_-^-(L_j-s) \re^{-\ri k_- s} + \sum_{m=1}^{n_s-2} v_m(s) \re^{\ri k_- r_m(s)}, \label{eqn:trans_ansatz}
\end{eqnarray}
where here: $v_0$ represents a (known) truncated series of ``rereflected'' plane waves (referred to as ``beams'' in~\cite{GrHeLa:13}), i.e.\ an approximation to the behaviour shown in Figure~\ref{fig:trans1}; $v_{\pm}^{\pm}(\cdot)\re^{\pm\ri k_\pm s}$ represents ``diffraction'' along $\Gamma_j$ from ``adjacent corners'', i.e.\ the behaviour as shown on the left of Figure~\ref{fig:trans2}, where here the phase functions $\re^{\pm\ri k_\pm s}$ capture the oscillations of these ``diffracted'' waves, but the amplitudes $v_{\pm}^{\pm}$ are unknown, and are approximated by piecewise polynomials on overlapping graded meshes (due to singularities at the corners) as for the sound-soft convex polygonal scattering problem described above; finally $v_m(s) \re^{\ri k_- r_m(s)}$ represents ``diffraction'' on $\Gamma_j$ emanating from non-adjacent corners, with here $r_m(s)$ denoting the distance from the $m$th corner to the point $x(s)$ parametrised by $s$ on $\Gamma_j$, i.e.\ the behaviour as shown on the right of Figure~\ref{fig:trans2}, where there the phase functions $\re^{\ri k_- r_m(s)}$ capture the oscillations of these ``diffracted'' waves, but the amplitudes $v_m$ are unknown, and are each approximated by piecewise polynomials (of order $p$) on a mesh on $\Gamma_j$, with mesh points lined up with potential discontinuities in $v_m$ arising from our sharp ``cutting off'' of the beams, as shown in Figure~\ref{fig:trans1} (for full details, see~\cite[\S3.2.3]{GrHeLa:13}).

This scheme is rather more complicated than the corresponding ones for impenetrable scatterers, but numerical results in~\cite[\S4]{GrHeLa:13} suggest that the best approximation achieved by this HNA approximation space has a relative error that decreases as $k$ increases, and further that represents a significant improvement over using $v_0$ (essentially the ``Geometrical Optics'' solution) on its own.

Here, we compare results from \cite[Table~1]{GrHeLa:13} with those presented in Table~1 above for the problem of sound soft scattering by an impenetrable convex polygon.  We consider the same geometry as described above for the case of scattering by a sound soft convex polygon, i.e.\ an equilateral triangle of side length $2\pi$, but now we allow the wave to pass through the triangle, as illustrated in Figure~\ref{fig:tri_trans}.
\begin{figure}[ht]
\begin{center}
  \includegraphics[width=0.75\linewidth]{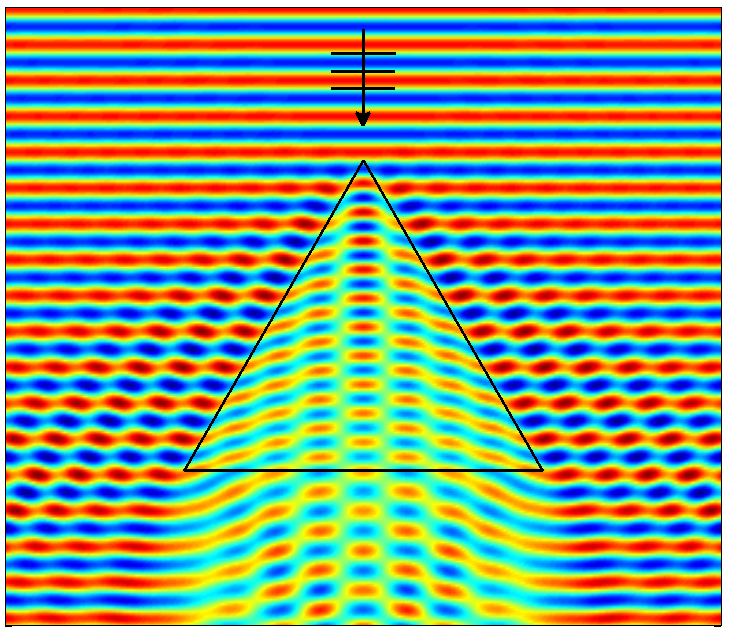}
\end{center}
\caption{Real part of the total field for scattering of a plane wave by a penetrable equilateral triangle.}
\label{fig:tri_trans}
\end{figure}
In the results below we take the refractive index of the media to be 1.31, so that for any given exterior wavenumber $k_+$, the interior wavenumber is $k_-=k_+(1.31+\xi\ri)$, where $\xi$ determines the level of absorption.  We fix $p=4$, in which case the total number of degrees of freedom required by the approximation space outlined above is 193 (compared to 192 for $p=3$ for the sound-soft convex polygon example, presented above).  In Figure~\ref{fig:trans_err} we plot the relative error in best approximation to both $u$ and $\partial u/\partial\nu$ on $\Gamma$, as computed with 193 degrees of freedom, for $\xi=0$ (zero absorption), for $\xi=0.025$ and for $\xi=0.05$.  The best approximation is computed via matching with a known ``reference'' solution computed using standard BEM on a very fine mesh - see~\cite{GrHeLa:13} for details.  For comparison, on each plot we also show the relative error in the HNA BEM solution to the problem of scattering by the sound soft convex polygon, as computed with 192 degrees of freedom (i.e.\ some of the results from Table~\ref{table:tri_errors}).
\begin{figure}[htbp]
\begin{minipage}{0.5\linewidth}
\centering
\includegraphics[scale=0.34]{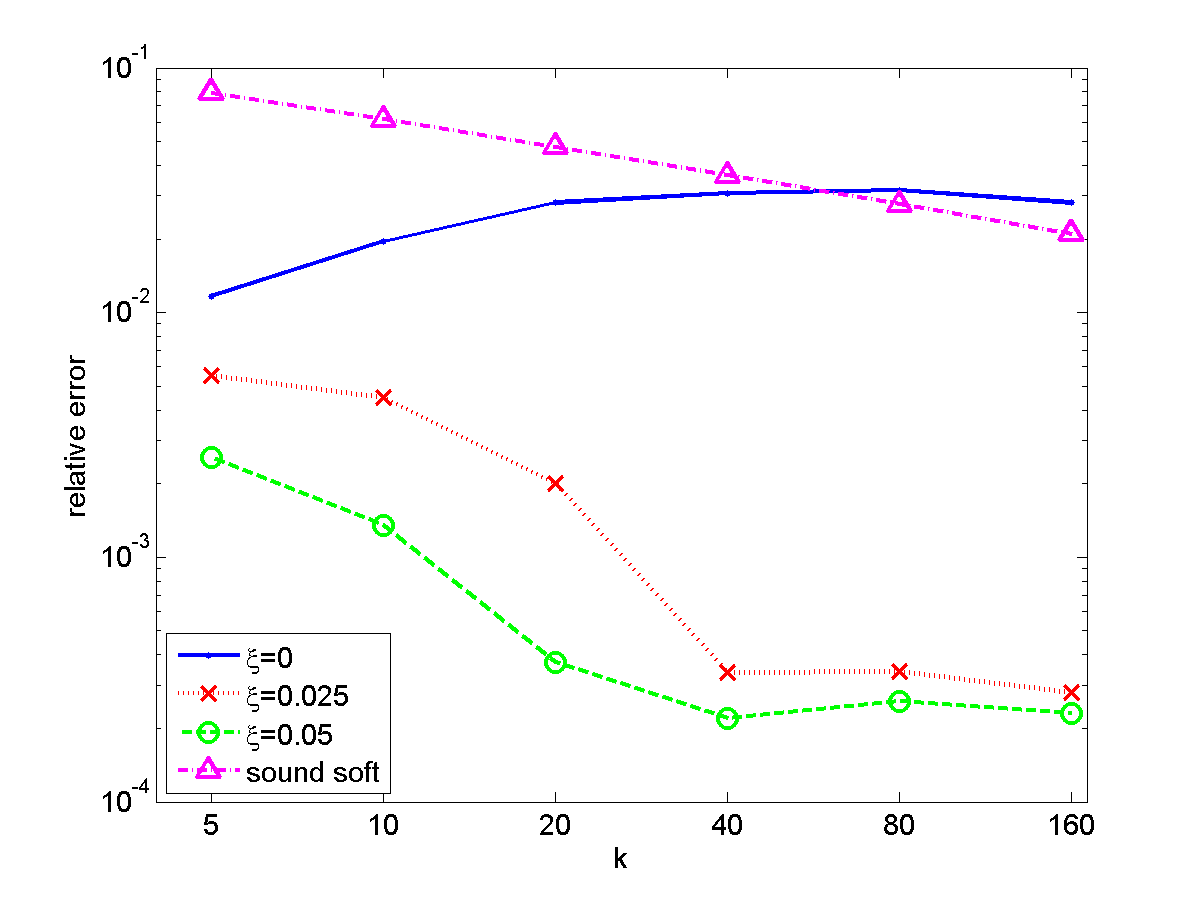}
\end{minipage}%
\begin{minipage}{0.5\linewidth}
\centering
\includegraphics[scale=0.34]{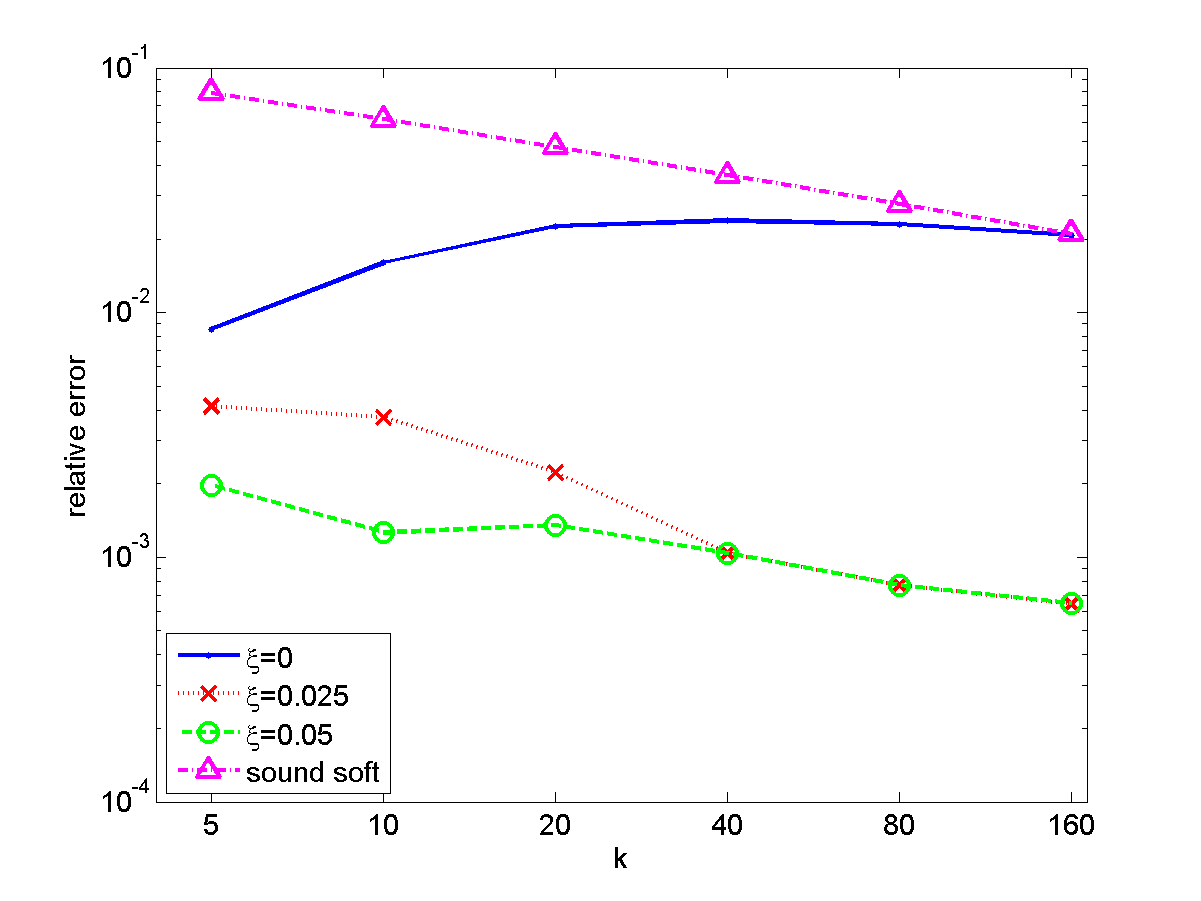}
\end{minipage}
\caption{Relative $L^2(\Gamma)$ best approximation errors in $u$ (left) and $\partial u/\partial\nu$ (right) for scattering by a penetrable triangle of varying absorption, and a comparison with the HNA BEM relative errors for scattering by a sound soft polygon.}
\label{fig:trans_err}
\end{figure}
The comparison here is not completely fair - for the sound soft problem we have computed the solution by solving the BIE using an HNA approximation space, whereas for the transmission problem we have merely fitted the best approximation from an HNA approximation space, and there is no guarantee that solving the BIE numerically would achieve this best approximation.  Also, in solving the sound soft problem, the relative error is in our approximation to $\partial u/\partial\nu$, and we compare that result here with the best approximation to both $u$ and $\partial u/\partial\nu$ for the transmission problem.  However, having said all that, we note that, except for the most challenging case of zero absorption, the error in best approximation from the HNA approximation space for the transmission problem is comparable or better than the HNA BEM solution for the much simpler problem of scattering by a sound soft convex polygon.

Although in this case there is no rigorous theory, unlike for the simpler cases of the screen and the impenetrable polygon, the results in Figure~\ref{fig:trans_err} (and see~\cite{GrHeLa:13} for a much wider range of examples) demonstrate that the HNA BEM has the potential to be successful even for this significantly more complicated scenario.  Further results in~\cite{GrHeLa:13}, also for the calculation of the solution in the domain and the far field pattern, show that good results can be achieved for a range of scatterers with different absorptions; as absorption reduces, so the influence of diffraction from non-adjacent corners increases, and we surmise that, in this case, it may be necessary to add additional terms to the ansatz~(\ref{eqn:trans_ansatz}) in order to achieve higher levels of accuracy.  We note though that results in~\cite{GrHeLa:13} suggest that the ansatz outlined above is sufficient to achieve 1\% relative error in the far-field pattern for any absorption and frequency (for the range of examples tested).

\subsubsection{Other boundary conditions and 3D problems}
We have focussed mostly in this section on sound soft scattering problems. There is no difficulty in extending the algorithms and much of the analysis to sound hard or impedance scattering problems. In particular, the HNA approach has been very successfully applied to the problem of scattering by convex polygons with impedance boundary conditions (see \cite{CWLM} for details), solving \eqref{eq:BIE_imp2} with $\eta=0$.  We do not include specific details of that case here, but note that the approximation space is very similar to that for scattering by sound-soft convex polygons, as detailed above, and also that a summary of the approach in that case and further numerical results can be found in~\cite{ChGrLaSp:12}.

Much more challenging is extension to 3D, because of the greater complexity of the high frequency asymptotics of the solution, in particular the much larger number of possible contributing ray paths and associated oscillatory phases. But at least for significant classes of 3D problems it seems likely that this methodology will be effective, leading to substantially reduced computational cost compared to standard BEMs.

As a starting point, we reported in \cite[\S7.6]{ChGrLaSp:12} initial numerical experiments by Hewett for scattering in 3D by a sound-soft square screen (the 3D version of the problem tackled in \S\ref{sec:screens}), exploring whether the high frequency ansatz \eqref{eqn:ansatz} is able to approximate the exact solution with a small number of degrees of freedom, experimenting with different choices for the phases $\phi_m$ and the number $M$ of oscillatory phases included. In work in progress this ansatz has been implemented in a Galerkin scheme analogous to that in \S\ref{sec:screens}~\cite{HaHeLaLa:14}. These results suggest that the best approximation results for certain 3D problems are broadly achieved by HNA BEM in practice, and that inclusion of appropriate oscillatory basis functions in the approximation space as outlined here can lead to high accuracy at reasonably high frequencies with a relatively small number of degrees of freedom, certainly compared to standard BEM, even in the 3D case.

\section{Links to the unified transform method}
\label{sec:uni}
As noted in the introduction, this paper appears in a collection of articles in significant part focussed on  the so-called ``unified transform'' or
``Fokas transform'' introduced by Fokas in 1997 \cite{Fo97} and developed further by Fokas and
collaborators since then (see, e.g., \cite{DeTrVa14} and the other papers in this collection). This method is on the one hand an analytical transform technique which can be employed to solve linear BVPs with constant coefficients in canonical domains, for which it serves as a generalisation of classical transform and separation of variables techniques (for a high frequency application, to acoustic scattering by a circular domain, see  \cite{FoSp12}). But also this method, when applied in general domains, has some aspects in common with BIE and boundary element methods, as discussed by Spence \cite{Spence14} in this volume.

The investigation of the unified transform method as a numerical scheme for the Helmholtz equation is arguably in its infancy. The focus to date has been on numerical experiments for particular geometries for the 2D interior Dirichlet problem \eqref{prob:idp}, and general methodologies and associated theoretical numerical analysis results are limited to date.    Spence \cite[\S10.3]{Spence14} in this volume reviews these developments in detail, proves some new theoretical convergence results for one particular proposed implementation, and makes connections with boundary integral equation and boundary element methods, and other numerical schemes. We add to this discussion in \S\ref{sec:unidp} below, in particular pointing out that one particular implementation (see Theorem \ref{thm:unified} below), approximating the unknown Neumann data from a space of traces of plane waves, computes precisely the best approximation from that space. We also make additional connections to related methods: the {\em least squares method} and the {\em method of fundamental solutions}: see Remarks \ref{rem:ls} and \ref{rem:mfs}.

The unified transform method, as articulated in \S\ref{sec:unidp} below and in \cite{SpFo:10,Spence14}, does not apply to exterior problems for the Helmholtz equation, at any rate to exterior problems set in the exterior of a bounded set $\Omega_-$, such as the exterior Dirichlet problem \eqref{prob:edp}. The issue is that plane wave (and generalised plane wave) solutions of the Helmholtz equation, which are fundamental to the method (see \S\ref{sec:unidp}) do not satisfy the standard Sommerfeld radiation condition \eqref{eqn:SRC} (for more discussion see \S\ref{sec:undg} below, and note that \cite{FoLe:14} {\em does} achieve an implementation for a particular exterior problem for the {\em modified} Helmholtz equation, i.e., \eqref{eqn:HE} with $k$ pure imaginary). But the unified transform method can be applied to so-called {\em rough surface scattering problems}, where the scatterer takes the form
\begin{equation} \label{eq:halfspace}
\Omega_- := \{ \bx = (\tilde \bx, x_d)\in \R^d: x_d < f(\tilde \bx)\},
\end{equation}
for some bounded, Lipschitz continuous function $f:\R^{d-1}\to\R$ so that $\Gamma$ is the graph of $f$ and the boundary value problem to be solved is posed in the perturbed half-space $\Omega_+:= \R^d \setminus \overline{\Omega_-}$. Generalised plane waves (as defined in \S\ref{sec:unidp}) that propagate upwards or decay in the vertical direction satisfy the appropriate radiation conditions in this case: in 2D these are the so-called {\em Rayleigh expansion radiation condition}, \eqref{eq:rerc} below, in the case when $f$ is periodic, and the {\em upwards propagating radiation condition} \cite{HalfPlaneRep} more generally.

Not only can the unified transform method be applied to these rough surface scattering problems, {\em it already has been applied in these cases}, developed independently in papers by DeSanto and co-authors from 1981 onwards \cite{DeSanto:81,DeSaErHeMi:98,DeSaErHeMi:01,DeSaErHeKrMiSw:01,ArChDeSa:06}. In \S\ref{sec:undg} below, we recall this method for the simplest of these problems, the 2D sound soft scattering problem \eqref{prob:ssp} in the particular case when the scatterer  $\Omega_-$ is a one-dimensional {\em diffraction grating} by which we will mean that $d=2$ and $f:\R\to \R$ is periodic (this case considered in particular in \cite{DeSaErHeMi:98,ArChDeSa:06}). We point out that the so-called {\em spectral-coordinate (SC)} and {\em spectral-spectral (SS)} methods proposed in \cite{DeSaErHeMi:98} correspond to two implementations of the unified transform methods with different choices of approximation space. We note that the $SS^*$ method proposed in \cite{ArChDeSa:06}, a variant of the SS method, corresponds precisely to the method for the interior Dirchlet problem \eqref{prob:idp} analysed in Theorem \ref{thm:unified} below, and we prove a new result (Theorem \ref{thm:unified2}) characterising and proving convergence of this method, sharpening \cite[Lemma 4.1]{ArChDeSa:06}. We also discuss the conditioning of the linear systems that arise from these methods.

\subsection{The unified transform method for the interior Dirichlet problem} \label{sec:unidp}
At the heart of the unified transform method is the so-called {\em global relation}. For linear elliptic PDEs with constant coefficients this global relation follows from the divergence theorem.  In particular, as described in \cite{SpFo:10,Spence14} for the Helmholtz equation \eqref{eqn:HE} (and the Laplace and modified Helmholtz equations), applying Green's second theorem to a solution $u\in H^1(D)$ of \eqref{eqn:HE} and a function $v\in \mathcal{R} := \{ v\in H^1(D)\cap C^2(D): \Delta u + k^2u=0 \mbox{ in } D\}$  in a bounded two-dimensional Lipschitz domain $D$ gives that
\begin{equation} \label{eq:gl}
\int_\Gamma \partial_\nu u \overline{\gamma v} \rd s = \int_\Gamma \gamma u \overline{\partial_\nu v} \rd s, \mbox{ for all } v\in \mathcal{R}.
\end{equation}
We remark that the lower part of \eqref{eq:grt_int}, i.e. \eqref{eq:grt_int} with $\bx\in \Omega_+$, is precisely a particular instance of \eqref{eq:gl}.

For \eqref{eqn:HE} with $k>0$, the {\em global relation} \cite{SpFo:10,Spence14} is \eqref{eq:gl} restricted to the subset $\mathcal{P} \subset \mathcal{R}$ of separable solutions of \eqref{eqn:HE} in Cartesian coordinates. In  two dimensions $\mathcal{P}$ is the one-parameter family $\mathcal{P} = \{v(\cdot,\theta): \theta\in \C\}$, where
$v(\bx,\theta) := \exp(\ri k(\cos \theta \,x_1 + \sin\theta \, x_2))$. For $\theta\in \R$, $v(\cdot,\theta)$ is a plane wave travelling in the direction $\hat \ba = (\cos\theta,\sin \theta)$. For complex $\theta$ we term $v(\cdot,\theta)$ a {\em generalised} plane wave (often alternatively an {\em evanescent} or {\em inhomogeneous} plane wave). The following lemma notes linear independence and density properties of $\mathcal{P}$.

\begin{lem} \label{lem:dense}
Suppose that $-k^2$ is not a Dirichlet eigenvalue of the Laplacian in $D$. Then:

(i) For $N\in \N$, $\gamma v(\cdot, \theta_1),...,\gamma v(\cdot,\theta_N)$ are linearly independent if $\theta_1,...,\theta_N\in\C$ are distinct.

(ii) The linear span of $\{\gamma v(\cdot,\theta):0\leq \theta<2\pi\}$ is dense in $L^2(\Gamma)$.
\end{lem}
\begin{proof}
(i) This is a development of standard arguments, e.g., \cite{AlVa:05}, which demonstrate the linear independence of ordinary plane waves. Suppose that $c_1,...,c_N\in\C$ and $v:= \sum_{n=1}^N c_n v(\cdot,\theta_n)=0$ on $\Gamma$. Then $v=0$ in $D$ and, since $v\in C^2(\R^2)$ and satisfies \eqref{eqn:HE} in $\R^2$, it follows by analyticity of $v$ \cite[p.~72]{CoKr:83} that $v=0$ in $\R^2$.

For $n=1,...,N$ set $\alpha_n := k\cos\theta_n$ and $\beta_n := k\sin \theta_n$. Suppose, without loss of generality, that $N=2M$ is even, that $N\geq 4$, and that $\alpha_1=1$, $\alpha_2=-1$, and $\alpha_{2m-1}=\alpha_{2m}$, in which case $\beta_{2m-1} = -\beta_{2m}\neq 0$, for $m=2,...,M$. Since $v((t,0))=0$ for $t\in\R$, it holds that
\begin{equation} \label{eq:linind}
\sum_{m=1}^{M+1} d_m \re^{\ri \eta_m t} = 0, \quad t\in\R,
\end{equation}
where $d_m:=c_m$, $\eta_m:= \alpha_m$, for $m=1,2$, and $\eta_{m+1}:= \alpha_{2m-1}$, $d_{m+1} := c_{2m-1}+c_{2m}$, for $m=2,...,M$. Since the $\theta_n$, $n=1,...,N$, are distinct, so also are the $\eta_m$, $m=1,...,M+1$. Let $n\in \{1,...,M+1\}$ be such that $\im{\eta_n}\leq \im{\eta_m}$, for $m=1,...,M+1$. Then, multiplying \eqref{eq:linind} by $\exp(-\ri\eta_n)$ and integrating it follows that
$$
\sum_{m=1}^M \frac{d_m}{A}\int_{A}^{2A}\re^{\ri (\eta_m-\eta_n) t}\,\rd t = 0, \quad A>0.
$$
Taking the limit $A\to\infty$ we see that $d_n=0$. Repeating this argument we deduce that $d_m=0$, $m=1,...,M+1$, so that $c_1=c_2=0$ and $c_{2m-1}=-c_{2m}$, $m=2,...,M$.

To conclude we note also that $\partial v(x)/\partial_{x_2}=0$ for $x=(t,0)$ and $t\in \R$. This implies that
\begin{equation} \nonumber
\sum_{m=3}^{M+1} e_m \re^{\ri \eta_m t} = 0, \quad t\in\R,
\end{equation}
where $e_m = \ri\beta_{2m}(c_{2m}-c_{2m-1})$. Arguing as for \eqref{eq:linind} we deduce that $e_m=0$, for $m=2,...,M+1$, so that $c_m=0$, for $m=1,...,2M$.

(ii) This is an easy consequence of a standard result on the Herglotz integral operator \cite[Theorem 5.24]{CoKr:92}.
\end{proof}

Spence \cite{Spence14} reviews the implementations to date of the unified transform method for the (2D) interior Dirichlet problem \eqref{prob:idp}.
  These implementations impose \eqref{eq:gl} (with $\gamma u = h$) for $v\in \mathcal{P}_N$, where $\mathcal{P}_N$ is an $N$-dimensional {\em test space} of generalised plane waves: explicitly, for some distinct $\theta_{j,N}\in \C$, $j=1,...,N$, $\mathcal{P}_N$ is the space spanned by $\{v(\cdot,\theta_{j,N}): 1\leq j\leq N\}$. An approximation $\phi_M$ to $\partial_\nu u$, which is an element of $Q_M\subset H^{-1/2}(\Gamma)$, the $M$-dimensional {\em trial space} with $M\geq N$, is obtained by requiring that
\begin{equation} \label{eq:fd}
\int_\Gamma \phi_M \overline{\gamma v_N} \rd s = \int_\Gamma h \overline{\partial_\nu v_N} \rd s, \mbox{ for all } v_N\in \mathcal{P}_N,
\end{equation}
this equation overdetermined if $M>N$ in which case it is imposed, e.g., in a least squares sense. Spence \cite{Spence14} tabulates the implementations to date, which vary in the choice of approximation space $Q_M$ and in the choice of generalised plane waves, i.e. in the choice of $\theta_{j,N}\in \C$, $j=1,...,N$.

Assume now that $-k^2$ is not a Dirichlet eigenvalue of the Laplacian, in other words (see the discussion below \eqref{prob:idp}), that the interior Dirichlet problem \eqref{prob:idp} has a unique solution for all $h\in H^{1/2}(\Gamma)$. In that case there is a well-defined operator $P_{\mathrm{DtN}}:H^{1/2}(\Gamma)\to H^{-1/2}(\Gamma)$, the {\em Dirichlet to Neumann map}, that takes the Dirichlet data $h\in H^{1/2}(\Gamma)$ to $\partial_\nu u\in H^{-1/2}(\Gamma)$, where $u$ is the solution of the BVP \eqref{prob:idp}. From \eqref{eq:bie_idp1} we see that, explicitly,
\begin{equation} \label{eq:dtn}
P_{\mathrm{DtN}} = S_k^{-1}(D_k + \textstyle{\frac{1}{2}}I).
\end{equation}
 With this notation the global relation \eqref{eq:gl} for the interior Dirichlet problem can be written equivalently as
\begin{equation} \label{eq:gl2}
\int_\Gamma \phi \overline \psi \rd s = \int_\Gamma h \overline{P_{\mathrm{DtN}} \psi} \rd s, \mbox{ for all } \psi\in H^{1/2}(\Gamma),
\end{equation}
where $\phi=\partial_\nu u\in H^{-1/2}(\Gamma)$.

The implementations of the unified transform method reviewed in \cite{Spence14} can be viewed as Petrov-Galerkin methods for the variational equation  \eqref{eq:gl2}, {\em Petrov}-Galerkin since the trial and test spaces differ, reflecting that the trial and test spaces, $H^{-1/2}(\Gamma)$ and $H^{1/2}(\Gamma)$, respectively, differ in \eqref{eq:gl2} at the continuous level. But
in the case that the Dirichlet data is sufficiently smooth, precisely when $h\in H^1(\Gamma)$ in which case, by Corollary \ref{cor:main_equivalence}, $\phi\in L^2(\Gamma)$, we can think of \eqref{eq:gl2} alternatively as a variational problem on $L^2(\Gamma)\times L^2(\Gamma)$. (Note that, from \eqref{eq:dtn} and \cite[Theorem 2.25]{ChGrLaSp:12}, $P_{\mathrm{DtN}}$ extends to a bounded mapping from $L^2(\Gamma)$ to $H^{-1}(\Gamma)$, this the dual space of $H^1(\Gamma)$ with respect to the $L^2(\Gamma)$ inner product, so that, when $h\in H^1(\Gamma)$, the right hand side of \eqref{eq:gl2} defines a continuous anti-linear functional on $L^2(\Gamma)$.) Indeed, in this setting, \eqref{eq:gl2} is the simplest possible variational problem: \eqref{eq:gl2} corresponds to \eqref{eq:var} with $\mathcal{V}=\mathcal{V}^\prime = L^2(\Gamma)$ and $A=I$, the identity operator! Of course this problem is trivially continuous and coercive, with continuity and coercivity constants, $C$ and $\alpha$ in Lemma \ref{lem:cea}, equal to one. Thus attractive is to choose $M=N$ and $Q_M := \mathcal{P}^\Gamma_N$, where $\mathcal{P}^\Gamma_N:= \gamma(\mathcal{P}_N) = \{ \gamma v:v\in \mathcal{P}\}$ is a space of restrictions of plane waves to $\Gamma$, for in that case \eqref{eq:fd} is a standard Galerkin method and, by C\'ea's lemma (Lemma \ref{lem:cea}) or standard properties of orthogonal projection in Hilbert spaces, e.g., \cite[Theorem 12.14]{Rudin:91}, one has the following result.

\begin{thm} \label{thm:unified}
Suppose that $-k^2$ is not a Dirichlet eigenvalue, so that \eqref{prob:idp} is uniquely solvable,  and $h\in H^1(\Gamma)$ so that $\phi:=\partial_\nu u\in L^2(\Gamma)$. Suppose also that $M=N\in \N$ and $Q_M:=\mathcal{P}^\Gamma_N:=\gamma(\mathcal{P}_N)$, where $\mathcal{P}_N$ is some $N$-dimensional subspace of $\mathcal{P}$. Then \eqref{eq:fd} has a unique solution which is $\phi_N=P_N \phi$, where $P_N:L^2(\Gamma)\to \mathcal{P}^\Gamma_N$ is orthogonal projection, so that $\phi_N$ is the best $L^2(\Gamma)$ approximation to $\phi$ in $\mathcal{P}^\Gamma_N$.
\end{thm}
With the choices made in this theorem it holds that
$$
\phi_N = \sum_{n=1}^N c_n \gamma v(\cdot,\theta_{n,N}),
$$
for some complex coefficients $c_n$, and \eqref{eq:fd} is equivalent to the linear system
\begin{equation} \label{eq:ls}
\sum_{n=1}^N a_{mn} c_n = \int_\Gamma h \,\overline{\partial_\nu v(\cdot,\theta_{m,N})} \,\rd s, \quad m = 1,...,N,
\end{equation}
where $a_{mn}= \int_\Gamma \gamma v(\cdot,\theta_{n,N}) \overline{\gamma v(\cdot,\theta_{m,N})} \rd s$. It is easy to see that the matrix $[a_{mn}]$ is Hermitian and positive semi-definite, indeed positive definite in view of Lemma \ref{lem:dense}(i): see, e.g., the discussion in \cite[\S5]{ArChDeSa:06}.

\begin{rem} \label{rem:ls} The same matrix $[a_{mn}]$ and a similar right hand side arises when solving the interior Dirichlet problem by a least squares method (e.g., \cite{AlVa:05}). Seek an approximation $u_N$ to the solution $u$ of \eqref{prob:idp} in the form
$$
u_N = \sum_{n=1}^N c^\prime_n v(\cdot,\theta_{n,N}),
$$
choosing the complex coefficients $c_n^\prime$ to minimise $\|h-\gamma u_N\|_{L^2(\Gamma)}$. Then the coefficients $c_n^\prime$ satisfy
$$
\sum_{n=1}^N a_{mn} c_n^\prime = \int_\Gamma h \overline{\gamma v(\cdot,\theta_{m,N})} \rd s, \quad m = 1,...,N.
$$
\end{rem}

\begin{rem} \label{rem:mfs} The focus above is on the {\em unified method} as a numerical method, i.e. on \eqref{eq:fd}. But we note that the same analysis and results apply (and in 2D and 3D) if we replace $\mathcal{P}_N$ in \eqref{eq:fd} by any $N$-dimensional subspace of $\mathcal{R}$. Not least Theorem \ref{thm:unified} holds (in 2D and 3D) with $\mathcal{P}_N$ any $N$-dimensional subspace of $\mathcal{R}$, and \eqref{eq:ls} holds with $v(\cdot,\theta_{j,N})$, $j=1,...,N$, replaced by any choice of basis for this amended $\mathcal{P}_N$.

One possible choice of $\mathcal{P}_N\subset \mathcal{R}$, this choice suggested by the lower part of \eqref{eq:grt_int}, is to take $\mathcal{P}_N$ to be the space spanned by $\{\overline{\Phi(\cdot,\bx_n)}:1\leq n\leq N\}$, where $\bx_1,...,\bx_N$ are distinct points in $\R^d\setminus \overline D$. With this choice \eqref{eq:fd} is a variant on the so-called {\em method of fundamental solutions}, e.g., \cite{AlVa:05}.
\end{rem}

\subsection{Diffraction gratings and the unified transform method} \label{sec:undg}

The previous subsection has described the unified transform method as a numerical method for interior problems, in particular the interior Dirichlet problem \eqref{prob:idp}, but our focus in this paper is acoustic {\em scattering} in which we are solving exterior problems. As noted above, we cannot see how the unified transform method, as currently formulated, can be applied to any of the exterior or scattering problems that we have stated in \S\ref{sec:BIE}, where we need to solve the Helmholtz equation in the exterior of  a bounded set $\Omega_-$, whose boundary is denoted by $\Gamma$. In particular, the global relation \eqref{eq:gl} does not hold for any generalised plane wave $v$ in this case.

But, as outlined above, the global relation {\em does} hold for some generalised plane waves in the rough surface scattering case, where the scatterer is the non-locally perturbed half-plane (in 2D) or half-space (in 3D), given by \eqref{eq:halfspace}. For this geometry, in both 2D and 3D, versions of the global relation and the unified transform method have been developed independently (and with different terminology) by DeSanto and co-workers \cite{DeSanto:81,DeSaErHeMi:98,DeSaErHeMi:01,DeSaErHeKrMiSw:01,ArChDeSa:06}.

In this section we describe (and elaborate on) this work for the simplest case, for which the numerical analysis is most complete \cite{DeSanto:81,DeSaErHeMi:98,ArChDeSa:06}, namely the 2D sound soft scattering problem and associated Dirichlet problem in the case when $f:\R\to\R$ in \eqref{eq:halfspace} is periodic, with some period $L$. For this scattering problem, in the case when the incident field $u^I$ is the plane wave \eqref{eqn:plane} for some $\hat \ba = (\sin \theta^I,-\cos \theta^I)$, where $\theta^I\in (-\pi/2,\pi/2)$ is the angle of incidence, it is natural to look for a solution $u$ to the scattering problem which is {\em quasi-periodic} with period $L$ and phase shift $\mu=k\sin \theta^I$, meaning that
\begin{equation} \label{eq:per}
u((x_1+L,x_2)) = \exp(\ri \mu L) u(\bx), \quad \mbox{for } \bx \in \Omega_+.
\end{equation}

Let $\Omega_+^L := \{\bx = (x_1,x_2)\in \Omega_+: 0<x_1<L\}$, $\Gamma^L:= \{\bx\in \Gamma:0<x_1<L\}$, and, for $\mu\in \R$, let $C^2_{\mu}(\Omega^L_+)$ denote those  $u\in C^2(\Omega_+)$ that satisfy \eqref{eq:per}. The standard radiation condition imposed on the scattering field $u^S:= u-u^I$ for this problem can be obtained by applying separation of variables in $\Omega_+^L$ to the Helmholtz equation \eqref{eqn:HE} under the constraint that $u\in C^2_{\mu}(\Omega_+)$. This leads to an expression for $u^S$ as a countable linear combination of generalised plane waves. Discarding those generalised plane waves which are growing exponentially away from $\Gamma$, or are plane waves propagating towards $\Gamma$, leads to the representation (the {\em Rayleigh expansion radiation condition} (RERC)) that
\begin{equation} \label{eq:rerc}
u^S(\bx) = \sum_{n\in\Z} c_n \exp(\ri k[\alpha_n x_1 + \beta_n x_2]), \quad \mbox{for } x_2 > f_+ := \max(f),
\end{equation}
for some complex coefficients $c_n$, where
\begin{equation} \label{eq:abdef}
\alpha_n  := \mu/k + 2\pi n/(kL) \;\mbox{ and } \;\beta_n := \left\{\begin{array}{cc}
                    \sqrt{1-\alpha_n^2}, & |\alpha_n| \leq 1, \\
                    \ri\sqrt{\alpha_n^2-1}, & |\alpha_n| > 1.
                  \end{array}\right.
\end{equation}
 The standard Dirichlet problem in this case is then:
\begin{equation} \label{prob:dgdp}
\begin{array}{l}
  \mbox{Given }h\in H_{\mu}^{1/2}(\Gamma^L), \mbox{ find }u\in C^2_{\mu}(\Omega^L_+)\cap H_{\mathrm{loc}}^1(\Omega^L_+)\\
  \mbox{such that (\ref{eqn:HE}) holds in } \Omega_+, \gamma u =  h \mbox{ on }\Gamma^L,\\
  \mbox{and } u \mbox{ satisfies the RERC  \eqref{eq:rerc}}.
\end{array}
\end{equation}
Here, for $\mu\in \R$, $H_{\mu}^{1/2}(\Gamma^L)$ is the closure in $H^{1/2}(\Gamma^L)$ of those $\phi\in C^\infty(\Gamma)$ that satisfy \eqref{eq:per} for $\bx \in \Gamma$. That \eqref{prob:dgdp} is uniquely solvable is shown in \cite{ElYa:02}.

For $\mu\in\R$, let
$$
\mathcal{R}_{\mu}:= \{v\in C^2_{\mu}(\Omega^L_+)\cap H_{\mathrm{loc}}^1(\Omega^L_+): v  \mbox{ satisfies } \eqref{eqn:HE} \mbox{ and }\eqref{eq:rerc}\},
$$
and note that $u\in \mathcal{R}_\mu$ if $u$ is a solution of \eqref{prob:dgdp}.
The numerical schemes in \cite{DeSaErHeMi:98} derive from the observation that, if $u$ satisfies \eqref{prob:dgdp}, then (where $\nu$ is the unit normal directed into $\Omega_+$)
\begin{equation} \label{eq:gldg}
\int_{\Gamma^L} \partial_\nu u \gamma v \rd s = \int_{\Gamma^L} h\partial_\nu v\rd s, \mbox{ for all } v\in \mathcal{R}_{-\mu},
\end{equation}
 this identity \eqref{eq:gldg} derived by applying Green's second theorem to $u$ and $v$ in $\{\bx\in \Omega_+^L: x_2<H\}$, for some $H>f_+$. The identity \eqref{eq:gldg} holds, in particular, for those generalised plane waves $v(\cdot,\theta)$ that are in the set
$
\mathcal{P}_{\mu} := \{v(\cdot,\theta_n): n\in\Z\},
$
where $\theta_n\in \C$ is defined by
$$
(\cos \theta_n,\sin\theta_n) = (-\alpha_n,\beta_n),
$$
with $\alpha_n$ and $\beta_n$ given by \eqref{eq:abdef}. These are the generalised plane waves that are elements of $\mathcal{R}_{-\mu}$.

Thus a version of the global relation holds, that
\begin{equation} \label{eq:gldg2}
\int_{\Gamma^L} \partial_\nu u \,\overline{\gamma v} \,\rd s = \int_{\Gamma^L} h \,\overline{\partial_\nu v}\,\rd s, \mbox{ for all } v\in \mathcal{P}^*_{\mu},
\end{equation}
where
$$
\mathcal{P}^*_{\mu} := \{v(\cdot,\theta): (\cos \theta,\sin\theta) = (\alpha_n,-\overline{\beta_n}) \mbox{ and }n\in\Z\}\subset \mathcal{P}.
$$
Explicitly, \eqref{eq:gldg2} is the sequence of equations
\begin{equation} \label{eq:gldg3}
\int_{\Gamma^L} \partial_\nu u \, v(\cdot,\theta_n) \,\rd s = \int_{\Gamma^L} h \,\partial_\nu v(\cdot,\theta_n)\,\rd s, \quad n\in\Z.
\end{equation}
Equation \eqref{eq:gldg2}, which is equivalent to \eqref{eq:gldg3}, uniquely determines $\partial_\nu u$ because of the following density result (cf., Lemma \ref{lem:dense}).

\begin{lem} \label{lem:dense2} (i) \cite[Corollary 3.2]{ArChDeSa:06} The functions $\{\gamma v(\cdot,\theta_n):n\in\Z\}\subset L^2(\Gamma^L)$ are linearly independent.

(ii) \cite[Lemma 3.1]{ArChDeSa:06} The linear span of $\mathcal{P}^*_{\mathrm{R}}$, which is the linear span of $\{v(\cdot,\theta_n):n\in\Z\}$, is dense in $L^2(\Gamma^L)$.
\end{lem}

In the sound soft scattering problem, the incident field $u^I$ is the plane wave \eqref{eqn:plane}, with $\hat \ba = (\sin \theta^I, -\cos \theta^I) = (\alpha_0, -\beta_0)$, and the scattered field $u^S$ is the solution of \eqref{prob:dgdp} with $h=-u^I|_{\Gamma^L}$. So \eqref{eq:gldg3} holds with $u$ replaced by $u^S$ and $h=-u^I|_{\Gamma^L}$. But also, applying Green's second theorem to $u^I$ and $v(\cdot,\theta_n)$ in $\{\bx\in \Omega_+^L: x_2<H\}$ for some $H>f_+$, we deduce that
\begin{equation} \label{eq:gldg4}
\int_{\Gamma^L} \partial_\nu u^I \, v(\cdot,\theta_n) \,\rd s = \int_{\Gamma^L} \gamma u^I \,\partial_\nu v(\cdot,\theta_n)\,\rd s -2\ri kL\delta_{0,n}, \quad n\in\Z,
\end{equation}
where $\delta_{m,n}$ is the Kronecker delta. Adding \eqref{eq:gldg3} (with $u$ replaced by $u^S$ and $h=-\gamma u^I$) and \eqref{eq:gldg4}, we see that the total field $u=u^I+u^S$ satisfies
\begin{equation} \label{eq:gldg5}
\int_{\Gamma^L} \partial_\nu u \, \gamma v(\cdot,\theta_n) \,\rd s = -2\ri kL\delta_{0,n}, \quad n\in\Z,
\end{equation}
this equation dating back to \cite{DeSanto:81} (and see \cite[\S4]{DeSaErHeMi:98} and \cite{ArChDeSa:06}). It follows, e.g. from Rellich identities,  that $\partial_\nu u\in L^2(\Gamma^L)$ \cite{ElYa:02}.

DeSanto et al.~\cite{DeSaErHeMi:98} propose discretisations of \eqref{eq:gldg5} in which we seek an approximation to $\partial_\nu u$ in $Q_N$, some $N$-dimensional subspace of $L^2(\Gamma^L)$, and we impose \eqref{eq:gldg5} for $n=n_{1},...,n_{N}$, for distinct integers $n_j$; in other words we impose \eqref{eq:gldg5} for $v(\cdot,\theta_n)\in \mathcal{P}_N$, the $N$-dimensional space spanned by $\{v(\cdot,\theta_{n_m}):1\leq m\leq N\}$. If $\{\chi_{1},...,\chi_{N}\}$ is a basis for $Q_N$, and $\psi_n:= \overline{\gamma v(\cdot,\theta_n)}$,  then the numerical method is to approximate
\begin{equation} \label{eq:phiN}
\partial_\nu u \approx \phi_N = \sum_{m=1}^N c_m \chi_m
\end{equation}
where the coefficients $c_m$ satisfy the linear system
\begin{equation} \label{eq:linsys}
\sum_{m=1}^N \int_{\Gamma^L} \chi_m\, \overline{\psi_{n_j}} \, \rd s\, c_m = -2\ri kL\delta_{0,n_j}, \quad j=1,...,N.
\end{equation}

DeSanto et al.~\cite{DeSaErHeMi:98} propose two different choices of space $Q_N$ and basis functions $\chi_m$, the {\em spectral coordinate} (SC) method, using a ``pulse'' basis (piecewise constant basis functions), in which $\chi_m$ is the characteristic function of $\{\bx\in \Gamma^L: (m-1)h<x_1<mh\}$, where $h=L/N$, and the {\em spectral-spectral} (SS) method, in which $\chi_m := \gamma v(\cdot, \theta_{n_m}+\pi)$, for $m=1,...,N$. Numerical experiments with the SC and SS methods are carried out in \cite{DeSaErHeMi:98}. The methods are extended to transmission problems in \cite{DeSaErHeMi:01}, and the SC method to 3D sound soft and sound hard scattering problems for doubly-periodic, diffraction gratings surfaces in \cite{DeSaErHeKrMiSw:01}.

In \cite{ArChDeSa:06} a variant of the SS method is proposed, the SS$^*$ method, characterised by the choice $\chi_m = \psi_{n_m} = \overline{\gamma v(\cdot,\theta_{n_m})}$, so that $Q_N=\mathcal{P}_N^\Gamma$, the space spanned by $\{\gamma \overline{v(\cdot,\theta_{n_m})}:1\leq m\leq N\}$. An attraction of this method, as observed in \cite{ArChDeSa:06}, is that, as with \eqref{eq:ls}, this leads to a coefficient matrix $A_N= [a_{jm}]$, in this case with $a_{jm} = \int_{\Gamma^L}  \psi_{n_m}\, \overline{\psi_{n_j}} \, \rd s$, that is Hermitian and positive definite.

Analogously to \eqref{eq:gl2}, \eqref{eq:gldg5} can be viewed as a variational formulation problem on $L^2(\Gamma^L)\times L^2(\Gamma^L)$: \eqref{eq:gldg5} corresponds to \eqref{eq:var} with $\mathcal{V}=\mathcal{V}^\prime = L^2(\Gamma^L)$ and $A=I$, the identity operator. Like \eqref{eq:gl2} this formulation is trivially continuous and coercive, with continuity and coercivity constants, $C$ and $\alpha$ in Lemma \ref{lem:cea}, equal to one. The choice $Q_N=\mathcal{P}_N^\Gamma$ is a Galerkin method for \eqref{eq:gldg5},  and by C\'ea's lemma (Lemma \ref{lem:cea}) one has the following result (cf,~Theorem \ref{thm:unified}).

\begin{thm} \label{thm:unified2}
Suppose that $Q_N =\mathcal{P}^\Gamma_N$. Then $\phi_N$, given by \eqref{eq:phiN} and \eqref{eq:linsys}, is  the best $L^2(\Gamma^L)$ approximation to $\partial_\nu u$ in $\mathcal{P}^\Gamma_N$.
\end{thm}

We note that this result improves on \cite[Lemma 4.1]{ArChDeSa:06} where it is shown, under the assumptions of this theorem, that
$$
\|\partial_\nu u - \phi_N\|_{L^2(\Gamma^L)} \leq 2 \, \inf_{\psi\in \mathcal{P}^\Gamma_N} \, \|\partial_\nu u - \psi\|_{L^2(\Gamma^L)}.
$$
Combining Theorem \ref{thm:unified2} with Lemma \ref{lem:dense2}(ii) we obtain the following corollary, guaranteeing convergence of the SS$^*$ method.

\begin{cor} \cite[Corollary 4.1]{ArChDeSa:06} Suppose that the sequence of subspaces $\mathcal{P}^\Gamma_1$, $\mathcal{P}^\Gamma_1$, ..., is chosen so that, for every $n\in\Z$, $\psi_n\in \mathcal{P}^\Gamma_N$, for all sufficiently large $N$, and that $Q_N =\mathcal{P}^\Gamma_N$, for all $N\in \N$. Then
$$
\|\partial_\nu u-\phi_N\|_{L^2(\Gamma^L)}\to 0 \quad \mbox{ as } \; N\to\infty.
$$
\end{cor}

\begin{rem} Numerical experiments in \cite{ArChDeSa:06} compare the performance of the SC, SS, and SS$^*$ methods -- to emphasise these can all be viewed as numerical implementations of the unified transform method for this sound soft diffraction grating problem -- for the case when the surface profile $f$ is sinusoidal, and for various sizes of $kL$ and $ka$, where $a$ is the surface amplitude. Perhaps predictably, given Theorem \ref{thm:unified2}, the SS$^*$ method is most reliable, but, for some geometries and some angles of incidence, all three methods perform very well, producing highly accurate results with around one degree of freedom per wavelength (see \cite[\S6]{ArChDeSa:06} for more details).
\end{rem}

\begin{rem} A potential difficulty with all of the SC, SS, and SS$^*$ methods is that the linear systems that arise can be very ill-conditioned. For the SS$^*$ method rigorous upper and lower bounds for $\mathrm{cond}(A_N)$, the condition number of the system matrix $A_N$, are computed in \cite{ArChDeSa:06}, and the effects of this ill-conditioning on the computed solution are estimated. Regarding the behaviour of $\kappa_N := \mathrm{cond}(A_N)$, the main results are that $\kappa_N$ remains low as long as $\mathcal{P}^\Gamma_N$ contains only plane waves, i.e. $\overline{\gamma v(\cdot,\theta_n)}$ with $\theta_n\in \R$, but necessarily eventually grows exponentially as $N\to\infty$.
\end{rem}

\begin{rem} The system matrix in the SS$^*$ method is the transpose of the matrix to be solved when the same scattering problem is solved by a least squares method: see \cite[\S5]{ArChDeSa:06} for more detail, and cf.~Remark \ref{rem:ls}.
\end{rem}


\begin{thebibliography}{10}

\bibitem{AlVa:05}
{\sc C.~J.~S. Alves and S.~S. Vaitchev}, {\em Numerical comparison of two
  meshfree methods for acoustic wave scattering}, Eng. Anal. Boundary Elements,
  29 (2005), pp.~371--382.

\bibitem{AmMa:98}
{\sc S.~Amini and N.~D. Maines}, {\em Preconditioned {K}rylov subspace methods
  for boundary element solution of the {H}elmholtz equation}, Internat. J.
  Numer. Methods Engrg., 41 (1998), pp.~875--898.

\bibitem{AnBoEcRe:10}
{\sc A.~Anand, Y.~Boubendir, F.~Ecevit, and F.~Reitich}, {\em Analysis of
  multiple scattering iterations for high-frequency scattering problems. {II}:
  The three-dimensional scalar case}, Numerische Mathematik, 114 (2010),
  pp.~373--427.

\bibitem{ArChLa:07}
{\sc S.~Arden, S.~N. Chandler-Wilde, and S.~Langdon}, {\em A collocation method
  for high-frequency scattering by convex polygons}, J. Comp. Appl. Math., 204
  (2007), pp.~334--343.

\bibitem{ArChDeSa:06}
{\sc T.~Arens, S.~N. Chandler-Wilde, and J.~A. DeSanto}, {\em On integral
  equation and least squares methods for scattering by diffraction gratings},
  Commun. Comput. Phys., 1 (2006), pp.~1010--1042.

\bibitem{At:97}
{\sc K.~E. Atkinson}, {\em The Numerical Solution of Integral Equations of the
  Second Kind.}, Cambridge Monographs on Applied and Computational Mathematics, CUP,
  1997.

\bibitem{BaHa:08}
{\sc L.~Banjai and W.~Hackbusch}, {\em Hierarchical matrix techniques for low
  and high frequency {H}elmholtz equation}, IMA J. Numer. Anal., 28 (2008),
  pp.~46--79.

\bibitem{Be:08}
{\sc M.~Bebendorf}, {\em Hierarchical matrices: a means to efficiently solve
  elliptic boundary value problems}, vol.~63 of Lecture Notes in Computational
  Science and Engineering, Springer, Berlin Heidelburg, 2008.

\bibitem{BePhSp:14}
{\sc T.~Betcke, J.~Phillips, and E.~A. Spence}, {\em Spectral decompositions
  and nonnormality of boundary integral operators in acoustic scattering}, IMA
  J. Numer. Anal., 34 (2014), pp.~700--731.

\bibitem{BS}
{\sc T.~Betcke and E.~A. Spence}, {\em Numerical estimation of coercivity
  constants for boundary integral operators in acoustic scattering}, SIAM J.
  Numer. Anal., 49 (2011), pp.~1572--1601.

\bibitem{BrWe:65}
{\sc H.~Brakhage and P.~Werner}, {\em {\"{U}ber das Dirichletsche
  Aussenraumproblem f\"{u}r die {H}elmholtzsche Schwingungsgleichung}}, Archiv
  der Mathematik, 16 (1965), pp.~325--329.

\bibitem{BGMR04}
{\sc O.~P. Bruno, C.~A. Geuzaine, J.~A. Monro~Jr., and F.~Reitich}, {\em
  Prescribed error tolerances within fixed computational times for scattering
  problems of arbitrarily high frequency: {t}he convex case}, Philos. Trans. R.
  Soc. Lond. Ser. A, 362 (2004), pp.~629--645.

\bibitem{BrLi:13}
{\sc O.~P. Bruno and S.~K. Lintner}, {\em A high-order integral solver for
  scalar problems of diffraction by screens and apertures in three-dimensional
  space}, J. Comput. Phys., 252 (2013), pp.~250--274.

\bibitem{BrRe:07}
{\sc O.~P. Bruno and F.~Reitich}, {\em High order methods for high-frequency
  scattering applications}, in Modeling and Computations in Electromagnetics,
  H.~Ammari, ed., vol.~59 of Lect. Notes Comput. Sci. Eng., Springer, 2007,
  pp.~129--164.

\bibitem{BuHi:05b}
{\sc A.~Buffa and R.~Hiptmair}, {\em {A coercive combined field integral
  equation for electromagnetic scattering}}, SIAM Journal on Numerical
  Analysis, 42 (2005), pp.~621--640.

\bibitem{BuSa:06}
{\sc A.~Buffa and S.~Sauter}, {\em {On the acoustic single layer potential:
  stabilization and Fourier analysis}}, SIAM Journal on Scientific Computing,
  28 (2006), pp.~1974--1999.

\bibitem{BuMi:71}
{\sc A.~J. Burton and G.~F. Miller}, {\em The application of integral equation
  methods to the numerical solution of some exterior boundary-value problems},
  Proceedings of the Royal Society of London, Series A, Mathematical and
  Physical Sciences, 323 (1971), pp.~201--210.

\bibitem{HalfPlaneRep}
{\sc S.~N. Chandler-Wilde}, {\em Boundary value problems for the {H}elmholtz
  equation in a half-plane}, in Proceedings of the Third International
  Conference on Mathematical and Numerical Aspects of Wave Propagation,
  G.~Cohen, ed., SIAM, Philadelphia, 1995, pp.~188--197.

\bibitem{ChGrLaSp:12}
{\sc S.~N. Chandler-Wilde, I.~G. Graham, S.~Langdon, and E.~A. Spence}, {\em
  Numerical-asymptotic boundary integral methods in high-frequency acoustic
  scattering}, Acta Numer., 21 (2012), pp.~89--305.

\bibitem{ScreenCoerc}
{\sc S.~N. Chandler-Wilde and D.~P. Hewett}, {\em Acoustic scattering by
  fractal screens: mathematical formulations and wavenumber-explicit continuity
  and coercivity estimates}, University of Reading preprint MPS-2013-17,
  (2013).

\bibitem{NonConvex}
{\sc S.~N. Chandler-Wilde, D.~P. Hewett, S.~Langdon, and A.~Twigger}, {\em A
  high frequency boundary element method for scattering by a class of nonconvex
  obstacles}, Numer. Math. (to appear),  (2014).

\bibitem{Convex}
{\sc S.~N. Chandler-Wilde and S.~Langdon}, {\em A {G}alerkin boundary element
  method for high frequency scattering by convex polygons}, SIAM J. Numer.
  Anal., 45 (2007), pp.~610--640.

\bibitem{CWLM}
{\sc S.~N. Chandler-Wilde, S.~Langdon, and M.~Mokgolele}, {\em A high frequency
  boundary element method for scattering by convex polygons with impedance
  boundary conditions}, Commun. Comput. Phys., 11 (2012), pp.~575--593.

\bibitem{ChLaRi:04}
{\sc S.~N. Chandler-Wilde, S.~Langdon, and L.~Ritter}, {\em A high-wavenumber
  boundary-element method for an acoustic scattering problem}, Philosophical
  Transactions of the Royal Society of London, Series A: Mathematical, Physical
  and Engineering Sciences, 362 (2004), pp.~647--671.

\bibitem{ChCrGiGrEtHuRoYaZh:06}
{\sc H.~Cheng, W.~Y. Crutchfield, Z.~Gimbutas, L.~F. Greengard, J.~F. Ethridge,
  J.~Huang, V.~Rokhlin, N.~Yarvin, and J.~Zhao}, {\em A wideband fast multipole
  method for the {H}elmholtz equation in three dimensions}, J. Comput. Phys.,
  216 (2006), pp.~300--325.

\bibitem{ChNe:00}
{\sc S.~H. Christiansen and J.~C. N\'ed\'elec}, {\em Preconditioners for the
  numerical solution of boundary integral equations from electromagnetism},
  C.R.\ Acad.\ Sci.\ I-Math, 331 (2000), pp.~733--738.

\bibitem{CoKr:92}
{\sc D.~Colton and R.~Kress}, {\em Inverse Acoustic and Electromagnetic
  Scattering Theory}, Springer-Verlag, Berlin, 1992.

\bibitem{CoKr:83}
{\sc D.~L. Colton and R.~Kress}, {\em Integral Equation Methods in Scattering
  Theory}, John Wiley \& Sons Inc., New York, 1983.

\bibitem{Co:04}
{\sc M.~Costabel}, {\em Time-dependent problems with the boundary integral
  equation method}, Encyclopedia of Computational Mechanics,  (2004).

\bibitem{Da:02}
{\sc E.~Darrigrand}, {\em Coupling of fast multipole method and microlocal
  discretization for the 3-D {H}elmholtz equation}, Journal of Computational
  Physics, 181 (2002), pp.~126--154.

\bibitem{DaHa:04}
{\sc E.~Darve and P.~Hav\'{e}}, {\em A fast multipole method for {M}axwell
  equations stable at all frequencies}, Philosophical Transactions:
  Mathematical, Physical and Engineering Sciences, 362 (2004), pp.~603--628.

\bibitem{DeTrVa14}
{\sc B.~Deconinck, T.~Trogdon, and V.~Vasan}, {\em The method of {F}okas for
  solving linear partial differential equations}, SIAM Rev., 56 (2014),
  pp.~159--186.

\bibitem{DeSanto:81}
{\sc J.~A. DeSanto}, {\em Scattering from a perfectly reflecting arbitrary
  periodic surface: an exact theory}, Radio Sci., 16 (1981), pp.~1315--1326.

\bibitem{DeSaErHeKrMiSw:01}
{\sc J.~A. DeSanto, G.~Erdmann, W.~Hereman, B.~Krause, M.~Misra, and E.~Swim},
  {\em Theoretical and computational aspects of scattering from periodic
  surfaces: two-dimensional perfectly reflecting surfaces using the
  spectral-coordinate method}, Waves Random Media, 11 (2001), pp.~455--487.

\bibitem{DeSaErHeMi:98}
{\sc J.~A. DeSanto, G.~Erdmann, W.~Hereman, and M.~Misra}, {\em Theoretical and
  computational aspects of scattering from rough surfaces: One-dimensional
  perfectly reflecting surfaces}, Waves Random Media, 8 (1998), pp.~385--414.

\bibitem{DeSaErHeMi:01}
\leavevmode\vrule height 2pt depth -1.6pt width 23pt, {\em Theoretical and
  computational aspects of scattering from rough surfaces: One-dimensional
  transmission interfaces}, Waves Random Media, 11 (2001), pp.~425--453.

\bibitem{DGS07}
{\sc V.~Dominguez, I.~G. Graham, and V.~P. Smyshlyaev}, {\em A hybrid
  numerical-asymptotic boundary integral method for high-frequency acoustic
  scattering}, Numer. Math., 106 (2007), pp.~471--510.

\bibitem{DoJiCh:03}
{\sc K.~C. Donepudi, J.~M. Jin, and W.~C. Chew}, {\em A grid-robust
  higher-order multilevel fast multipole algorithm for analysis of 3-d
  scatterers}, Electromagnetics, 23 (2003), pp.~315--330.

\bibitem{ER09}
{\sc F.~Ecevit and F.~Reitich}, {\em Analysis of multiple scattering iterations
  for high-frequency scattering problems. {I}: The two-dimensional case},
  Numer. Math., 114 (2009), pp.~271--354.

\bibitem{ElYa:02}
{\sc J.~Elschner and M.~Yamamoto}, {\em An inverse problem in periodic
  diffractive optics: reconstruction of {L}ipschitz grating profiles}, Appl.
  Anal., 81 (2001), pp.~1307--1328.

\bibitem{EnSt:07}
{\sc S.~Engleder and O.~Steinbach}, {\em {Modified boundary integral
  formulations for the {H}elmholtz equation}}, Journal of Mathematical Analysis
  and Applications, 331 (2007), pp.~396--407.

\bibitem{EnSt:08}
\leavevmode\vrule height 2pt depth -1.6pt width 23pt, {\em {Stabilized boundary
  element methods for exterior {H}elmholtz problems}}, Numerische Mathematik,
  110 (2008), pp.~145--160.

\bibitem{ErGa:12}
{\sc O.~G. Ernst and M.~J. Gander}, {\em Why it is difficult to solve
  {H}elmholtz problems with classical iterative methods}, in Numerical Analysis
  of Multiscale Problems, I.~G. Graham, T.~Y. Hou, O.~Lakkis, and R.~Scheichl,
  eds., vol.~83 of Lecture Notes in Computational Science and Engineering,
  Springer, 2012, pp.~325--363.

\bibitem{Fo97}
{\sc A.~S. Fokas}, {\em A unified transform method for solving linear and
  certain nonlinear PDEs}, Proc. R. Soc. Lond. A, 453 (1997), pp.~1411--1443.

\bibitem{FoLe:14}
{\sc A.~S. Fokas and J.~Lenells}, {\em The unified transform for the modified
  Helmholtz equation in the exterior of a square}, to appear in Unified Transform for Boundary Value Problems: Applications and Advances,
  A.~S. Fokas and B.~Pelloni, eds., SIAM, 2014.

\bibitem{FoSp12}
{\sc A.~S. Fokas and E.~A. Spence}, {\em Synthesis, as opposed to separation,
  of variables}, SIAM Rev., 54 (2012), pp.~291--324.

\bibitem{GH11}
{\sc M.~Ganesh and S.~C. Hawkins}, {\em A fully discrete {G}alerkin method for
  high frequency exterior acoustic scattering in three dimensions}, J. Comput.
  Phys., 230 (2011), pp.~104--125.

\bibitem{GBR05}
{\sc C.~Geuzaine, O.~Bruno, and F.~Reitich}, {\em On the {O(1)} solution of
  multiple-scattering problems}, IEEE Trans. Magn., 41 (2005), pp.~1488--1491.

\bibitem{GrLoMeSp14}
{\sc I.~G. Graham, M.~L\"{o}hndorf, J.~M. Melenk, and E.~A. Spence}, {\em When
  is the error in the $h$-{BEM} for solving the {H}elmholtz equation bounded
  independently of $k$?}, BIT Num. Math. (to appear),  (2014).

\bibitem{GrHeLa:13}
{\sc S.~P. Groth, D.~P. Hewett, and S.~Langdon}, {\em Hybrid
  numerical-asymptotic approximation for high frequency scattering by
  penetrable convex polygons}, IMA J. Appl. Math. (to appear),  (2014).

\bibitem{HaHeLaLa:14}
{\sc J.~A. Hargeaves, D.~P. Hewett, Y.~W. Lam, and S.~Langdon}, {\em A high
  frequency boundary element method for scattering by three-dimensional
  screens}.
\newblock {I}n preparation.

\bibitem{HaCh:03}
{\sc P.~J. Harris and K.~Chen}, {\em On efficient preconditioners for iterative
  solution of a {G}alerkin boundary element equation for the three-dimensional
  exterior {H}elmholtz problem}, J. Comput. Appl. M., 156 (2003), pp.~303--318.

\bibitem{HeLaCh:13}
{\sc D.~P. Hewett, S.~Langdon, and S.~N. Chandler-Wilde}, {\em A
  frequency-independent boundary element method for scattering by
  two-dimensional screens and apertures}, IMA J. Numer. Anal. (to appear),
  University of Reading preprint MPS-2013-18,  (2013)

\bibitem{HeLaMe:11}
{\sc D.~P. Hewett, S.~Langdon, and J.~M. Melenk}, {\em A high frequency $hp$
  boundary element method for scattering by convex polygons}, SIAM J. Numer.
  Anal., 51 (2013), pp.~629--653.

\bibitem{HiJeUr:13}
{\sc R.~Hiptmair, C.~Jerez-Hanckes, and C.~Urzua}, {\em Optimal operator
  preconditioning for boundary elements on open curves}, Tech. Rep. 2013-48,
  Seminar for Applied Mathematics, ETH Z{\"u}rich, 2013.
\newblock Submitted to SIAM Journal on Numerical Analysis.

\bibitem{HiKi:12}
{\sc R.~Hiptmair and L.~Kielhorn}, {\em BETL - a generic boundary element
  template library}, Tech. Rep. 2012-36, Seminar for Applied Mathematics, ETH
  Z{\"u}rich, 2012.

\bibitem{HsWe:04}
{\sc G.~C. Hsaio and W.~L. Wendland}, {\em Boundary element methods: Foundation
  and error analysis}, in Fundamentals, E.~Stein, R.~de~Borst, and T.~J.~R.
  Hughes, eds., vol.~1 of Encylopedia of Computational Mechanics, John Wiley
  and Sons, Ltd., 2004, pp.~339--373.

\bibitem{HuVa:07a}
{\sc D.~Huybrechs and S.~Vandewalle}, {\em A sparse discretization for integral
  equation formulations of high frequency scattering problems}, SIAM J. Sci.
  Comput., 29 (2007), pp.~2305--2328.

\bibitem{Ih:98}
{\sc F.~Ihlenburg}, {\em Finite element analysis of acoustic scattering},
  vol.~132, Springer Verlag, 1998.

\bibitem{Kress}
{\sc R.~Kress}, {\em Linear Integral Equations}, Springer-Verlag, New York, 2nd
  ed., 1999.

\bibitem{LaCh:06}
{\sc S.~Langdon and S.~N. Chandler-Wilde}, {\em A wavenumber independent
  boundary element method for an acoustic scattering problem}, SIAM J.\ Numer.\
  Anal., 43 (2006), pp.~2450--2477.

\bibitem{LaMoCh:10}
{\sc S.~Langdon, M.~Mokgolele, and S.~N. Chandler-Wilde}, {\em High frequency
  scattering by convex curvilinear polygons}, J. Comput. Appl. Math., 234
  (2010), pp.~2020--2026.

\bibitem{Le:65}
{\sc R.~Leis}, {\em Zur dirichletschen randwertaufgabe des aussenraumes der
  schwingungsgleichung}, Mathematische Zeitschrift, 90 (1965), pp.~205--211.

\bibitem{LoMe:10}
{\sc M.~L\"{o}hndorf and J.~M. Melenk}, {\em Wavenumber-explicit hp-{BEM} for
  high frequency scattering}, SIAM J. Numer. Anal., 49 (2011), pp.~2340--2363.


\bibitem{McLean}
{\sc W.~McLean}, {\em Strongly Elliptic Systems and Boundary Integral
  Equations}, CUP, 2000.

\bibitem{MeMeUrRa:10}
{\sc M.~Messner, M.~Messner, P.~Urthaler, and F.~Rammerstorfer}, {\em
  Hyperbolic and elliptic numerical analysis ({HyENA})}, 2010.

\bibitem{MoSp14}
{\sc A.~Moiola and E.~A. Spence}, {\em Is the {H}elmholtz equation really
  sign-indefinite?} SIAM Rev., 56 (2014), pp.~274--312.

\bibitem{Ne:01}
{\sc J.-C. N\'ed\'elec}, {\em Acoustic and electromagnetic equations: integral
  representations for harmonic problems}, Springer-Verlag New York, Inc., 2001.

\bibitem{OfStWe:06}
{\sc G.~Of, O.~Steinbach, and W.~L. Wendland}, {\em The fast multipole method
  for the symmetric boundary integral formulation}, IMA J. Numer. Anal., 26
  (2006), pp.~272--296.

\bibitem{Pa:65}
{\sc O.~I. Panich}, {\em On the question of the solvability of exterior
  boundary-value problems for the wave equation and for a system of {M}axwell's
  equations (in {R}ussian)}, Uspekhi Mat. Nauk, 20:1(121) (1965), pp.~221--226.

\bibitem{Rudin:91}
{\sc W.~Rudin}, {\em Functional Analysis, 2nd Ed.}, McGraw-Hill, New York,
  1991.

\bibitem{sauter-schwab11}
{\sc S.~A. Sauter and C.~Schwab}, {\em Boundary Element Methods}, vol.~39 of
  Springer Series in Computational Mathematics, Springer-Verlag, Berlin, 2011.
\newblock Translated and expanded from the 2004 German original.

\bibitem{Sl:92}
{\sc I.~H. Sloan}, {\em Error analysis of boundary integral methods}, Acta
  Numer., 1 (1992), pp.~287--339.

\bibitem{Sl:95}
\leavevmode\vrule height 2pt depth -1.6pt width 23pt, {\em Boundary element
  methods}, in Theory and Numerics of ODEs and PDEs, M.~Ainsworth, J.~Levesley,
  W.~A. Light, and M.~Marletta, eds., vol.~IV of Advances in Numerical
  Analysis, Oxford Science Publications, 1995, pp.~143--180.

\bibitem{SmArBePhSc:13}
{\sc W.~\'{S}migaj, S.~Arridge, T.~Betcke, J.~Phillips, and M.~Schweiger}, {\em
  Solving boundary integral problems with {BEM++}}, ACM Trans. Math. Software
  (to appear),  (2014).

\bibitem{Spence14}
{\sc E.~A. Spence}, {\em ``{W}hen all else fails, integrate by parts'' - an
  overview of new and old variational formulations for linear elliptic PDEs},
  to appear in Unified Transform for Boundary Value Problems: Applications and Advances,
  A.~S. Fokas and B.~Pelloni, eds., SIAM, 2014.

\bibitem{SpChGrSm:11}
{\sc E.~A. Spence, S.~N. Chandler-Wilde, I.~G. Graham, and V.~P. Smyshlyaev},
  {\em A new frequency-uniform coercive boundary integral equation for acoustic
  scattering}, Communications on Pure and Applied Mathematics, 64 (2011),
  pp.~1384--1415.

\bibitem{SpFo:10}
{\sc E.~A. Spence and A.~S. Fokas}, {\em A new transform method I:
  domain-dependent fundamental solutions and integral representations}, Proc.
  R. Soc. A, 466 (2010), pp.~2259--2281.

\bibitem{SpKaSm:14}
{\sc E.~A. Spence, I.~V. Kamotski, and V.~P. Smyshlyaev}, {\em Coercivity of
  combined boundary integral equations in high-frequency scattering}, Comm.
  Pure Appl. Math. (to appear),  (2014).

\bibitem{St:08}
{\sc O.~Steinbach}, {\em Numerical Approximation Methods for Elliptic Boundary
  Value Problems: Finite and Boundary Elements}, Springer, New York, 2008.

\bibitem{stephan87}
{\sc E.~P. Stephan}, {\em Boundary integral equations for screen problems in
  {$\mathbb{R}^3$}}, Integral Equat. Oper. Th., 10 (1987), pp.~236--257.

\bibitem{ToWe93}
{\sc R.~H. Torres and G.~V. Welland}, {\em The {H}elmholtz equation and
  transmission problems with {L}ipschitz interfaces}, Indiana Univ. Math. J.,
	42 (1993), pp 1457--1486.
\end{thebibliography}

\end{document}